\documentclass[12pt]{article}

\usepackage{hyperref}
\hypersetup{
  colorlinks   = true,    
  urlcolor     = blue,    
  linkcolor    = blue,    
  citecolor    = red      
}
\usepackage{epstopdf}
\usepackage[caption=false]{subfig}

\usepackage{amsmath}
\usepackage{csquotes}
\usepackage{tikz}
\usepackage{upgreek}
\usepackage{amssymb}
\usepackage{amsthm}
\usepackage{amsfonts}
\numberwithin{equation}{section}
\usepackage{booktabs,caption}
\usepackage[flushleft]{threeparttable}
\usepackage{stmaryrd}
\usepackage{wasysym}
\usepackage{makecell}
\usepackage{float}
\usepackage{graphicx}
\usepackage{xcolor}
\usepackage{leftindex}
\usepackage{longtable}
\allowdisplaybreaks

\newtheorem{theorem}{Theorem}[section]
\newtheorem{definition}[theorem]{Definition}

\newtheorem{lemma}[theorem]{Lemma}

\newtheorem{corollary}[theorem]{Corollary}
\newtheorem{notation}[theorem]{Notation}
\newtheorem{remark}[theorem]{Remark}

\theoremstyle{definition}
\newtheorem{example}[theorem]{Example}

\usepackage{nccmath}
\usepackage{mathtools}
\usepackage[reftex]{theoremref}
\usepackage{tikz-cd} 
\usepackage[left=1in, right=1in, top=1in, bottom=1in]{geometry}

\begin{document}


\title{$(\sigma, \tau)$-Derivations of Group Rings with Applications}

\author{Praveen Manju and Rajendra Kumar Sharma}
\date{}
\maketitle

\begin{center}
\noindent{\small Department of Mathematics, \\Indian Institute of Technology Delhi, \\ Hauz Khas, New Delhi-110016, India$^{1}$}
\end{center}

\footnotetext[1]{{\em E-mail addresses:} \url{praveenmanjuiitd@gmail.com}(Praveen Manju), \url{rksharmaiitd@gmail.com}(Rajendra Kumar Sharma).}

\medskip

\begin{abstract}
Leo Creedon and Kieran Hughes in \cite{Creedon2019} studied derivations of a group ring $RG$ (of a group $G$ over a commutative unital ring $R$) in terms of generators and relators of group $G$. In this article, we do that for $(\sigma, \tau)$-derivations. We develop a necessary and sufficient condition such that a map $f:X \rightarrow RG$ can be extended uniquely to a $(\sigma, \tau)$-derivation $D$ of $RG$, where $R$ is a commutative ring with unity, $G$ is a group having a presentation $\langle X \mid Y \rangle$ ($X$ the set of generators and $Y$ the set of relators) and $(\sigma, \tau)$ is a pair of $R$-algebra endomorphisms of $RG$ which are $R$-linear extensions of the group endomorphisms of $G$. Further, we classify all inner $(\sigma, \tau)$-derivations of the group algebra $RG$ of an arbitrary group $G$ over an arbitrary commutative unital ring $R$ in terms of the rank and a basis of the corresponding $R$-module consisting of all inner $(\sigma, \tau)$-derivations of $RG$. We obtain several corollaries, particularly when $G$ is a $(\sigma, \tau)$-FC group or a finite group $G$ and when $R$ is a field. We also prove that if $R$ is a unital ring and $G$ is a group whose order is invertible in $R$, then every $(\sigma, \tau)$-derivation of $RG$ is inner. We apply the results obtained above to study $\sigma$-derivations of commutative group algebras over a field of positive characteristic and to classify all inner and outer $\sigma$-derivations of dihedral group algebras $\mathbb{F}D_{2n}$ ($D_{2n} = \langle a, b \mid a^{n} = b^{2} = 1, b^{-1}ab = a^{-1}\rangle$, $n \geq 3$) over an arbitrary field $\mathbb{F}$ of any characteristic. Finally, we give the applications of these twisted derivations in coding theory by giving a formal construction with examples of a new code called IDD code.
\end{abstract}



\section{Introduction}\label{section 1}
The notion of derivations, introduced from analytic theory, is old and plays a significant role in the research of structure and property in algebraic systems. Derivations play an essential role in mathematics and physics. The theory of derivations has been developed in rings and numerous algebras and helps to study and understand their structure. For example, BCI-algebras \cite{Muhiuddin2012}, MV-algebras \cite{KamaliArdakani2013, Mustafa2013}, Banach algebras \cite{Raza2016}, von Neumann algebras \cite{Brear1992}, incline algebras that have many applications \cite{Kim2014}, lattices that are very important in fields such as information theory: information recovery, information access management, and cryptanalysis \cite{Chaudhry2011}. Differentiable manifolds, operator algebras, $\mathbb{C}^{*}$-algebras, and representation theory of Lie groups are being studied using derivations \cite{Klimek2021}. For a historical account and further applications of derivations, we refer the reader to \cite{Atteya2019, Haetinger2011, MohammadAshraf2006}.

In this article, we consider the pure algebraic structure, namely, the group ring and its $(\sigma, \tau)$-derivations which have numerous applications (see \cite{AleksandrAlekseev2020} and the references within). For a history of group rings, we refer the reader to {\cite[Chapter 3]{CPM2002}}, and for a history of derivations, we refer the reader to the survey articles  \cite{Haetinger2011, MohammadAshraf2006}. The idea of an $(s_{1}, s_{2})$-derivation was introduced by Jacobson \cite{Jacobson1964}. These derivations were later on commonly called as $(\sigma, \tau)$- or $(\theta, \phi)$-derivations. These have been highly studied in prime and semiprime rings and have been principally used in solving functional equations \cite{Brear1992}. Twisted derivations have numerous applications. Derivations, especially $(\sigma, \tau)$-derivations of group rings have various applications in coding theory \cite{Creedon2019, Boucher2014}. The explicit description of derivations is useful in constructing codes (see \cite{Creedon2019}). They are used in multiplicative deformations and discretizations of derivatives that have many applications in models of quantum phenomena and the analysis of complex systems and processes. They are extensively investigated in physics and engineering. Using twisted derivations, Lie algebras are generalized to hom-Lie algebras, and the central extension theory is developed for hom-Lie algebras analogously to that for Lie algebras. Just as Lie algebras were initially studied as algebras of derivations, Hom-Lie algebras were defined as algebras of twisted derivations. The generalizations (deformations and analogs) of the Witt algebra, the complex Lie algebra of derivations on the algebra of Laurent polynomials $\mathbb{C}[t, t^{-1}]$ in one variable, are obtained using twisted derivations. Deformed Witt and Virasoro-type algebras have applications in analysis, numerical mathematics, algebraic geometry, arithmetic geometry, number theory, and physics. We refer the reader to \cite{Hartwig2006, ilwale2023noncommutative, Larsson2017, Siebert1996} for details. Twisted derivations have been used to generalize Galois theory over division rings and in the study of $q$-difference operators in number theory. For more applications of $(\sigma, \tau)$-derivations, we refer the reader to \cite{AleksandrAlekseev2020} and the references within.

Derivations and $(\sigma, \tau)$-derivations of group rings (defined purely algebraically) have not received much attention. The study of derivations of group rings begins with the paper \cite{Smith1978}, where the author studies derivations of group rings of a finitely-generated, torsion-free, nilpotent group $G$ over a field $\mathbb{F}$. For instance, it is shown that such group rings always contain an outer derivation. She, using the notion of derivations of $\mathbb{F}G$, proved that the Krull dimension of such a group ring $\mathbb{F}G$ is finite and equals $\mathbb{F}$-rank of $G$ if $G$ satisfies certain additional conditions. In {\cite[Theorem 1]{Spiegel1994}}, the main theorem of \cite{Spiegel1994}, it is proved that every derivation of an integral group ring of a finite group is inner. In {\cite[Theorem 1.1]{MiguelFerrero1995}}, the main theorem of \cite{MiguelFerrero1995}, the authors demonstrate that if $G$ is a torsion group whose center $Z(G)$ has a finite index in $G$ and $R$ is a semiprime ring such that $\text{char}(R)$ is either $0$ or does not divide the order of any element of $G$, then every $R$-derivation of the group ring $RG$ is inner. In \cite{Chaudhuri2019} and \cite{Chaudhuri2021}, the author generalizes the above results of \cite{Spiegel1994} and \cite{MiguelFerrero1995} respectively to $(\sigma, \tau)$-derivations of group rings of finite groups over a field and an integral domain where $\sigma, \tau$ satisfy certain conditions. In \cite{Chaudhuri2019}, the author proves the following main theorem:
\begin{theorem}[{\cite[Theorem 1.1]{Chaudhuri2019}}]
Let $G$ be a finite group and $R$ be an integral domain with unity such that $|G|$ is invertible in $R$. Let $\sigma, \tau$ be $R$-algebra endomorphisms of $RG$ such that they fix $Z(RG)$ elementwise.
\begin{enumerate}
\item[(i)] If $R$ is a field, then every $(\sigma, \tau)$-derivation of $RG$ is inner.
\item[(ii)] If $R$ is an integral domain which is not a field and if $\sigma, \tau$ are $R$-linear extensions of the group homomorphisms of $G$, then every $(\sigma, \tau)$-derivation of $RG$ is inner.
\end{enumerate}
\end{theorem} 
\noindent She also gives the following as a corollary of the above theorem. If $G$ is a finite group, $R = \mathbb{Z}$ and $\sigma, \tau$ satisfy the conditions of part (ii) of the above theorem, then every $(\sigma, \tau)$-derivation of $\mathbb{Z}G$ is inner. In {\cite[Theorem 1.1]{Chaudhuri2021}}, the main theorem of \cite{Chaudhuri2021}, the author proved that if $R$ is a semiprime ring with unity such that either $R$ does not have torsion elements or that if $R$ has $p$-torsion elements, then $p$ does not divide $|G|$, where $G$ is a torsion group such that $[G:Z(G)] < \infty$ and if $\sigma, \tau$ are $R$-algebra endomorphisms of $RG$ which fix $Z(RG)$ elementwise, then there is a ring $T$ such that $R \subseteq T$, $Z(R) \subseteq Z(T)$ and for the natural extensions of $\sigma, \tau$ to $TG$, $H^{1}(TG, \leftindex_{\sigma} {TG}_{\tau}) = \{0\}$, the degree $1$ Hochschild cohomology, that is, there exists a ring extension $T$ of the semiprime ring $R$ such that all $(\sigma, \tau)$-derivations of $TG$ are inner for the natural extensions $\sigma, \tau$ to $TG$. A. A. Arutyunov, in his several papers (\cite{AleksandrAlekseev2020, Arutyunov2021, A.A.Arutyunov2020, Arutyunov2023, Arutyunov2020a, Arutyunov2020, arutyunov2019smooth}), studies ordinary and twisted derivations using topology and characters. For details, we refer the reader to the end of Subsections \ref{subsection 2.3} and \ref{subsection 3.2}, where we have highlighted the work done in these papers. In \cite{OrestD.Artemovych2020}, the authors considered the group ring of a group $G$ over a unital ring $R$ such that all prime divisors of orders of elements in $G$ are invertible in $R$. They proved that if $R$ is finite and $G$ is a torsion FC-group, all derivations of $RG$ are inner. They obtained similar results for other classes of groups $G$ and rings $R$. In \cite{Creedon2019}, the authors studied the derivations of group rings over a commutative unital ring $R$ in terms of the generators and relators of the group. In {\cite[Theorem 2.5]{Creedon2019}}, which is the main theorem of their article, the authors gave a necessary and sufficient condition under which a map from the generating set of the group $G$ to the group ring $RG$ can be extended to a derivation of $RG$. Applying this characterization, the authors in {\cite[Theorem 3.4]{Creedon2019}} classified the ordinary derivations of commutative group algebras over a field of prime characteristic $p$ by giving the dimension and a basis for the vector space of all derivations over the underlying field of prime characteristic. In {\cite[Theorem 3.11 and Theorem 3.13]{Creedon2019}}, the authors also gave an explicit description of the ordinary derivations and ordinary inner derivations of dihedral group algebras over a field of characteristic $2$. In our article, we study the problems or questions for more generalized $(\sigma, \tau)$-derivations of group rings. This paper will work with $(\sigma, \tau)$-derivations and inner $(\sigma, \tau)$-derivations of group algebras over arbitrary characteristics. In this article, we also consider another problem, namely, the twisted derivation problem.  The derivation problem for group rings asks if all the derivations in a group ring are inner or if the space of outer derivations is trivial. We refer the reader to \cite{AleksandrAlekseev2020, A.A.Arutyunov2020, Arutyunov2020b, Arutyunov2020} for the history and importance of the derivation problem. In this article, we consider the analogous problem for $(\sigma, \tau)$-derivations.

We have divided the manuscript into seven sections. Section \ref{section 2(NN)} states some basic definitions and facts. In Section \ref{section 2}, we study $(\sigma, \tau)$-derivations of a group ring $RG$ of a group $G$ over a commutative unital ring $R$ in terms of the generators and relators of group $G$. The section has been further subdivided into four subsections. In Subsection \ref{subsection 2.1}, we state and prove useful results that we will need later in the proof of the section's main Theorems \ref{theorem 2.4(N)} and \ref{theorem 2.4}. In Subsection \ref{subsection 2.2(N)}, we prove the main Theorem \ref{theorem 2.4(N)} for $\sigma$-derivations with the help of a crucial Lemma \ref{lemma 2.3(N)}. In Subsection \ref{subsection 2.2}, we generalize the results of Subsection \ref{subsection 2.2(N)} to $(\sigma, \tau)$-derivations. We prove the main Theorem \ref{theorem 2.4} with the help of an important Lemma \ref{lemma 2.3}. The basic idea for proofs of Theorems \ref{theorem 2.4(N)} and \ref{theorem 2.4} is from \cite{Creedon2019}. However, the use of the universal property of free groups in our proofs of the main theorems of Section \ref{section 2} makes the proofs different from those of \cite{Creedon2019}. We will see that the universal property of free groups plays a crucial role in the proof of the main theorems (Theorems \ref{theorem 2.4(N)} and \ref{theorem 2.4}) of this section. In Subsection \ref{subsection 2.3}, we apply the results obtained in studying $\sigma$-derivations of commutative group algebras.

In Section \ref{section 3}, we study the inner $(\sigma, \tau)$-derivations of group algebras. We do that using the notion of doubly-twisted conjugacy classes, namely, $(\sigma, \tau)$-conjugacy classes, which extend the concept of usual conjugacy classes in classic group theory when $\sigma = id = \tau$. The notion of twisted conjugacy classes in group theory was introduced in the twentieth century. In Subsection \ref{subsection 3.1}, we observe some results and terminologies based on doubly-twisted conjugacy classes analogous to that in classical group theory. In Subsection \ref{subsection 3.2}, we obtain the main results of the section. In the section's main theorem, namely, Theorem \ref{theorem 3.13}, we classify all inner $(\sigma, \tau)$-derivations of the group ring $RG$ of an arbitrary group $G$ over an arbitrary commutative unital ring $R$ by finding the rank of the $R$-module  $\text{Inn}_{(\sigma, \tau)}(\mathbb{F}G)$ of all inner $(\sigma, \tau)$-derivations and a basis of it. Before this, no result seems to be available in the literature that classifies all inner $(\sigma, \tau)$-derivations of the group algebra $RG$ of an arbitrary group $G$ over an arbitrary commutative unital ring $R$. We obtain several corollaries of Theorem \ref{theorem 3.13}. Corollary \ref{corollary 3.13(1)} is stated for the inner $(\sigma, \tau)$-derivations of $RG$ when $G$ is a $(\sigma, \tau)$-FC group. Corollary \ref{corollary 3.13(2)} is stated for the inner $(\sigma, \tau)$-derivations of $RG$ when $G$ is a finite group. Corollary \ref{corollary 3.13(3)} is stated for the inner $(\sigma, \tau)$-derivations of $RG$, when $R$ is a field and $G$ is a finite group. In Theorem \ref{theorem 3.14}, we prove that every $(\sigma, \tau)$-derivation of $RG$ is an inner $(\sigma, \tau)$-derivation if $R$ is a unital ring and $G$ is a finite group whose order is invertible in $R$. Corollary \ref{corollary 3.15} is one of the consequences of Corollary \ref{corollary 3.13(2)} and Theorem \ref{theorem 3.14}.

In Section \ref{section 4}, we study the application of the results obtained in Sections \ref{section 2} and \ref{section 3} in classifying all inner and outer $\sigma$-derivations of the dihedral group algebra $\mathbb{F}D_{2n}$ ($n \geq 3$) over an arbitrary field $\mathbb{F}$ of any characteristic. Describing the derivation algebra consisting of the derivations of group algebra is a well-known problem. This section describes the $\sigma$-derivation algebras of dihedral group algebras over an arbitrary field. For this, we subdivide the section into four subsections. In Subsection \ref{subsection 4.1}, we consider the group algebra $\mathbb{F}D_{2n}$ of $D_{2n}$ over a field $\mathbb{F}$ of characteristic $0$ or an odd rational prime $p$. We classify all $\sigma$-derivations of $\mathbb{F}D_{2n}$ explicitly by providing the dimension and an $\mathbb{F}$-basis of the $\sigma$-derivation algebra $\mathcal{D}_{\sigma}(\mathbb{F}D_{2n})$. In Subsection \ref{subsection 4.2}, we do the same for the group algebra $\mathbb{F}D_{2n}$ of $D_{2n}$ over a field $\mathbb{F}$ of characteristic $2$. In Subsection \ref{subsection 4.1}, we determine all $\sigma$-conjugacy classes of $D_{2n}$. In Subsection \ref{subsection 4.4}, we classify all inner and outer $\sigma$-derivations of $\mathbb{F}D_{2n}$ over an arbitrary field $\mathbb{F}$ thus solving the $\sigma$-twisted derivation problem for dihedral group algebras over an arbitrary field. This section thus illustrates how important the results of Sections \ref{section 2} and \ref{section 3} are in the theory of twisted derivations of group rings.

In Section \ref{section 6}, we discuss the applications of our results in coding theory. We give the notion of a new code in coding theory, namely, an Image of Derivation-Derived (IDD) code. We give the definition and construction and illustrate via some examples. Finally, in Section \ref{section 7}, we conclude our findings.

\section{Basic Definitions}\label{section 2(NN)}
Below, we state some common knowledge. Let $R$ be a commutative unital ring and $\mathcal{A}$ be an associative $R$-algebra. Let $(\sigma, \tau)$ be a pair of $R$-algebra endomorphisms of $\mathcal{A}$. 

\begin{definition}\label{definition 2.1(NN)}
An $R$-linear map $d:\mathcal{A} \rightarrow \mathcal{A}$ that satisfies $d(\alpha \beta) = d(\alpha) \beta + \alpha d(\beta)$ for all $\alpha, \beta \in \mathcal{A}$, is called a derivation of $\mathcal{A}$. It is called inner if there exists some $\beta \in \mathcal{A}$ such that $d(\alpha) = \beta \alpha - \alpha \beta$ for all $\alpha \in \mathcal{A}$, and then we denote it by $d_{\beta}$. The elements of the quotient of the $R$-module of all derivations of $\mathcal{A}$ by the $R$-submodule of all inner derivations are called outer derivations.
\end{definition}

\begin{definition}\label{definition 2.2(NN)}
A $(\sigma, \tau)$-derivation $D:\mathcal{A} \rightarrow \mathcal{A}$ is an $R$-linear map that satisfies the $(\sigma, \tau)$-twisted generalized identity: $D(\alpha \beta) = D(\alpha)\tau(\beta) + \sigma(\alpha) D(\beta)$ for all $\alpha, \beta \in \mathcal{A}$. It is called inner if there exists some $\beta \in \mathcal{A}$ such that $D(\alpha) = \beta \tau(\alpha) - \sigma(\alpha) \beta$ for all $\alpha \in \mathcal{A}$, and then we denote it by $D_{\beta}$. The elements of the quotient of the $R$-module of all $(\sigma, \tau)$-derivations of $\mathcal{A}$ by the $R$-submodule of all inner $(\sigma, \tau)$-derivations are called outer $(\sigma, \tau)$-derivations.
\end{definition}

\begin{definition}\label{definition 2.3(NN)}
If $\tau = \sigma$, then the generalized identity becomes $D(\alpha \beta) = D(\alpha) \sigma(\beta) + \sigma(\alpha) D(\beta)$; and for the sake of convenience, we call this $(\sigma, \sigma)$-derivation as a $\sigma$-derivation. A $\sigma$-derivation $D:\mathcal{A} \rightarrow \mathcal{A}$ is called an inner $\sigma$-derivation if there exists some $\beta \in \mathcal{A}$ such that $D(\alpha) = \beta \sigma(\alpha) - \sigma(\alpha) \beta$ for all $\alpha \in \mathcal{A}$. The elements of the quotient of the $R$-module of all $\sigma$-derivations of $\mathcal{A}$ by the $R$-submodule of all inner $\sigma$-derivations are called outer $\sigma$-derivations.
\end{definition}

\begin{notation}\label{notation 2.4(NN)}
We denote the set of all $(\sigma, \tau)$-derivations on $\mathcal{A}$ by $\mathcal{D}_{(\sigma, \tau)}(\mathcal{A})$, the set of all $\sigma$-derivations on $\mathcal{A}$ by $\mathcal{D}_{\sigma}(\mathcal{A})$ and the set of all ordinary derivations on $\mathcal{A}$ by $\mathcal{D}(\mathcal{A})$. We denote the corresponding set of inner $(\sigma, \tau)$-derivations, inner $\sigma$-derivations and ordinary inner derivations on $\mathcal{A}$ by $\text{Inn}_{(\sigma, \tau)}(\mathcal{A})$, $\text{Inn}_{\sigma}(\mathcal{A})$ and $\text{Inn}(\mathcal{A})$ respectively. We denote the corresponding set of outer $(\sigma, \tau)$-derivations, outer $\sigma$-derivations and ordinary outer derivations on $\mathcal{A}$ by $\text{Out}_{(\sigma, \tau)}(\mathcal{A})$, $\text{Out}_{\sigma}(\mathcal{A})$ and $\text{Out}(\mathcal{A})$ respectively.
\end{notation}

\begin{remark}\label{remark 2.5(NN)}
If $\sigma = \tau = id_{\mathcal{A}}$ (the identity map on $\mathcal{A}$), then the usual Leibniz identity holds, and $(\sigma, \tau)$-derivation, inner $(\sigma, \tau)$-derivation, and outer $(\sigma, \tau)$-derivation respectively become the ordinary derivation, ordinary inner derivation and ordinary outer derivation on $\mathcal{A}$. Defining componentwise sum and module action, $\mathcal{D}_{(\sigma, \tau)}(\mathcal{A})$ becomes an $R$- as well as $\mathcal{A}$-module, and  $\text{Inn}_{(\sigma, \tau)}(\mathcal{A})$ become its submodule. If $1$ is the unity in $\mathcal{A}$ and $D$ is a $(\sigma, \tau)$-derivation of $\mathcal{A}$, then $D(1) = 0$.
\end{remark}

In view of the the above Definition \ref{definition 2.2(NN)} and the Remark \ref{remark 2.5(NN)}, the outer $(\sigma, \tau)$-derivations are precisely the elements of the factor module $\text{Out}_{(\sigma, \tau)}(\mathcal{A}) = \frac{\mathcal{D}_{(\sigma, \tau)}(\mathcal{A})}{\text{Inn}_{(\sigma, \tau)}(\mathcal{A})}$. Also, note that the set $\mathcal{D}_{(\sigma, \tau)}(\mathcal{A}) \setminus \text{Inn}_{(\sigma, \tau)}(\mathcal{A})$ is the set of all non-inner $(\sigma, \tau)$-derivations of $\mathcal{A}$. The following lemma establishes a connection between our notions of outer $(\sigma, \tau)$-derivations and non-inner $(\sigma, \tau)$-derivations of $\mathcal{A}$.

\begin{lemma}\label{lemma 2.6(NN)}
Let $T = \{D_{i} \in \mathcal{D}_{(\sigma, \tau)}(\mathcal{A}) \mid i \in I\}$ ($I$ some indexing set) be a left transversal of $\text{Inn}_{(\sigma, \tau)}(\mathcal{A})$ in $\mathcal{D}_{(\sigma, \tau)}(\mathcal{A})$ with $0$ as the coset representative of the coset $\text{Inn}_{(\sigma, \tau)}(\mathcal{A})$. Then the non-inner $(\sigma, \tau)$-derivations of $\mathcal{A}$ correspond to the elements in the set $\bigcup_{D_{i} \in T \setminus \{0\}} (D_{i} + \text{Inn}_{(\sigma, \tau)}(\mathcal{A}))$. More precisely, $\mathcal{D}_{(\sigma, \tau)}(\mathcal{A}) \setminus \text{Inn}_{(\sigma, \tau)}(\mathcal{A}) = \bigcup_{D_{i} \in T \setminus \{0\}} (D_{i} + \text{Inn}_{(\sigma, \tau)}(\mathcal{A}))$.
\end{lemma}
\begin{proof}
Let $D \in \mathcal{D}_{(\sigma, \tau)}(\mathcal{A}) \setminus \text{Inn}_{(\sigma, \tau)}(\mathcal{A})$. Then $D \in \mathcal{D}_{(\sigma, \tau)}(\mathcal{A})$ but $D \notin \text{Inn}_{(\sigma, \tau)}(\mathcal{A})$. This implies that $D + \text{Inn}_{(\sigma, \tau)}(\mathcal{A}) \neq \text{Inn}_{(\sigma, \tau)}(\mathcal{A})$ so that $D + \text{Inn}_{(\sigma, \tau)}(\mathcal{A}) = D_{i} + \text{Inn}_{(\sigma, \tau)}(\mathcal{A})$ for some $D_{i} \in T \setminus \{0\}$. Therefore, $D - D_{i} \in \text{Inn}_{(\sigma, \tau)}(\mathcal{A})$ so that $D - D_{i} = D_{0}$ for some $D_{0} \in \text{Inn}_{(\sigma, \tau)}(\mathcal{A})$. Finally, we get that $D = D_{i} + D_{0}$ so that $D \in D_{i} + \text{Inn}_{(\sigma, \tau)}(\mathcal{A})$. Hence, $D \in \bigcup_{D_{i} \in T \setminus \{0\}} (D_{i} + \text{Inn}_{(\sigma, \tau)}(\mathcal{A}))$. Therefore, $\mathcal{D}_{(\sigma, \tau)}(\mathcal{A}) \setminus \text{Inn}_{(\sigma, \tau)}(\mathcal{A}) \subseteq \bigcup_{D_{i} \in T \setminus \{0\}} (D_{i} + \text{Inn}_{(\sigma, \tau)}(\mathcal{A}))$.

Conversely, let $D \in \bigcup_{D_{i} \in T \setminus \{0\}} (D_{i} + \text{Inn}_{(\sigma, \tau)}(\mathcal{A}))$. So $D \in D_{i} + \text{Inn}_{(\sigma, \tau)}(\mathcal{A})$ for some $D_{i} \in T \setminus \{0\}$. So $D + \text{Inn}_{(\sigma, \tau)}(\mathcal{A}) = D_{i} + \text{Inn}_{(\sigma, \tau)}(\mathcal{A})$. But since $D_{i} + \text{Inn}_{(\sigma, \tau)}(\mathcal{A}) \neq \text{Inn}_{(\sigma, \tau)}(\mathcal{A})$, so $D + \text{Inn}_{(\sigma, \tau)}(\mathcal{A}) \neq \text{Inn}_{(\sigma, \tau)}(\mathcal{A})$ so that $D \notin \text{Inn}_{(\sigma, \tau)}(\mathcal{A})$, that is, $D \in \mathcal{D}_{(\sigma, \tau)}(\mathcal{A}) \setminus \text{Inn}_{(\sigma, \tau)}(\mathcal{A})$. Therefore, $\bigcup_{D_{i} \in T \setminus \{0\}} (D_{i} + \text{Inn}_{(\sigma, \tau)}(\mathcal{A})) \subseteq \mathcal{D}_{(\sigma, \tau)}(\mathcal{A}) \setminus \text{Inn}_{(\sigma, \tau)}(\mathcal{A})$.
\end{proof}

\begin{remark}\label{remark 2.7(NN)}
In view of Lemma \ref{lemma 2.6(NN)}, $D$ is a non-inner derivation of $\mathcal{A}$ if and only if $D + \text{Inn}_{(\sigma, \tau)}(\mathcal{A})$ is a non-zero element of $\text{Out}_{(\sigma, \tau)}(\mathcal{A}) = \frac{\mathcal{D}_{(\sigma, \tau)}(\mathcal{A})}{\text{Inn}_{(\sigma, \tau)}(\mathcal{A})}$. In other words, $D$ is a non-inner derivation of $\mathcal{A}$ if and only if $D + \text{Inn}_{(\sigma, \tau)}(\mathcal{A})$ is a non-trivial outer derivation of $\mathcal{A}$. Therefore, studying the non-trivial outer derivations of $\mathcal{A}$ is equivalent to studying the non-inner derivations of $\mathcal{A}$ (see {\cite[Chapter 11]{pierce}} for details).
\end{remark} 

Many authors have defined an outer derivation of a ring $R$ (or algebra $\mathcal{A}$) to be a non-inner derivation of $R$ (or $\mathcal{A}$). Some references in this regard are \cite{batty1978derivations, chuang2005identities, dhara2022note, eroǧlu2017images, hall1972derivations, miles1964derivations, prajapati2022b, sakai1966derivations, weisfeld1960derivations}. In \cite{arutyunov2019smooth} and \cite{mishchenko2020description}, the authors initially describe the set $\mathcal{D}(\mathcal{A}) \setminus \text{Inn}(\mathcal{A})$ as the set of outer derivations but then argue that it is more natural to call the quotient module $\text{Out}(\mathcal{A}) = \frac{\mathcal{D}(\mathcal{A})}{\text{Inn}(\mathcal{A})}$ as the set of outer derivations because this module can be interpreted as $1^{\text{st}}$ Hochschild cohomology module of $\mathcal{A}$ with coefficients in $\mathcal{A}$. Similarly, the quotient module $\text{Out}_{(\sigma, \tau)}(\mathcal{A}) = \frac{\mathcal{D}_{(\sigma, \tau)}(\mathcal{A})}{\text{Inn}_{(\sigma, \tau)}(\mathcal{A})}$ is our set of outer $(\sigma, \tau)$-derivations because this module can be interpreted as $1^{\text{st}}$ $(\sigma, \tau)$-Hochschild cohomology module of $\mathcal{A}$ with coefficients in $\mathcal{A}$ (see \cite{ilwale2023noncommutative}).

\begin{definition}\label{definition 2.8(NN)}
$\mathcal{A}$ is said to be $(\sigma, \tau)$-differentially trivial if $\mathcal{A}$ has only zero $(\sigma, \tau)$-derivation, that is, $\mathcal{D}_{(\sigma, \tau)}(\mathcal{A}) = \{0\}$.
\end{definition}

\begin{definition}\label{definition 2.9(NN)}
If $R$ is a ring and $G$ is a group, then the group ring of $G$ over $R$ is defined as the set $$RG = \{\sum_{g \in G} a_{g} g \mid a_{g} \in R, \forall g \in G \hspace{0.2cm} \text{and} \hspace{0.2cm} |\text{supp}(\alpha)| < \infty \},$$ where for $\alpha = \sum_{g \in G} a_{g} g$, $\text{supp}(\alpha)$ denotes the support of $\alpha$ that consists of elements from $G$ that appear in the expression of $\alpha$. The set $RG$ is a ring concerning the componentwise addition and multiplication defined respectively by: For $\alpha = \sum_{g \in G} a_{g} g$, $\beta = \sum_{g \in G} b_{g} g$ in $RG$, $$(\sum_{g \in G} a_{g} g ) + (\sum_{g \in G} b_{g} g) = \sum_{g \in G}(a_{g} + b_{g}) g \hspace{0.2cm} \text{and} \hspace{0.2cm} \alpha \beta = \sum_{g, h \in G} a_{g} b_{h} gh.$$ If the ring $R$ is commutative having unity $1$ and the group $G$ is abelian having identity $e$, then $RG$ becomes a commutative unital algebra over $R$ with identity $1 = 1e$. We adopt the convention that empty sums are $0$, and empty products are $1$.
\end{definition}

\section{$(\sigma, \tau)$-Derivations of Group Rings}\label{section 2}
In this section, we establish a necessary and sufficient condition on a map $f:X \rightarrow RG$ such that $f$ can be extended to a $(\sigma, \tau)$-derivation of the group ring $RG$, where $X$ is a set of generators of group $G$, $R$ is a commutative unital ring and $(\sigma, \tau)$ is a pair of unital (map the unity to itself) $R$-algebra endomorphisms of $RG$ that are $R$-linear extensions of the group endomorphisms of $G$. However, we first obtain the results for the case $\tau = \sigma$. Then, we use this to classify all $\sigma$-derivations of a commutative group algebra over a field of positive characteristic.

\subsection{Some Useful Results}\label{subsection 2.1}
We will use the universal property of free groups in the proof of our main Theorems \ref{theorem 2.4(N)} and \ref{theorem 2.4}.

\begin{theorem}\th\label{theorem 2.1}
(Universal Property of a Free Group): Let $X$ be an alphabet and $F(X)$ be the free group on the alphabet $X$. Let $G$ be a group and $f:X \rightarrow G$ be a set function. Then there exists a unique group homomorphism $\tilde{f}:F(X) \rightarrow G$ such that the following diagram commutes:
\[
  \begin{tikzcd}
    X \arrow{r}{i} \arrow[swap]{dr}{f} & F(X) \arrow{d}{\tilde{f}} \\
     & G
  \end{tikzcd}
\] where $i:X \rightarrow F(X)$ is the inclusion map, that is, $\tilde{f}(x) = f(x)$ for all $x \in X$.
\end{theorem}

\begin{lemma}\th\label{lemma 2.2}
Let $\mathcal{A}$ be an associative algebra with unity over a unital ring $R$, and let $\sigma$ and $\tau$ be two algebra endomorphisms of $\mathcal{A}$ which map unity $1$ of $\mathcal{A}$ to $1$ itself. Let $D$ be a $(\sigma, \tau)$-derivation of $\mathcal{A}$. Then
\begin{enumerate}
\item[(i)] $D\left(\prod_{i=1}^{k} \alpha_{i}\right) =  \sum_{i=1}^{k} \left(\prod_{j=1}^{i-1}\sigma(\alpha_{j})\right) D(\alpha_{i}) \left(\prod_{j=i+1}^{k}\tau(\alpha_{j})\right)$ for all $\alpha_{i} \in \mathcal{A}$.
\item[(ii)] $D(\alpha^{k}) = \sum_{i=0}^{k-1} (\sigma(\alpha))^{i} D(\alpha) (\tau(\alpha))^{k-i-1}$ for all $\alpha \in \mathcal{A}$ and $k \in \mathbb{N}$.
\item[(iii)] $\sum_{i=0}^{r-1} (\sigma(\alpha))^{i} D(\alpha) (\tau(\alpha))^{r-1-i} = 0$ for all units $\alpha$ in $\mathcal{A}$ of order $r$.
\item[(iv)] $D(\alpha^{k}) = \left(\sum_{i=0}^{k-1}(\sigma(\alpha))^{i} (\tau(\alpha))^{k-i-1}\right)D(\alpha)$ for all $\alpha \in \mathcal{A}$ such that $\sigma(\alpha)$ and $\tau(\alpha)$ commute with $D(\alpha)$ and $k \in \mathbb{N}$. In particular, if $\sigma = \tau$ and $\sigma(\alpha)$ commutes with $D(\alpha)$, then $D(\alpha^{k}) = k (\sigma(\alpha))^{k-1}D(\alpha)$.
\item[(v)] $D(\alpha^{k}) = k (\sigma(\alpha))^{k-1}D(\alpha)$ for any $\alpha$ which is a unit in $\mathcal{A}$ such that $\sigma(\alpha)$ commutes with $D(\alpha)$ and for any $k \in \mathbb{Z}$.
\end{enumerate}
\end{lemma}
\begin{proof}
\textbf{(i)} We prove the equality by using induction on $k$. For $k = 1$, the left side is $D(\alpha_{1})$ and the right side is $\sum_{i=1}^{1} 1 D(\alpha_{1}) 1 = D(\alpha_{1})$ which are both equal. Now, assume that the result holds for $k-1$. Then \begin{eqnarray*}
D(\prod_{i=1}^{k} \alpha_{i}) & = & D(\prod_{i=1}^{k-1}\alpha_{i})\tau(\alpha_{k}) + \sigma(\prod_{i=1}^{k-1}\alpha_{i}) D(\alpha_{k}) \\ & = & \left(\sum_{i=1}^{k-1} \left(\prod_{j=1}^{i-1}\sigma(\alpha_{j})\right) D(\alpha_{i}) \left(\prod_{j=i+1}^{k-1}\tau(\alpha_{j})\right)\right)\tau(\alpha_{k}) + \sigma(\prod_{i=1}^{k-1}\alpha_{i}) D(\alpha_{k}) \\ & = & \sum_{i=1}^{k} \left(\prod_{j=1}^{i-1}\sigma(\alpha_{j})\right) D(\alpha_{i}) \left(\prod_{j=i+1}^{k}\tau(\alpha_{j})\right).
\end{eqnarray*}

Therefore, by induction, the result holds for all $k \in \mathbb{N}$.

\textbf{(ii)} Putting $\alpha_{i} = \alpha$ for all $i \in \{1, ..., k\}$ in (i), we get that \begin{eqnarray*}
D(\alpha^{k}) & = & \sum_{i=1}^{k} \left(\left(\prod_{j=1}^{i-1}\sigma(\alpha)\right)D(\alpha) \left(\prod_{j=i+1}^{k}\tau(\alpha)\right)\right) = \sum_{i=0}^{k-1}(\sigma(\alpha))^{i} D(\alpha) (\tau(\alpha))^{k-i-1}.
\end{eqnarray*}

\textbf{(iii)} Let $\alpha \in \mathcal{A}$ be a unit of order $r$ so that $\alpha^{r} = 1$. Then by (ii), $$0 = D(1) = D(\alpha^{r}) = \sum_{i=0}^{r-1} (\sigma(\alpha))^{i} D(\alpha) (\tau(\alpha))^{r-1-i}.$$

\textbf{(iv)} Let $\alpha \in \mathcal{A}$ such that $\sigma(\alpha)$ and $\tau(\alpha)$ commute with $D(\alpha)$. Then by (ii), $$D(\alpha^{k}) = \left(\sum_{i=0}^{k-1}(\sigma(\alpha))^{i} (\tau(\alpha))^{k-i-1}\right)D(\alpha)$$ for $k \in \mathbb{N}$. If $\tau = \sigma$, then $$D(\alpha^{k}) = \left(\sum_{i=0}^{k-1}(\sigma(\alpha))^{k-1}\right) D(\alpha) \\ = k (\sigma(\alpha))^{k-1}D(\alpha).$$

\textbf{(v)} Let $\alpha$ be a unit in $\mathcal{A}$ such that $\sigma(\alpha)$ and $D(\alpha)$ commute with each other. The result holds for all $k \in \mathbb{N}$ by (iv). For $k = 0$, $D(\alpha^{0}) = D(1) = 0$ and the right side is trivially $0$ being an empty sum. So, the equality also holds for $k=0$.

We have $$\sigma(\alpha^{-1})D(\alpha) = \sigma(\alpha^{-1})D(\alpha)\sigma(\alpha \alpha^{-1}) = \sigma(\alpha^{-1})\sigma(\alpha)D(\alpha)\sigma(\alpha^{-1}) = D(\alpha)\sigma(\alpha^{-1}).$$ Therefore, $\alpha^{-1}$ is also a unit in $\mathcal{A}$ such that $\sigma(\alpha^{-1})$ and $D(\alpha)$ commute with each other. Thus, $$0 = D(1) = D(\alpha^{-1} \alpha) = D(\alpha^{-1})\sigma(\alpha) + \sigma(\alpha^{-1})D(\alpha)$$ so that $D(\alpha^{-1}) = - \sigma(\alpha^{-2}) D(\alpha).$
Further, $$\sigma(\alpha^{-1})D(\alpha^{-1}) = \sigma(\alpha^{-1}) (- \sigma(\alpha^{-2} )D(\alpha)) = (- \sigma(\alpha^{-2}) D(\alpha))\sigma(\alpha^{-1}) = D(\alpha^{-1}) \sigma(\alpha^{-1}).$$
Now, by (iv), we have that for any $k \in \mathbb{N}$, 
\begin{eqnarray*}
D(\alpha^{-k}) & = & D((\alpha^{-1})^{k}) \\ & = & k (\sigma(\alpha^{-1}))^{k-1}D(\alpha^{-1}) \\ & = & k (\sigma(\alpha))^{-k+1}(- (\sigma(\alpha))^{-2} D(\alpha)) \\ & = & -k (\sigma(\alpha))^{-k-1}D(\alpha).
\end{eqnarray*}
Therefore, the result in (v) is true for all $k \in \mathbb{Z}$.
\end{proof}

\subsection{The Main Theorem for $\sigma$-Derivations}\label{subsection 2.2(N)}
\begin{lemma}\th\label{lemma 2.3(N)}
Let $G = \langle X \mid Y \rangle$ be a group with $X$ as its set of generators and $Y$ as its set of relators. Let $R$ be a commutative unital ring, $\sigma$ be an $R$-algebra endomorphism of $RG$, and an $R$-linear extension of a group endomorphism of $G$. Let $F(X)$ be the free group on $X$. Then any map $f:X \rightarrow RG$ can be extended uniquely to a map $\tilde{f}:F(X) \rightarrow RG$ such that \begin{equation}\label{eq 2.1(N)}
\tilde{f}(vw) = \tilde{f}(v)\tilde{\sigma}(w) + \tilde{\sigma}(v)\tilde{f}(w), \hspace{0.1cm} \forall \hspace{0.1cm} v, w \in F(X),\end{equation} where $\tilde{\sigma}:F(X) \rightarrow G$ is the unique group homomorphism such that $\tilde{\sigma}(x) = \sigma(x)$  for all $x \in X$.
\end{lemma}
\begin{proof}
By \th\ref{theorem 2.1}, $\sigma$ can be extended uniquely to a group homomorphism $\tilde{\sigma}:F(X) \rightarrow G$ such that $\tilde{\sigma}(x) = \sigma(x)$ for all $x \in X$. Put $X^{-1} = \{x^{-1} \mid x \in X\}$. Then for any $x \in X$, $$1 = \tilde{\sigma}(1) = \tilde{\sigma}(x) \tilde{\sigma}(x^{-1}) = \sigma(x) \tilde{\sigma}(x^{-1})$$ so that $\tilde{\sigma}(x^{-1}) = \sigma(x^{-1})$. Therefore, $\tilde{\sigma}(g) = \sigma(g)$ for all $g \in G$.

Define $\tilde{f}:F(X) \rightarrow RG$ as \begin{equation}\label{eq 2.2(N)}
\tilde{f}(x) = \begin{cases}
f(x) & \text{if $x \in X$} \\
-\sigma(x)f(x^{-1})\sigma(x) & \text{if $x \in X^{-1}$} \\
0 & \text{if $x = 1$}
\end{cases}
\end{equation} and if $w = \prod_{i=1}^{m}x_{i}$, for $x_{i} \in X \cup X^{-1}$ ($1 \leq i \leq m$), then \begin{equation}\label{eq 2.3(N)}
\tilde{f}(w) = \sum_{i=1}^{m} \left(\left(\prod_{j=1}^{i-1} \sigma(x_{j})\right)\tilde{f}(x_{i})\left(\prod_{j=i+1}^{m} \sigma(x_{j})\right)\right).\end{equation}
Let $0 \leq m \leq n$, $v = \prod_{i=1}^{m} x_{i}$ and $w = \prod_{i=m+1}^{n}x_{i}$. Then by (\ref{eq 2.2(N)}) and (\ref{eq 2.3(N)}), 
\begin{equation*}
\begin{aligned}
\tilde{f}(vw) & = \tilde{f}((\prod_{i=1}^{m} x_{i})(\prod_{i=m+1}^{n}x_{i})) =  \tilde{f}(\prod_{i=1}^{n}x_{i}) = \sum_{i=1}^{n} \left(\left(\prod_{j=1}^{i-1} \sigma(x_{j})\right)\tilde{f}(x_{i})\left(\prod_{j=i+1}^{n} \sigma(x_{j})\right)\right)
\\ & = \left(\sum_{i=1}^{m} \left(\prod_{j=1}^{i-1}\sigma(x_{j})\right)\tilde{f}(x_{i})\left(\prod_{j=i+1}^{m}\sigma(x_{j})\right)\right)\left(\prod_{j=m+1}^{n}\sigma(x_{j})\right) \\ &\quad + (\prod_{i=1}^{m}\sigma(x_{j}))(\sum_{i=m+1}^{n}(\prod_{j=m+1}^{i-1}\sigma(x_{j}))\tilde{f}(x_{i})(\prod_{j=i+1}^{n}\sigma(x_{j}))) 
\\ & = \tilde{f}(v)\tilde{\sigma}(w) + \tilde{\sigma}(v)\tilde{f}(w).
\end{aligned}
\end{equation*}
Therefore, $\tilde{f}$ satisfies (\ref{eq 2.1(N)}).

For any word $w$ on $X$, let $\overline{w}$ denote the reduced word on $X$. To show that the map $\tilde{f}$ is well-defined, it must be shown that $\tilde{f}(w) = \tilde{f}(\overline{w})$ for all words $w$ on $X$.

Using (\ref{eq 2.1(N)}) and (\ref{eq 2.2(N)}), note that for any $x \in X$, $$\tilde{f}(xx^{-1}) = \tilde{f}(x)\tilde{\sigma}(x^{-1}) + \tilde{\sigma}(x) \tilde{f}(x^{-1}) = f(x)\sigma(x^{-1}) - \sigma(x) \sigma(x^{-1}) f(x) \sigma(x^{-1}) = 0.$$ Similarly, for any $x \in X^{-1}$, $\tilde{f}(xx^{-1}) = 0$. Therefore, for any two words $v$ and $w$ on $X$ and for any $x \in X \cup X^{-1}$, \begin{eqnarray*}
\tilde{f}(vxx^{-1}w) & = & \tilde{f}(v)\tilde{\sigma}(xx^{-1}w) + \tilde{\sigma}(v)\tilde{f}(xx^{-1}w) \\ & = & \tilde{f}(v)\tilde{\sigma}(w) + \tilde{\sigma}(v)(\tilde{f}(xx^{-1})\tilde{\sigma}(w) + \tilde{\sigma}(xx^{-1})\tilde{f}(w)) \\ & = & \tilde{f}(v)\tilde{\sigma}(w) +  \tilde{\sigma}(v)\tilde{f}(xx^{-1})\tilde{\sigma}(w) + \tilde{\sigma}(v)\tilde{f}(w) \\ & = & \tilde{f}(v)\tilde{\sigma}(w) + \tilde{\sigma}(v)\tilde{f}(w) \\ & = & \tilde{f}(vw).
\end{eqnarray*}
Therefore, $\tilde{f}(w) = \tilde{f}(\overline{w})$ for all words $w$ on $X$.

Now, it remains to show the uniqueness of $\tilde{f}$. If possible, suppose that $\tilde{g}:F(X) \rightarrow RG$ is another extension of $f$ different from $\tilde{f}$ that satisfies $$\tilde{g}(vw) = \tilde{g}(v)\tilde{\sigma}(w) + \tilde{\sigma}(v)\tilde{g}(w)$$ for all $v, w \in F(X)$.

Note that $$\tilde{g}(1) = \tilde{g}(1)\tilde{\sigma}(1) + \tilde{\sigma}(1)\tilde{g}(1) = \tilde{g}(1) + \tilde{g}(1)$$ so that $\tilde{g}(1) = 0$. Therefore, $\tilde{g}(1) = 0 = \tilde{f}(1)$.

Again, since $\tilde{f}$ and $\tilde{g}$ are extensions of $f:X \rightarrow RG$, therefore, $$\tilde{g}(x) = f(x) = \tilde{f}(x)$$ for all $x \in X$.

Further, for any $x \in X$, $$0 = \tilde{g}(1) = \tilde{g}(xx^{-1}) = \tilde{g}(x)\tilde{\sigma}(x^{-1}) + \tilde{\sigma}(x)\tilde{g}(x^{-1})$$ so that $$\tilde{g}(x^{-1}) = - \sigma(x^{-1}) f(x) \sigma(x^{-1}) = \tilde{f}(x^{-1}).$$ Therefore, $\tilde{g}(x) = \tilde{f}(x)$ for all $x \in X \cup X^{-1}$.

Since $\tilde{g} \neq \tilde{f}$, so there exists some $w_{0} \in F(X)$ such that $\tilde{g}(w_{0}) \neq \tilde{f}(w_{0})$ and $\tilde{g}(w) = \tilde{f}(w)$ for all words $w \in F(X)$ whose length is strictly less than that of $w_{0}$. Now, as $w_{0} \in F(X)$, so $w_{0} = \prod_{i=1}^{m_{0}}x_{i}$ for some $x_{i} \in X \cup X^{-1}$ ($1 \leq i \leq m_{0}$), where $m_{0}$ is the length of the word $w_{0}$. Since $\prod_{i=1}^{m_{0}-1}x_{j}, x_{m_{0}} \in F(X)$ such that their lengths are strictly less than the length $m_{0}$ of $w_{0}$, so by our choice of $w_{0}$, $$\tilde{g}(\prod_{i=1}^{m_{0}-1}x_{j}) = \tilde{f}(\prod_{i=1}^{m_{0}-1}x_{j}) ~~~~~~ \text{and} ~~~~~~ \tilde{g}(x_{m_{0}}) = \tilde{f}(x_{m_{0}}).$$

This gives 
\begin{eqnarray*}
\tilde{g}(w_{0}) & = & \tilde{g}(\prod_{i=1}^{m_{0}}x_{i}) = \tilde{g}((\prod_{i=1}^{m_{0}-1}x_{i})x_{m_{0}}) \\ & = & \tilde{g}(\prod_{i=1}^{m_{0}-1}x_{j})\tilde{\sigma}(x_{m_{0}}) + \tilde{\sigma}(\prod_{i=1}^{m_{0}-1}x_{j}) \tilde{g}(x_{m_{0}}) \\ & = & \tilde{f}(\prod_{i=1}^{m_{0}-1}x_{j})\tilde{\sigma}(x_{m_{0}}) + \tilde{\sigma}(\prod_{i=1}^{m_{0}-1}x_{j}) \tilde{f}(x_{m_{0}}) \\ & = & \tilde{f}(\prod_{i=1}^{m_{0}}x_{i}) = \tilde{f}((\prod_{i=1}^{m_{0}-1}x_{i})x_{m_{0}}) \\ & = & \tilde{f}(w_{0})
\end{eqnarray*}
But this is a contradiction to the fact that $\tilde{g}(w_{0}) \neq \tilde{f}(w_{0})$. Therefore, $\tilde{f}$ is the unique extension of $f:X \rightarrow RG$ from $F(X)$ to $RG$ that satisfies (\ref{eq 2.1(N)}).
\end{proof}

\begin{theorem}\th\label{theorem 2.4(N)}
Under the hypotheses of \th\ref{lemma 2.3(N)}, a map $f:X \rightarrow RG$ can be extended to a $\sigma$-derivation $D$ of $RG$ if and only if $\tilde{f}(y) = 0$ for all $y \in Y$.
\end{theorem}
\begin{proof}
First, suppose that the given map $f:X \rightarrow RG$ can be extended to a $\sigma$-derivation $D$ of $RG$. So $D(x) = f(x)$ for all $x \in X$. Also, for any $x \in X$, $$0 = D(1) = D(xx^{-1}) = D(x)\sigma(x^{-1}) + \sigma(x)D(x^{-1})$$ so that $$D(x^{-1}) = - \sigma(x^{-1})D(x) \sigma(x^{-1}) = - \sigma(x^{-1})f(x)\sigma(x^{-1}) = \tilde{f}(x^{-1}).$$ Therefore, $D(x) = \tilde{f}(x)$ for all $x \in X \cup X^{-1}$.

Now, let $y \in Y$. Then $y = \prod_{i=1}^{k}y_{i}$ for some $y_{i} \in X \cup X^{-1}$ ($1 \leq i \leq k$). By (\ref{eq 2.3(N)}) and \th\ref{lemma 2.2} (i),
\begin{eqnarray*}
\tilde{f}(y) = \tilde{f}(\prod_{i=1}^{k}y_{i}) & = & \sum_{i=1}^{k}\left(\left(\prod_{j=1}^{i-1}\sigma(y_{j})\right)\tilde{f}(y_{i})\left(\prod_{j=i+1}^{k} \sigma(y_{j})\right)\right) \\ & = & \sum_{i=1}^{k}\left(\left(\prod_{j=1}^{i-1}\sigma(y_{j})\right)D(y_{i})\left(\prod_{j=i+1}^{k} \sigma(y_{j})\right)\right) \\ & = & D(\prod_{i=1}^{k}y_{i}) = D(y) = D(1) = 0.
\end{eqnarray*}
Therefore, $\tilde{f}(y) = 0$ for all $y \in Y$.

Conversely, assume that $\tilde{f}(y) = 0$ for all $y \in Y$. Note that $\tilde{\sigma}(y) = \sigma(y) = 1$ for all $y \in Y$. Therefore, for any $y \in Y$, $$0 = \tilde{f}(1) = \tilde{f}(yy^{-1}) = \tilde{f}(y)\tilde{\sigma}(y^{-1}) + \tilde{\sigma}(y)\tilde{f}(y^{-1}) = \tilde{f}(y^{-1}).$$ Therefore, $\tilde{f}(y^{-1}) = 0$ for all $y \in Y$.

Let $Y^{F(X)}$ be the normal closure of $Y$ in $F(X)$. Then $$Y^{F(X)} = \langle w^{-1}yw \mid w \in F(X), y \in Y\rangle$$ and it is the kernel of the unique onto group homomorphism $\phi: F(X) \rightarrow G$ which is the identity on $X$, that is, $G \cong \frac{F(X)}{Y^{F(X)}}$. Let $\varepsilon \in \{1, -1\}$. Then using (\ref{eq 2.1(N)}) and the fact that $\tilde{f}(y^{\varepsilon}) = 0$ for all $y \in Y$, we get that for any $w \in F(X)$, \begin{eqnarray*}\tilde{f}(w^{-1}y^{\varepsilon}w) & = & \tilde{f}(w^{-1})\tilde{\sigma}(y^{\varepsilon}w) + \tilde{\sigma}(w^{-1})\tilde{f}(y^{\varepsilon}w) \\ & = & \tilde{f}(w^{-1})\tilde{\sigma}(w) + \tilde{\sigma}(w^{-1})(\tilde{f}(y^{\varepsilon})\tilde{\sigma}(w) + \tilde{\sigma}(y^{\varepsilon})\tilde{f}(w))\\ & = &  \tilde{f}(w^{-1})\tilde{\sigma}(w) + \tilde{\sigma}(w^{-1})\tilde{f}(w) \\ & = & \tilde{f}(w^{-1}w) = \tilde{f}(1) = 0. \end{eqnarray*}
Now, let $v \in Y^{F(X)}$. Then $v = \prod_{i=1}^{m}w_{i}^{-1}y_{i}^{\varepsilon}w_{i}$ for some $m \in \mathbb{N}$, $w_{i} \in F(X)$, $y_{i} \in Y$ ($1 \leq i \leq m$). As shown above, the result holds for $m=1$. Assume that the result holds for $m=k-1$, that is, $\tilde{f}(\prod_{i=1}^{k-1}w_{i}^{-1}y_{i}^{\varepsilon}w_{i}) = 0$. Then using (\ref{eq 2.3(N)}) and the induction hypothesis, we get that for $m=k$,
\begin{eqnarray*}
\tilde{f}(v) & = & \tilde{f}(\prod_{i=1}^{k}w_{i}^{-1}y_{i}^{\varepsilon}w_{i}) = \tilde{f}((\prod_{i=1}^{k-1}w_{i}^{-1}y_{i}^{\varepsilon}w_{i})(w_{k}^{-1}y_{k}^{\varepsilon}w_{k})) \\ & = & \tilde{f}(\prod_{i=1}^{k-1}w_{i}^{-1}y_{i}^{\varepsilon}w_{i})\tilde{\sigma}(w_{k}^{-1}y_{k}^{\varepsilon}w_{k}) + \tilde{\sigma}(\prod_{i=1}^{k-1}w_{i}^{-1}y_{i}^{\varepsilon}w_{i}) \tilde{f}(w_{k}^{-1}y_{k}^{\varepsilon}w_{k}) \\ & = & 0\tilde{\sigma}(w_{k}^{-1}y_{k}^{\varepsilon}w_{k}) + \tilde{\sigma}(\prod_{i=1}^{k-1}w_{i}^{-1}y_{i}^{\varepsilon}w_{i})0 = 0.
\end{eqnarray*}
Therefore, $\tilde{f}(v) = 0$ for all $v \in Y^{F(X)}$. Also, $\tilde{\sigma}(v) = 1$ for all $v \in Y^{F(X)}$. Therefore, for any $w \in F(X)$ and $v \in Y^{F(X)}$, $$\tilde{f}(wv) = \tilde{f}(w)\tilde{\sigma}(v) + \tilde{\sigma}(w)\tilde{f}(v) = \tilde{f}(w)1 + \tilde{\sigma}(w)0 = \tilde{f}(w).$$

Let $g \in G$ and $g = \prod_{i=1}^{m}x_{i}$ for some $m \in \mathbb{N}$, $x_{i} \in X \cup X^{-1}$ ($1 \leq i \leq m$). So $g$ is a word on $X$ so that $g \in F(X)$. In fact, $g$ is the reduced word in $F(X)$ with $\phi(g) = g$. If $w \in F(X)$ such that $\phi(w) = g$, then $w$ is either equivalent to the reduced word $g$ or the expression for $w$ (as a product of elements from $X \cup X^{-1}$) contains elements from $Y^{F(X)}$. In any case, $\tilde{\sigma}(w) = \sigma(g)$.

Now, define $D:G \rightarrow RG$ by $$D(g) = \tilde{f}(w)$$ where for $g \in G$, $w \in F(X)$ is such that $\phi(w) = g$. Such a $w$ exists as $\phi: F(X) \rightarrow G$ is an onto group homomorphism, or as explained in the above paragraph, we can take $w=g$. The map $D$ is well-defined because as shown already, $\tilde{f}:F(X) \rightarrow RG$ is a well-defined map with $\tilde{f}(wv) = \tilde{f}(w)$ for all $v \in Y^{F(X)}$.

Now, let $g, h \in G$. Then there exists some $u, w \in F(X)$ such that $$\phi(u) = g ~~~~~~ \text{and} ~~~~~~ \phi(w) = h.$$ Then $$\tilde{\sigma}(u) = \sigma(g), ~~~~ \tilde{\sigma}(w) = \sigma(h) ~~~~ \text{and} ~~~~ \phi(uw) = \phi(u)\phi(w) = gh.$$ Therefore, by the definition of $D$, $$D(g) = \tilde{f}(u), ~~~~ D(h) = \tilde{f}(w) ~~~~ \text{and} ~~~~ D(gh) = \tilde{f}(uw).$$ Finally, using (\ref{eq 2.1(N)}), we get that
$$D(gh) = \tilde{f}(uw) = \tilde{f}(u)\tilde{\sigma}(w) + \tilde{\sigma}(u)\tilde{f}(w) = D(g)\sigma(h) + \sigma(g)D(h).$$
Therefore, $D(gh) = D(g)\sigma(h) + \sigma(g)D(h)$ for all $g, h \in G$.

Now, extend $D$ $R$-linearly to the whole of $RG$. We denote the extended map also by $D$. So if $\lambda \in R$, and $\alpha = \sum_{g \in G}\lambda_{g}g$, $\beta = \sum_{h \in G}\mu_{h}h \in RG$ for some $\lambda_{g}, \mu_{h} \in R$ ($g, h \in G$), then $D(\alpha + \beta) = D(\alpha) + D(\beta)$ and $D(\lambda \alpha) = \lambda D(\alpha)$ as $D$ is an $R$-linear map. Also, \begin{eqnarray*}
D(\alpha \beta) & = & D((\sum_{g \in G}\lambda_{g}g)(\sum_{g \in G}\mu_{h}h)) \\ & = & D(\sum_{g,h \in G}\lambda_{g} \mu_{h}(gh)) = \sum_{g,h \in G}\lambda_{g} \mu_{h}D(gh) \\ & = & \sum_{g,h \in G}\lambda_{g} \mu_{h}(D(g)\sigma(h) + \sigma(g)D(h)) \\ & = & \sum_{g,h \in G}\lambda_{g} \mu_{h}D(g)\sigma(h) + \sum_{g,h \in G}\lambda_{g} \mu_{h}\sigma(g)D(h) \\ & = & D(\sum_{g \in G} \lambda_{g}g) \sigma(\sum_{h \in G} \mu_{h}h) + \sigma(\sum_{g \in G} \lambda_{g}g) D(\sum_{h \in G} \mu_{h}h) \\ & = & D(\alpha)\sigma(\beta) + \sigma(\alpha)D(\beta).
\end{eqnarray*}
Therefore, $D$ is a $\sigma$-derivation of $RG$ that extends $f:X \rightarrow RG$ to the whole of $RG$.

Now, it remains to show the uniqueness of $D$. If possible, suppose that $\tilde{D}: RG \rightarrow RG$ with $\tilde{D} \neq D$ is another $\sigma$-derivation of $RG$ that extends $f:X \rightarrow RG$. Then there exists some $g_{0} \in G$ such that $\tilde{D}(g_{0}) \neq D(g_{0})$ and $\tilde{D}(g) = D(g)$ for all $g \in G$ such that the length of $g$ is strictly less than that of $g_{0}$.

Note that $D(1) = 0 = \tilde{D}(1)$. Also, $\tilde{D}(x) = f(x) = D(x)$ for all $x \in X$. Now, for any $x \in X$, $$0 = \tilde{D}(1) = \tilde{D}(xx^{-1}) = \tilde{D}(x)\sigma(x^{-1}) + \sigma(x)\tilde{D}(x^{-1}) = f(x)\sigma(x^{-1}) + \sigma(x)\tilde{D}(x^{-1})$$ so that $\tilde{D}(x^{-1}) = -\sigma(x^{-1})f(x)\sigma(x^{-1}) = \tilde{f}(x^{-1}) = D(x^{-1})$.

We can write $g_{0} \in G$ as $g_{0} = x_{0}y_{0}$ for some $x_{0}, y_{0} \in G$ such that the lengths of $x_{0}$ and $y_{0}$ are strictly less than that of $g_{0}$. Then in view of our assumption, $\tilde{D}(x_{0}) = D(x_{0})$ and $\tilde{D}(y_{0}) = D(y_{0})$. Therefore, \begin{eqnarray*}
\tilde{D}(g_{0}) & = & \tilde{D}(x_{0}y_{0}) \\ & = & \tilde{D}(x_{0})\sigma(y_{0}) + \sigma(x_{0})\tilde{D}(y_{0}) \\ & = & D(x_{0})\sigma(y_{0}) + \sigma(x_{0})D(y_{0}) \\ & = & D(x_{0}y_{0}) \\ & = & D(g_{0}).
\end{eqnarray*} But this is a contradiction to the fact that $\tilde{D}(g_{0}) \neq D(g_{0})$. Therefore, $D: RG \rightarrow RG$ is the unique $\sigma$-derivation that extends the given map $f:X \rightarrow RG$.

Hence proved.
\end{proof}

\subsection{The Main Theorem for $(\sigma, \tau)$-Derivations}\label{subsection 2.2}
In this Subsection, we prove Lemma \ref{lemma 2.3(N)} and Theorem \ref{theorem 2.4(N)}, in general, for $(\sigma, \tau)$-derivations. The idea of the proofs is almost similar, but we still provide the proofs for the sake of completeness.

\begin{lemma}\th\label{lemma 2.3}
Let $G = \langle X \mid Y \rangle$ be a group with $X$ as its set of generators and $Y$ as its set of relators. Let $R$ be a commutative unital ring, and $\sigma$ and $\tau$ be $R$-algebra endomorphisms of $RG$, which are $R$-linear extensions of group endomorphisms of $G$. Let $F(X)$ be the free group on $X$. Then any map $f:X \rightarrow RG$ can be extended uniquely to a map $\tilde{f}:F(X) \rightarrow RG$ such that \begin{equation}\label{eq 2.1}
\tilde{f}(vw) = \tilde{f}(v)\tilde{\tau}(w) + \tilde{\sigma}(v)\tilde{f}(w), \hspace{0.1cm} \forall \hspace{0.1cm} v, w \in F(X),\end{equation} where $\tilde{\sigma}, \tilde{\tau}:F(X) \rightarrow G$ are the unique group homomorphisms such that $\tilde{\sigma}(x) = \sigma(x)$ and $\tilde{\tau}(x) = \tau(x)$ for all $x \in X$.
\end{lemma}
\begin{proof}
By \th\ref{theorem 2.1}, $\sigma$ and $\tau$ can be extended uniquely to group homomorphisms $\tilde{\sigma}:F(X) \rightarrow G$ and $\tilde{\tau}:F(X) \rightarrow G$ such that $\tilde{\sigma}(x) = \sigma(x)$ and $\tilde{\tau}(x) = \tau(x)$ for all $x \in X$. 

As shown in the proof of Lemma \ref{lemma 2.3(N)}, $\tilde{\sigma}(g) = \sigma(g)$ for all $g \in G$. On similar lines, it can be shown that $\tilde{\tau}(g) = \tau(g)$ for all $g \in G$.

Define $\tilde{f}:F(X) \rightarrow RG$ as \begin{equation}\label{eq 2.2}
\tilde{f}(x) = \begin{cases}
f(x) & \text{if $x \in X$} \\
-\sigma(x)f(x^{-1})\tau(x) & \text{if $x \in X^{-1}$} \\
0 & \text{if $x = 1$}
\end{cases}
\end{equation} and if $w = \prod_{i=1}^{m}x_{i}$, for $x_{i} \in X \cup X^{-1}$ ($1 \leq i \leq m$), then \begin{equation}\label{eq 2.3}
\tilde{f}(w) = \sum_{i=1}^{m} \left(\left(\prod_{j=1}^{i-1} \sigma(x_{j})\right)\tilde{f}(x_{i})\left(\prod_{j=i+1}^{m} \tau(x_{j})\right)\right).\end{equation}
Let $0 \leq m \leq n$, $v = \prod_{i=1}^{m} x_{i}$ and $w = \prod_{i=m+1}^{n}x_{i}$. Then by (\ref{eq 2.2}) and (\ref{eq 2.3}), 
\begin{equation*}
\begin{aligned}
\tilde{f}(vw) & = \tilde{f}((\prod_{i=1}^{m} x_{i})(\prod_{i=m+1}^{n}x_{i})) =  \tilde{f}(\prod_{i=1}^{n}x_{i}) = \sum_{i=1}^{n} \left(\left(\prod_{j=1}^{i-1} \sigma(x_{j})\right)\tilde{f}(x_{i})\left(\prod_{j=i+1}^{n} \tau(x_{j})\right)\right)
\\ & = \left(\sum_{i=1}^{m} \left(\prod_{j=1}^{i-1}\sigma(x_{j})\right)\tilde{f}(x_{i})\left(\prod_{j=i+1}^{m}\tau(x_{j})\right)\right)\left(\prod_{j=m+1}^{n}\tau(x_{j})\right) \\ &\quad + (\prod_{i=1}^{m}\sigma(x_{j}))(\sum_{i=m+1}^{n}(\prod_{j=m+1}^{i-1}\sigma(x_{j}))\tilde{f}(x_{i})(\prod_{j=i+1}^{n}\tau(x_{j}))) 
\\ & = \tilde{f}(v)\tilde{\tau}(w) + \tilde{\sigma}(v)\tilde{f}(w).
\end{aligned}
\end{equation*}
Therefore, $\tilde{f}$ satisfies (\ref{eq 2.1}).

For any word $w$ on $X$, let $\overline{w}$ denote the reduced word on $X$. To show that the map $\tilde{f}$ is well-defined, it must be shown that $\tilde{f}(w) = \tilde{f}(\overline{w})$ for all words $w$ on $X$.

Using (\ref{eq 2.1}) and (\ref{eq 2.2}), note that for any $x \in X$, $$\tilde{f}(xx^{-1}) = \tilde{f}(x)\tilde{\tau}(x^{-1}) + \tilde{\sigma}(x) \tilde{f}(x^{-1}) = f(x)\tau(x^{-1}) - \sigma(x) \sigma(x^{-1}) f(x) \tau(x^{-1}) = 0.$$ Similarly, for any $x \in X^{-1}$, $\tilde{f}(xx^{-1}) = 0$. Therefore, for any two words $v$ and $w$ on $X$ and for any $x \in X \cup X^{-1}$, \begin{eqnarray*}\tilde{f}(vxx^{-1}w) & = & \tilde{f}(v)\tilde{\tau}(xx^{-1}w) + \tilde{\sigma}(v)\tilde{f}(xx^{-1}w) \\ & = & \tilde{f}(v)\tilde{\tau}(w) + \tilde{\sigma}(v)(\tilde{f}(xx^{-1})\tilde{\tau}(w) + \tilde{\sigma}(xx^{-1})\tilde{f}(w)) \\ & = & \tilde{f}(v)\tilde{\tau}(w) +  \tilde{\sigma}(v)\tilde{f}(xx^{-1})\tilde{\tau}(w) + \tilde{\sigma}(v)\tilde{f}(w) \\ & = & \tilde{f}(v)\tilde{\tau}(w) + \tilde{\sigma}(v)\tilde{f}(w) \\ & = & \tilde{f}(vw).
\end{eqnarray*} Therefore, $\tilde{f}(w) = \tilde{f}(\overline{w})$ for all words $w$ on $X$.

Now, as in the proof of Lemma \ref{lemma 2.3(N)}, we show the uniqueness of $\tilde{f}$. If possible, suppose that $\tilde{g}:F(X) \rightarrow RG$ is another extension of $f$ different from $\tilde{f}$ that satisfies $$\tilde{g}(vw) = \tilde{g}(v)\tilde{\tau}(w) + \tilde{\sigma}(v)\tilde{g}(w)$$ for all $v, w \in F(X)$.

Note that $$\tilde{g}(1) = \tilde{g}(1)\tilde{\tau}(1) + \tilde{\sigma}(1)\tilde{g}(1) = \tilde{g}(1) + \tilde{g}(1)$$ so that $\tilde{g}(1) = 0$. Therefore, $\tilde{g}(1) = 0 = \tilde{f}(1)$.

Again, since $\tilde{f}$ and $\tilde{g}$ are extensions of $f:X \rightarrow RG$, therefore, $$\tilde{g}(x) = f(x) = \tilde{f}(x)$$ for all $x \in X$.

Further, for any $x \in X$, $$0 = \tilde{g}(1) = \tilde{g}(xx^{-1}) = \tilde{g}(x)\tilde{\tau}(x^{-1}) + \tilde{\sigma}(x)\tilde{g}(x^{-1})$$ so that $$\tilde{g}(x^{-1}) = - \sigma(x^{-1}) f(x) \tau(x^{-1}) = \tilde{f}(x^{-1}).$$ Therefore, $\tilde{g}(x) = \tilde{f}(x)$ for all $x \in X \cup X^{-1}$.

Following the steps as in the proof of Lemma \ref{lemma 2.3(N)}, since $\tilde{g} \neq \tilde{f}$, so there exists some $w_{0} \in F(X)$ such that $\tilde{g}(w_{0}) \neq \tilde{f}(w_{0})$ and $\tilde{g}(w) = \tilde{f}(w)$ for all words $w \in F(X)$ whose length is strictly less than that of $w_{0}$. As $w_{0} \in F(X)$, so $w_{0} = \prod_{i=1}^{m_{0}}x_{i}$ for some $x_{i} \in X \cup X^{-1}$ ($1 \leq i \leq m_{0}$), where $m_{0}$ is the length of the word $w_{0}$. Since $\prod_{i=1}^{m_{0}-1}x_{j}, x_{m_{0}} \in F(X)$ such that their lengths are strictly less than the length $m_{0}$ of $w_{0}$, so by our choice of $w_{0}$, $$\tilde{g}(\prod_{i=1}^{m_{0}-1}x_{j}) = \tilde{f}(\prod_{i=1}^{m_{0}-1}x_{j}) ~~~~~~ \text{and} ~~~~~~ \tilde{g}(x_{m_{0}}) = \tilde{f}(x_{m_{0}}).$$

This gives \begin{eqnarray*}
\tilde{g}(w_{0}) & = & \tilde{g}(\prod_{i=1}^{m_{0}}x_{i}) = \tilde{g}((\prod_{i=1}^{m_{0}-1}x_{i})x_{m_{0}}) \\ & = & \tilde{g}(\prod_{i=1}^{m_{0}-1}x_{j})\tilde{\tau}(x_{m_{0}}) + \tilde{\sigma}(\prod_{i=1}^{m_{0}-1}x_{j}) \tilde{g}(x_{m_{0}}) \\ & = & \tilde{f}(\prod_{i=1}^{m_{0}-1}x_{j})\tilde{\tau}(x_{m_{0}}) + \tilde{\sigma}(\prod_{i=1}^{m_{0}-1}x_{j}) \tilde{f}(x_{m_{0}})\\ & = & \tilde{f}(\prod_{i=1}^{m_{0}}x_{i}) = \tilde{f}((\prod_{i=1}^{m_{0}-1}x_{i})x_{m_{0}}) \\ & = & \tilde{f}(w_{0}).
\end{eqnarray*} But this is a contradiction to the fact that $\tilde{g}(w_{0}) \neq \tilde{f}(w_{0})$. Therefore, $\tilde{f}$ is the unique extension of $f:X \rightarrow RG$ from $F(X)$ to $RG$ that satisfies (\ref{eq 2.1}).
\end{proof}

\begin{theorem}\th\label{theorem 2.4}
Under the hypotheses of \th\ref{lemma 2.3}, a map $f:X \rightarrow RG$ can be extended to a $(\sigma, \tau)$-derivation $D$ of $RG$ if and only if $\tilde{f}(y) = 0$ for all $y \in Y$.
\end{theorem}
\begin{proof}
First, suppose that the given map $f:X \rightarrow RG$ can be extended to a $(\sigma, \tau)$-derivation $D$ of $RG$. So $D(x) = f(x)$ for all $x \in X$. Also, for any $x \in X$, $$0 = D(1) = D(xx^{-1}) = D(x)\tau(x^{-1}) + \sigma(x)D(x^{-1})$$ so that $$D(x^{-1}) = - \sigma(x^{-1})D(x) \tau(x^{-1}) = - \sigma(x^{-1})f(x)\tau(x^{-1}) = \tilde{f}(x^{-1}).$$ Therefore, $D(x) = \tilde{f}(x)$ for all $x \in X \cup X^{-1}$.

Now, let $y \in Y$. Then $y = \prod_{i=1}^{k}y_{i}$ for some $y_{i} \in X \cup X^{-1}$ ($1 \leq i \leq k$). By (\ref{eq 2.3}) and \th\ref{lemma 2.2} (i),
\begin{eqnarray*}
\tilde{f}(y) = \tilde{f}(\prod_{i=1}^{k}y_{i}) & = & \sum_{i=1}^{k}\left(\left(\prod_{j=1}^{i-1}\sigma(y_{j})\right)\tilde{f}(y_{i})\left(\prod_{j=i+1}^{k} \tau(y_{j})\right)\right) \\ & = & \sum_{i=1}^{k}\left(\left(\prod_{j=1}^{i-1}\sigma(y_{j})\right)D(y_{i})\left(\prod_{j=i+1}^{k} \tau(y_{j})\right)\right) \\ & = & D(\prod_{i=1}^{k}y_{i}) = D(y) = D(1) = 0.
\end{eqnarray*}
Therefore, $\tilde{f}(y) = 0$ for all $y \in Y$.

Conversely, assume that $\tilde{f}(y) = 0$ for all $y \in Y$. Note that $\tilde{\sigma}(y) = \sigma(y) = 1$ and $\tilde{\tau}(y) = \tau(y) = 1$ for all $y \in Y$. Therefore, for any $y \in Y$, $$0 = \tilde{f}(1) = \tilde{f}(yy^{-1}) = \tilde{f}(y)\tilde{\tau}(y^{-1}) + \tilde{\sigma}(y)\tilde{f}(y^{-1}) = \tilde{f}(y^{-1}).$$ Therefore, $\tilde{f}(y^{-1}) = 0$ for all $y \in Y$.

As in the proof of Theorem \ref{theorem 2.4(N)}, let $Y^{F(X)}$ be the normal closure of $Y$ in $F(X)$. Then $Y^{F(X)} = \langle w^{-1}yw \mid w \in F(X), y \in Y\rangle$ and it is the kernel of the unique onto group homomorphism $\phi: F(X) \rightarrow G$ which is identity on $X$, that is, $G \cong \frac{F(X)}{Y^{F(X)}}$. Let $\varepsilon \in \{1, -1\}$. Then using (\ref{eq 2.1}) and the fact that $\tilde{f}(y^{\varepsilon}) = 0$ for all $y \in Y$, we get that for any $w \in F(X)$, \begin{eqnarray*}\tilde{f}(w^{-1}y^{\varepsilon}w) & = & \tilde{f}(w^{-1})\tilde{\tau}(y^{\varepsilon}w) + \tilde{\sigma}(w^{-1})\tilde{f}(y^{\varepsilon}w) \\ & = & \tilde{f}(w^{-1})\tilde{\tau}(w) + \tilde{\sigma}(w^{-1})(\tilde{f}(y^{\varepsilon})\tilde{\tau}(w) + \tilde{\sigma}(y^{\varepsilon})\tilde{f}(w))\\ & = &  \tilde{f}(w^{-1})\tilde{\tau}(w) + \tilde{\sigma}(w^{-1})\tilde{f}(w) \\ & = & \tilde{f}(w^{-1}w) = \tilde{f}(1) = 0. \end{eqnarray*}

Now, let $v \in Y^{F(X)}$. Then $v = \prod_{i=1}^{m}w_{i}^{-1}y_{i}^{\varepsilon}w_{i}$ for some $m \in \mathbb{N}$, $w_{i} \in F(X)$, $y_{i} \in Y$ ($1 \leq i \leq m$). As shown above, the result holds for $m=1$. Assume that the result holds for $m=k-1$, that is, $\tilde{f}(\prod_{i=1}^{k-1}w_{i}^{-1}y_{i}^{\varepsilon}w_{i}) = 0$. Then using (\ref{eq 2.3}) and the induction hypothesis, we get that for $m=k$,
\begin{eqnarray*}
\tilde{f}(v) & = & \tilde{f}(\prod_{i=1}^{k}w_{i}^{-1}y_{i}^{\varepsilon}w_{i}) = \tilde{f}((\prod_{i=1}^{k-1}w_{i}^{-1}y_{i}^{\varepsilon}w_{i})(w_{k}^{-1}y_{k}^{\varepsilon}w_{k})) \\ & = & \tilde{f}(\prod_{i=1}^{k-1}w_{i}^{-1}y_{i}^{\varepsilon}w_{i})\tilde{\tau}(w_{k}^{-1}y_{k}^{\varepsilon}w_{k}) + \tilde{\sigma}(\prod_{i=1}^{k-1}w_{i}^{-1}y_{i}^{\varepsilon}w_{i}) \tilde{f}(w_{k}^{-1}y_{k}^{\varepsilon}w_{k}) \\ & = & 0\tilde{\tau}(w_{k}^{-1}y_{k}^{\varepsilon}w_{k}) + \tilde{\sigma}(\prod_{i=1}^{k-1}w_{i}^{-1}y_{i}^{\varepsilon}w_{i})0 = 0.
\end{eqnarray*}
Therefore, $\tilde{f}(v) = 0$ for all $v \in Y^{F(X)}$. Also, $\tilde{\sigma}(v) = \tilde{\tau}(v) = 1$ for all $v \in Y^{F(X)}$. Therefore, for any $w \in F(X)$ and $v \in Y^{F(X)}$, $$\tilde{f}(wv) = \tilde{f}(w)\tilde{\tau}(v) + \tilde{\sigma}(w)\tilde{f}(v) = \tilde{f}(w)1 + \tilde{\sigma}(w)0 = \tilde{f}(w).$$

Let $g \in G$ and $g = \prod_{i=1}^{m}x_{i}$ for some $m \in \mathbb{N}$, $x_{i} \in X \cup X^{-1}$ ($1 \leq i \leq m$). So $g$ is a word on $X$ so that $g \in F(X)$. In fact, $g$ is the reduced word in $F(X)$ with $\phi(g) = g$. If $w \in F(X)$ such that $\phi(w) = g$, then $w$ is either equivalent to the reduced word $g$ or the expression for $w$ (as a product of elements from $X \cup X^{-1}$) contains elements from $Y^{F(X)}$. In any case, $\tilde{\sigma}(w) = \sigma(g)$ and $\tilde{\tau}(w) = \tau(g)$.

Now, define $D:G \rightarrow RG$ by $$D(g) = \tilde{f}(w)$$ where for $g \in G$, $w \in F(X)$ is such that $\phi(w) = g$. Such a $w$ exists as $\phi: F(X) \rightarrow G$ is an onto group homomorphism, or as explained in the above paragraph, we can take $w=g$. The map $D$ is well-defined because as shown already, $\tilde{f}:F(X) \rightarrow RG$ is a well-defined map with $\tilde{f}(wv) = \tilde{f}(w)$ for all $v \in Y^{F(X)}$.

Now, let $g, h \in G$. Then there exists some $u, w \in F(X)$ such that $$\phi(u) = g ~~~~~~ \text{and} ~~~~~~ \phi(w) = h.$$ Then $$\tilde{\sigma}(u) = \sigma(g), ~~~~ \tilde{\tau}(w) = \tau(h) ~~~~ \text{and} ~~~~ \phi(uw) = \phi(u)\phi(w) = gh.$$ Therefore, by the definition of $D$, $$D(g) = \tilde{f}(u), ~~~~ D(h) = \tilde{f}(w) ~~~~ \text{and} ~~~~ D(gh) = \tilde{f}(uw).$$ Finally using (\ref{eq 2.1}), we get that
$$D(gh) = \tilde{f}(uw) = \tilde{f}(u)\tilde{\tau}(w) + \tilde{\sigma}(u)\tilde{f}(w) = D(g)\tau(h) + \sigma(g)D(h).$$
Therefore, $D(gh) = D(g)\tau(h) + \sigma(g)D(h)$ for all $g, h \in G$.

Now, extend $D$ $R$-linearly to the whole of $RG$. We denote the extended map also by $D$. So if $\lambda \in R$, and $\alpha = \sum_{g \in G}\lambda_{g}g$, $\beta = \sum_{h \in G}\mu_{h}h \in RG$ for some $\lambda_{g}, \mu_{h} \in R$ ($g, h \in G$), then $D(\alpha + \beta) = D(\alpha) + D(\beta)$ and $D(\lambda \alpha) = \lambda D(\alpha)$ as $D$ is an $R$-linear map. Also, \begin{eqnarray*}
D(\alpha \beta) & = & D((\sum_{g \in G}\lambda_{g}g)(\sum_{g \in G}\mu_{h}h)) \\ & = & D(\sum_{g,h \in G}\lambda_{g} \mu_{h}(gh)) = \sum_{g,h \in G}\lambda_{g} \mu_{h}D(gh) \\ & = & \sum_{g,h \in G}\lambda_{g} \mu_{h}(D(g)\tau(h) + \sigma(g)D(h)) \\ & = & \sum_{g,h \in G}\lambda_{g} \mu_{h}D(g)\tau(h) + \sum_{g,h \in G}\lambda_{g} \mu_{h}\sigma(g)D(h) \\ & = & D(\sum_{g \in G} \lambda_{g}g) \tau(\sum_{h \in G} \mu_{h}h) + \sigma(\sum_{g \in G} \lambda_{g}g) D(\sum_{h \in G} \mu_{h}h) \\ & = & D(\alpha)\tau(\beta) + \sigma(\alpha)D(\beta).
\end{eqnarray*}
Therefore, $D$ is a $(\sigma, \tau)$-derivation of $RG$ that extends $f:X \rightarrow RG$ to the whole of $RG$.

Now, it remains to show the uniqueness of $D$. If possible, suppose that $\tilde{D}: RG \rightarrow RG$ with $\tilde{D} \neq D$ is another $(\sigma, \tau)$-derivation of $RG$ that extends $f:X \rightarrow RG$. Then there exists some $g_{0} \in G$ such that $\tilde{D}(g_{0}) \neq D(g_{0})$ and $\tilde{D}(g) = D(g)$ for all $g \in G$ such that the length of $g$ is strictly less than that of $g_{0}$.

Note that $D(1) = 0 = \tilde{D}(1)$. Also, $\tilde{D}(x) = f(x) = D(x)$ for all $x \in X$. Now, for any $x \in X$, $$0 = \tilde{D}(1) = \tilde{D}(xx^{-1}) = \tilde{D}(x)\tau(x^{-1}) + \sigma(x)\tilde{D}(x^{-1}) = f(x)\tau(x^{-1}) + \sigma(x)\tilde{D}(x^{-1})$$ so that $\tilde{D}(x^{-1}) = -\sigma(x^{-1})f(x)\tau(x^{-1}) = \tilde{f}(x^{-1}) = D(x^{-1})$.

We can write $g_{0} \in G$ as $g_{0} = x_{0}y_{0}$ for some $x_{0}, y_{0} \in G$ such that the lengths of $x_{0}$ and $y_{0}$ are strictly less than that of $g_{0}$. Then in view of our assumption, $\tilde{D}(x_{0}) = D(x_{0})$ and $\tilde{D}(y_{0}) = D(y_{0})$. Therefore, \begin{eqnarray*}
\tilde{D}(g_{0}) & = & \tilde{D}(x_{0}y_{0}) \\ & = & \tilde{D}(x_{0})\tau(y_{0}) + \sigma(x_{0})\tilde{D}(y_{0}) \\ & = & D(x_{0})\tau(y_{0}) + \sigma(x_{0})D(y_{0}) \\ & = & D(x_{0}y_{0}) \\ & = & D(g_{0}).
\end{eqnarray*} But this is a contradiction to the fact that $\tilde{D}(g_{0}) \neq D(g_{0})$. Therefore, $D: RG \rightarrow RG$ is the unique $(\sigma, \tau)$-derivation that extends the given map $f:X \rightarrow RG$.

Hence proved.
\end{proof}

\begin{remark}\th\label{remark 2.5} 
The assumption that the ring $R$ is commutative is necessary in \th\ref{theorem 2.4}.

For example, let $R$ be a non-commutative ring and $a, b \in R$ such that $ab \neq ba$.

Let $X = \{x, y\}$ and $G = F(X)$ be the free group on $X$.

Define $f:X \rightarrow RG$ by $$f(x) = a ~~ \text{and} ~~ f(y) = b.$$

Define two group endomorphisms $\sigma$ and $\tau$ of $G$ by $$\sigma(x) = x, ~ \sigma(y) = y ~~ \text{and} ~~ \tau(x) = y, ~ \tau(y) = x.$$

If possible, suppose that $f$ can be extended to a $(\sigma, \tau)$-derivation $D$ of $RG$. Then \begin{eqnarray*}
D(xby) & = & D(bxy) = bD(xy) \\ & = & b(D(x)\tau(y) + \sigma(x)D(y)) \\ & = & b(f(x)\tau(y) + \sigma(x)f(y)) \\ & = & b(ax + xb) = bax + bxb \\ & = & (ba+b^{2})x
\end{eqnarray*} and 
\begin{eqnarray*}
D(x)\tau(by) + \sigma(x)D(by) & = & D(x)b\tau(y) + \sigma(x)bD(y) \\ & = & f(x)b\tau(y) + \sigma(x)bf(y) \\ & = & abx + xb^{2} \\ & = & (ab + b^{2})x.
\end{eqnarray*}

Since $ab \neq ba$, therefore, $D(xby) \neq D(x)\tau(by) + \sigma(x) D(by)$. This gives a contradiction.
\end{remark}

\subsection{Application: $\sigma$-Derivations of Commutative Group Algebras}\label{subsection 2.3}
In this subsection, we apply the results of Subsection \ref{subsection 2.2(N)} to classify the $\sigma$-derivations of commutative group algebras. In this subsection, we assume that $R$ is a commutative unital ring, $G$ is a group, and $\sigma$ is an $R$-algebra endomorphism of $RG$ that is an $R$-linear extension of a group endomorphism of $G$.

\begin{theorem}\th\label{theorem 2.6}
Let $H$ be a subgroup of $G$ such that $\sigma(H) = \{\sigma(h) \mid h \in H\}$ is a torsion central subgroup of $G$ and the order of every element of $\sigma(H)$ is invertible in $R$. Then every $\sigma$-derivation of $RG$ is an $RH$-derivation, that is, $D(RH) = \{0\}$ for all $D \in \mathcal{D}_{\sigma}(RG)$.
\end{theorem}
\begin{proof}
Let $D \in \mathcal{D}_{\sigma}(RG)$ and $h \in H$. The order of $\sigma(h)$ is finite, say, $r$, and is invertible in $R$. Since $\sigma(h)$ commutes with $D(h)$, therefore, by \th\ref{lemma 2.2} (v), $$0 = D(1) = D(h^{r}) = r(\sigma(h))^{r-1}D(h).$$
Since $r = |\sigma(h)|$ is invertible in $R$ and hence invertible in $RG$, we get that $D(h) = 0$ from above. Therefore, $D(h) = 0$ for all $h \in H$.

Now, let $\alpha \in RH$. Then $\alpha = \sum_{h \in H} \lambda_{h}h$ for some $\lambda_{h} \in R$ ($h \in H$). Since $D$ is $R$-linear, so $$D(\alpha) = D(\sum_{h \in H} \lambda_{h}h) = \sum_{h \in H} \lambda_{h} D(h) = \sum_{h \in H} \lambda_{h} 0 = 0.$$ Therefore, $D(RH) = \{0\}$.
\end{proof}

\begin{corollary}\th\label{corollary 2.7}
If $G$ is a finite abelian group and $\mathbb{F}$ is a field of characteristic $0$, then $\mathbb{F}G$ has no non-zero $\sigma$-derivations.
\end{corollary}
\begin{proof}
Follows by taking $H = G$ in \th\ref{theorem 2.6}.
\end{proof}

Note that Corollary \ref{corollary 2.7} is not a criterion. That is, if $\mathbb{F}$ is any field, $G$ is any group, and $\sigma$ is any endomorphism of $RG$, which is an $R$-linear extension of a group endomorphism of $G$, then it is not necessary that $G$ is a finite abelian group or $\mathbb{F}$ is a field of characteristic $0$. Below, we give some examples to illustrate this.

Examples \ref{example 2.4.1} and \ref{example 2.4.2} illustrate the situation when $\mathbb{F}G$ has no non-zero $\sigma$-derivations and $G$ is a finite abelian group, but field $\mathbb{F}$ is not of characteristic $0$.

A subgroup $H$ of a group is said to be $p$-regular for a rational prime $p$ if every element of $H$ has an order which is not a multiple of $p$.
\begin{example}\label{example 2.4.1}
Let $p$ and $q$ be two distinct rational primes. Let $G = C_{q^{t}} = \langle x \mid x^{q^{t}} = 1 \rangle$ be a cyclic group of order $q^{t}$ (for some $t \in \mathbb{N}$) and $\mathbb{F}$ be a field of characteristic $p$. Then $G$ is a $p$-regular finite abelian group. Further, let $\sigma$ be an $\mathbb{F}$-algebra endomorphism of $\mathbb{F}G$ which is an $\mathbb{F}$-linear extension of a group endomorphism of $G$.

Then by Theorem \ref{theorem 2.6}, every $\sigma$-derivation of $\mathbb{F}G$ is an $\mathbb{F}G$-derivation, that is, $D(\mathbb{F}G) = \{0\}$ for all $D \in \mathcal{D}_{\sigma}(RG)$. In other words, $\mathbb{F}G$ has no non-zero $\sigma$-derivations.

Note that, here $\mathbb{F}G$ has no non-zero $\sigma$-derivations and $G$ is a finite abelian group, but the field $\mathbb{F}$ is not of characteristic $0$; $\mathbb{F}$ is a field of characteristic $p$ (odd rational prime).
\end{example}

The example below is a more generalized version of Example \ref{example 2.4.1}.
\begin{example}\label{example 2.4.2}
Let $p, q_{1}, ..., q_{s}$ ($s \in \mathbb{N}$) be distinct odd rational primes and $t_{1}, ..., t_{s}$ be positive integers. For each $i \in \{1, ..., s\}$, let $C_{q^{t_{i}}} = \langle x_{i} \mid x_{i}^{t_{i}} = 1 \rangle$ denote a cyclic group of order $q^{t_{i}}$. Let $G = C_{q^{t_{1}}} \times ... \times C_{q^{t_{s}}}$ be a finite abelian group of order $q^{t_{1}}...q^{t_{s}}$. Then $G$ becomes a $p$-regular group. Further, let $\mathbb{F}$ be a field of characteristic $p$ and $\sigma$ be an $\mathbb{F}$-algebra endomorphism of $\mathbb{F}G$ which is an $\mathbb{F}$-linear extension of a group endomorphism of $G$.

Then by Theorem \ref{theorem 2.6}, every $\sigma$-derivation of $\mathbb{F}G$ is an $\mathbb{F}G$-derivation, that is, $D(\mathbb{F}G) = \{0\}$ for all $D \in \mathcal{D}_{\sigma}(RG)$. In other words, $\mathbb{F}G$ has no non-zero $\sigma$-derivations.

Again, note that $\mathbb{F}G$ has no non-zero $\sigma$-derivations and $G$ is a finite abelian group, but the field $\mathbb{F}$ is not of characteristic $0$; $\mathbb{F}$ is a field of characteristic $p$ (odd rational prime).
\end{example}

Examples \ref{example 2.4.3} and \ref{example 2.4.4} below illustrate the following two situations:
\begin{enumerate}
\item $\mathbb{F}G$ has no non-zero $\sigma$-derivations and $\mathbb{F}$ is a field of characteristic $0$, but the group $G$ is not finite abelian group.
\item $\mathbb{F}G$ has no non-zero $\sigma$-derivations, but neither the field $\mathbb{F}$ is of characteristic $0$ nor the group $G$ is finite abelian.
\end{enumerate}
\begin{example}\label{example 2.4.3}
Consider the dihedral group $D_{2n} = \langle a, b \mid a^{n} = 1 = b^{2}, (ab)^{2} = 1 \rangle$ of order $2n$ ($n \geq 3$) with $n$ odd. Let $\mathbb{F}$ be a field of characteristic $0$ or an odd rational prime $p$.

The map $\sigma_{-1}:D_{2n} \rightarrow D_{2n}$ given by $\sigma_{-1}(a) = 1$ and $\sigma_{-1}(b) = 1$ is the trivial group endomorphism of $D_{2n}$. Extend $\sigma_{-1}$ $\mathbb{F}$-linearly to an $\mathbb{F}$-algebra endomorphism of $\mathbb{F}D_{2n}$.

Since $\sigma_{-1}(b) = \sigma_{-1}(ab) = 1 \in D_{2n} \cap Z(\mathbb{F} D_{2n})$, so by the note following Definition \ref{definition 4.1}, $\bar{C}(\sigma_{-1}(b)) = \bar{C}(\sigma_{-1}(ab)) = \{0\}$, where for any $\beta \in \mathbb{F}D_{2n}$, $\bar{C}(\beta) = \{\alpha \in \mathbb{F}D_{2n} \mid \alpha \beta = - \beta \alpha\}$ denotes the anti-centralizer of $\beta$ in $\mathbb{F}D_{2n}$ (see Definition \ref{definition 4.1}).

First, we determine $\mathcal{D}_{\sigma_{-1}}(\mathbb{F}D_{2n})$. For this, let $f:X = \{a, b\} \rightarrow \mathbb{F}D_{2n}$ be a map that can be extended to a $\sigma_{-1}$-derivation of $\mathbb{F}D_{2n}$. Then by Theorem \ref{theorem 2.4(N)}, this happens if and only if $$\tilde{f}(a^{n}) = 0, ~~~~ \tilde{f}(b^{2}) = 0 ~~~~ \text{and} ~~~~ \tilde{f}((ab)^{2}) = 0,$$ where $\tilde{f}:F(X) \rightarrow \mathbb{F}D_{2n}$ is the unique extension of $f$ defined in (\ref{eq 2.2(N)}) and (\ref{eq 2.3(N)}) and satisfying (\ref{eq 2.1(N)}), and $F(X)$ is the free group on $X$.

As explained in the paragraph after the proof of Lemma \ref{lemma 4.2}, $\tilde{f}(b^{2}) = 0$ if and only if $f(b) \in \bar{C}(\sigma_{-1}(b))$, and $\tilde{f}((ab)^{2}) = 0$ if and only if $\tilde{f}(ab) \in \bar{C}(\sigma_{-1}(ab))$. But since $\bar{C}(\sigma_{-1}(b)) = \bar{C}(\sigma_{-1}(ab)) = \{0\}$, therefore, $f(b) =  \tilde{f}(ab) = 0$.

Further, by (\ref{eq 4.2}), $f(a) = \tilde{f}(ab) \sigma_{-1}(b) - \sigma_{-1}(a) f(b) \sigma_{-1}(b)$. But since $f(b) = 0 = \tilde{f}(ab)$, therefore, $f(a) = 0$.

Since $f(a) = 0 = f(b)$, therefore, $\mathbb{F}D_{2n}$ has no non-zero $\sigma_{-1}$-derivations, that is, \\ $\mathcal{D}_{\sigma_{-1}}(\mathbb{F}D_{2n}) = \{0\}$. Therefore, $\mathbb{F}D_{2n}$ has no non-zero $\sigma_{-1}$-derivations.

But note that here the group $D_{2n}$ is non-abelian. Also, the field $\mathbb{F}$ can be of characteristic either $0$ or an odd rational prime $p$.
\end{example}

\begin{example}\label{example 2.4.4}
Now, consider the dihedral group $D_{2n}$ with $n$ even and a field $\mathbb{F}$ of characteristic $0$ or an odd rational prime $p$.

As stated in Section \ref{section 4}, the map $\sigma_{2}:D_{2n} \rightarrow D_{2n}$ defined by $\sigma_{2} = a^{s}$ ($s = 0, \frac{n}{2}$) and \\ $\sigma_{2}(b) = a^{t}$ ($t = 0, \frac{n}{2}$) is a group endomorphism of $D_{2n}$. Extend $\sigma_{2}$ $\mathbb{F}$-linearly to an $\mathbb{F}$-algebra endomorphism of $\mathbb{F}D_{2n}$.

Again, since $\sigma_{2}(b) = \sigma_{2}(ab) \in D_{2n} \cap Z(\mathbb{F} D_{2n})$, so by the note following Definition \ref{definition 4.1}, $\bar{C}(\sigma_{2}(b)) = \bar{C}(\sigma_{2}(ab)) = \{0\}$. Therefore, as shown in the above Example \ref{example 2.4.3}, it can be shown on the similar lines that $\mathcal{D}_{\sigma_{2}}(\mathbb{F}D_{2n}) = \{0\}$. Therefore, $\mathbb{F}D_{2n}$ has no non-zero $\sigma_{2}$-derivations.

Again, note that here the group $D_{2n}$ is non-abelian, and the field $\mathbb{F}$ can be of characteristic either $0$ or an odd rational prime $p$.
\end{example}

The following corollary is an immediate consequence of \th\ref{theorem 2.6}, and generalizes Examples \ref{example 2.4.1} and \ref{example 2.4.2}.

\begin{corollary}\th\label{corollary 2.8}
Let $G$ be a finite abelian group and $H$ be a $p$-regular subgroup of $G$, where $p$ is a rational prime. Let $\mathbb{F}$ be a field of characteristic $p$. Then every $\sigma$-derivation of $\mathbb{F}G$ is an $\mathbb{F}H$-derivation. 
\end{corollary}

By the Fundamental Theorem of Finite Abelian Groups, every finite Abelian group $G$ is isomorphic to a direct product of a $p$-regular group and a $p$-group, where $p$ is a rational prime, that is, $G \cong H \times K$, where $H$ is a $p$-regular group and $K$ is a $p$-group. The group $K$ then has the presentation $K = \langle X \mid Y\rangle$, where $X = \{x_{1}, ..., x_{r}\}$ is a set of generators of $K$ for some $r \in \mathbb{N}$ and $Y = \{x_{i}^{p^{r_{i}}}, x_{i}^{-1}x_{j}^{-1}x_{i}x_{j},  \hspace{0.1cm} \text{for some} \hspace{0.1cm} r_{i} \in \mathbb{N}, \hspace{0.1cm} i \in \{1, ..., r\}\}$. Now, we have the following theorem, which classifies all $\sigma$-derivations of $\mathbb{F}G$, where $\mathbb{F}$ is a field of characteristic $p$ and $G$ is a finite abelian group in the above form.

\begin{theorem}\th\label{theorem 2.9}
The dimension of $\mathcal{D}_{\sigma}(\mathbb{F}G)$ is $r|G|$ and a basis is $$\mathcal{B} = \{gD_{i} \mid g \in G, 1 \leq i \leq r\},$$ where for each $i \in \{1, ..., r\}$, $D_{i}:\mathbb{F}G \rightarrow \mathbb{F}G$ is a $\sigma$-derivation of $\mathbb{F}G$ defined by $D_{i}(x_{j}) = \delta_{ij}$, the Kronecker delta.
\end{theorem}
\begin{proof}
By \th\ref{corollary 2.8}, every $\sigma$-derivation of $\mathbb{F}G$ is an $\mathbb{F}H$-derivation.

For each $i \in \{1, ..., r\}$, define $f_{i}:X \rightarrow \mathbb{F}G$ by $$f_{i}(x_{j}) = \delta_{ij}, ~~ \text{that is}, ~~ f_{i}(x_{j}) = \begin{cases}
1 & \text{if $j = i$} \\
0 & \text{if $j \neq i$}
\end{cases}.$$ Then by \th\ref{lemma 2.3(N)}, each $f_{i}$ can be extended uniquely to a map $\tilde{f}_{i}:F(X) \rightarrow \mathbb{F}G$ that satisfies (\ref{eq 2.1(N)}). Let $i, k, l \in \{1, ..., r\}$. Then by (\ref{eq 2.3(N)}), $$\tilde{f}_{i}(x_{k}^{p^{r_{k}}}) = \sum_{t=1}^{p^{r_{k}}}\left(\left(\prod_{j=1}^{t-1}\sigma(x_{k})\right)\tilde{f}_{i}(x_{k})\left(\prod_{j=t+1}^{p^{r_{k}}}\sigma(x_{k})\right)\right) = p^{r_{k}}((\sigma(x_{k}))^{p^{r_{k}}-1}f_{i}(x_{k}) = 0,$$ and
\begin{equation*}
\begin{aligned}
\tilde{f}_{i}(x_{k}^{-1}x_{l}^{-1}x_{k}x_{l}) & = \tilde{f}_{i}(x_{k}^{-1})\tilde{\sigma}(x_{l}^{-1}x_{k}x_{l}) + \tilde{\sigma}(x_{k}^{-1})\tilde{f}_{i}(x_{l}^{-1}x_{k}x_{l})
\\ & = \tilde{f}_{i}(x_{k}^{-1})\sigma(x_{l}^{-1}x_{k}x_{l}) + \sigma(x_{k}^{-1})(\tilde{f}_{i}(x_{l}^{-1})\tilde{\sigma}(x_{k}x_{l}) + \tilde{\sigma}(x_{l}^{-1})\tilde{f}_{i}(x_{k}x_{l}))
\\ & = -\sigma(x_{k}^{-1})f_{i}(x_{k})\sigma(x_{k}^{-1})\sigma(x_{l}^{-1}x_{k}x_{l}) + \sigma(x_{k}^{-1})\tilde{f}_{i}(x_{l}^{-1})\sigma(x_{k}x_{l}) \\ &\quad + \sigma(x_{k}^{-1})\sigma(x_{l}^{-1})\tilde{f}_{i}(x_{k}x_{l})
\\ & = -\sigma(x_{k}^{-1})f_{i}(x_{k}) - \sigma(x_{l}^{-1})f_{i}(x_{l})\sigma(x_{l}^{-1}) \sigma(x_{l}) + \sigma(x_{k}^{-1})\sigma(x_{l}^{-1})(\tilde{f}_{i}(x_{k})\tilde{\sigma}(x_{l}) \\ &\quad + \tilde{\sigma}(x_{k})\tilde{f}_{i}(x_{l}))
\\ & = -\sigma(x_{k}^{-1})f_{i}(x_{k}) - \sigma(x_{l}^{-1})f_{i}(x_{l}) + \sigma(x_{k}^{-1})f_{i}(x_{k}) + \sigma(x_{l}^{-1})f_{i}(x_{l}) = 0.
\end{aligned}
\end{equation*}
So we have proved that $\tilde{f}_{i}(y) = 0$ for all $y \in Y$ and $i \in \{1, ..., r\}$. Therefore, by \th\ref{theorem 2.4(N)}, for each $i \in \{1, ..., r\}$, the map $f_{i}:X \rightarrow \mathbb{F}G$ can be extended uniquely to a $\sigma$-derivation $D_{i}$ of $\mathbb{F}G$.

Put $\mathcal{B} = \{gD_{i} \mid g \in G, 1 \leq i \leq r\}$. We prove that  $\mathcal{B}$ is a basis of the $\mathbb{F}$-vector space $\mathcal{D}_{\sigma}(\mathbb{F}G)$.

Since for each $i \in \{1, ..., r\}$, $D_{i}$ is an extension of $f_{i}$, therefore, $D_{i}(x_{j}) = f_{i}(x_{j}) = \delta_{ij}$ for all $1 \leq i, j \leq r$.

Now, let $D \in \mathcal{D}_{\sigma}(\mathbb{F}G)$ and $i, j \in \{1, ..., r\}$. Let $D(x_{j}) = \sum_{g \in G}\lambda_{jg} g$ for some $\lambda_{jg} \in \mathbb{F}$ ($g \in G$). Then $$D(x_{j}) = \sum_{g \in G}\lambda_{jg} g1 = \sum_{g \in G}\lambda_{jg} gD_{j}(x_{j}) = \left(\sum_{g \in G}\lambda_{jg} (gD_{j})\right)(x_{j}).$$

Further, $$0 = D_{i}(1) = D_{i}(x_{j}x_{j}^{-1}) = D_{i}(x_{j}) \sigma(x_{j}^{-1}) + \sigma(x_{j}) D_{i}(x_{j}^{-1})$$ so that $$D_{i}(x_{j}^{-1}) = - \sigma(x_{j}^{-1})D_{i}(x_{j})\sigma(x_{j}^{-1}) = - \sigma(x_{j}^{-2}) D_{i}(x_{j}) = - \sigma(x_{j}^{-2}) \delta_{ij}.$$

Similarly, $D(x_{j}^{-1}) = - \sigma(x_{j}^{-2}) D(x_{j})$ so that $$D(x_{j}^{-1}) = \sum_{g \in G}\lambda_{jg} g (- \sigma(x_{j}^{-2})) = \sum_{g \in G}\lambda_{jg} g D_{j}(x_{j}^{-1}) = (\sum_{g \in G}\lambda_{jg} (g D_{j}))(x_{j}^{-1}).$$

Now, since $X$ is a generating set of the group $K$, therefore, for any $g_{0} \in G$, as $D$ is an $\mathbb{F}H$-derivation, $D(g_{0}) = D(x_{i_{1}}^{\varepsilon_{1}} x_{i_{2}}^{\varepsilon_{2}} ... x_{i_{s}}^{\varepsilon_{s}})$ for some $s \in \mathbb{N}$, $i_{l} \in \{1, ..., r\}$, $x_{i_{l}} \in X$ and $\varepsilon_{l} \in \{-1, 1\}$ ($1 \leq l \leq s$). So now, it immediately follows that $$D = \sum_{i=1}^{r} \sum_{g \in G} \lambda_{ig} (gD_{i}).$$ Therefore, the set $\mathcal{B}$ spans $\mathcal{D}_{\sigma}(\mathbb{F}G)$ over $\mathbb{F}$.

Now, assume that $$\sum_{i=1}^{r} \sum_{g \in G}c_{ig}(gD_{i}) = 0$$ for some $c_{ig} \in \mathbb{F}$ ($1 \leq i \leq r$ and $g \in G$). Then $$\sum_{i=1}^{r} \sum_{g \in G}c_{ig}(gD_{i})(x_{j}) = 0$$ for all $1 \leq j \leq r$. Then using the fact that $D_{i}(x_{j}) = \delta_{ij}$ ($1 \leq i, j \leq r$), we get that $$\sum_{g \in G}c_{jg}g = 0$$ for all $1 \leq j \leq r$. This and the fact that $G$ is an $\mathbb{F}$-linearly independent subset of $\mathbb{F}G$ together imply that for all $1 \leq j \leq r$, $c_{jg} = 0$ for all $g \in G$.
\end{proof}

Finally, we compare our results in this section with some of the earlier known results on twisted derivations of commutative algebras. Before this, there have been several attempts to describe the $R$-module of ordinary or twisted derivations of an $R$-algebra $\mathcal{A}$ ($R$ a ring). For example, in {\cite[Subsection 3.1]{Creedon2019}}, the authors studied ordinary derivations of commutative group algebras. Subsection \ref{subsection 2.3} of our article generalizes the results of {\cite[Subsection 3.1]{Creedon2019}} to $\sigma$-derivations with a more elaborate and complete justification. In \cite{Chaudhuri}, the author studied $(\sigma, \tau)$-derivations of commutative algebras over a commutative unital ring. In {\cite[Theorem 3.1]{Chaudhuri}}, the author studied the universal mapping properties of $(\sigma, \tau)$-derivations of commutative algebras over a commutative ring with unity. The author also studied the $(\sigma, \tau)$-derivations of the ring of algebraic integers, $O_{K} = \mathbb{Z}[\sqrt{d}]$ ($d$ a square-free integer) of a quadratic number field $K$ by considering $O_{K}$ as a (commutative) $\mathbb{Z}$-algebra. More precisely, in {\cite[Theorem 4.2]{Chaudhuri}}, the author showed that if $\sigma, \tau$ are two different ring endomorphisms of $O_{K}$, then any $\mathbb{Z}$-linear map $D:O_{K} \rightarrow O_{K}$ with $D(1) = 0$ is a $(\sigma, \tau)$-derivation of $O_{K}$, and also gave a sufficient condition under which a $(\sigma, \tau)$-derivation of $O_{K}$ is inner. These rings $O_{K}$ provide several examples of non-UFDs.  In \cite{Spiegel1994}, the author described the derivations of $\mathbb{Z}G$ for any abelian group $G$. In {\cite[Theorem 4.6]{Chaudhuri2021}}, the author gives an element in $RG$ (a commutative group algebra of an abelian group $G$ over a commutative unital ring $R$ with certain conditions) that makes a $(\sigma, \tau)$-derivation of $RG$ inner under the condition that there exists some $b \in RG$ with $(\tau - \sigma)(b)$ invertible in $RG$, where $\sigma, \tau$ are $R$-algebra endomorphisms of $RG$. In \cite{richard2008quasi}, the authors proved that a certain bracket on the $\sigma$-derivations of a commutative algebra preserves inner derivations (see Proposition 2.3.1), and based on this obtained structural results providing new insights into $\sigma$-derivations on Laurent polynomials in one variable, $\mathbb{C}[t, t^{-1}]$ which is a UFD. In \cite{Hartwig2006}, the authors developed an approach to deformations of the Witt and Virasoro algebras based on twisted derivations. The authors proved that if $\mathcal{A}$ (an algebra over $\mathbb{C}$) is a unique factorization domain (UFD), then the $\mathcal{A}$-module $\mathcal{D}_{(\sigma, id_{\mathcal{A}})}(\mathcal{A})$ of all $(\sigma, id_{\mathcal{A}})$-derivations of $\mathcal{A}$ is a free $\mathcal{A}$-module of rank one. Here $\sigma$ and $\tau$ are two different algebra endomorphisms of $\mathcal{A}$ and the generator is a $(\sigma, \tau)$-derivation given by $\Delta = \frac{\tau - \sigma}{g}$, where $g = \text{gcd}((\tau - \sigma)(\mathcal{A}))$ (see Theorem 4). The authors also defined generalized Witt algebra using a twisted derivation. In \cite{richard2008quasi}, the authors further noted that a $(\sigma, id_{\mathcal{A}})$-derivation $a \Delta$ is inner if and only if $g$ divides $a$. More precisely, the $\mathcal{A}$-submodule of all inner $(\sigma, id_{\mathcal{A}})$-derivations is of rank one with generator $g \mathcal{A} \Delta$ (see Proposition 2.4.1). Then the authors obtained some more deep and precise results for the algebra $\mathcal{A} = \mathbb{C}[t, t^{-1}]$, which is also a UFD. Note that, in this article, we have considered our set of $(\sigma, \tau)$-derivations of the group ring $RG$ ($R$ a commutative unital ring, $G$ a group) as an $R$-module rather than an $RG$-module.

There are several papers, especially by A. A. Arutyunov (as mentioned in the Introduction Section \ref{section 1}), in which the description of the derivations of group algebras is given. In \cite{AleksandrAlekseev2020}, the authors considered the $(\sigma, \tau)$ and $(\sigma, id)$-derivations of the group algebra $\mathbb{C}G$ of a discrete countable group $G$, where $\sigma, \tau$ are $\mathbb{C}$-algebra endomorphisms of $\mathbb{C}G$ which are $\mathbb{C}$-linear extensions of the group endomorphisms of $G$. In {\cite[Theorem 4]{AleksandrAlekseev2020}}, the authors proved a decomposition theorem expressing the $\mathbb{C}$-vector space of all $(\sigma, \tau)$-derivations of $\mathbb{C}G$ of a finitely generated $(\sigma, \tau)$-FC group, where a $(\sigma, \tau)$-FC group is a group in which every $(\sigma, \tau)$-conjugacy class is of finite size.

In \cite{Arutyunov2021}, the authors studied the ordinary derivations of finite and FC groups using topological and character techniques. They obtained decomposition theorems for derivations of the group ring $AG$ of a finite group or FC group (a group in which every conjugacy class is of finite size) $G$ over a commutative unital ring $A$ (see {\cite[Theorems 3.2, 4.1]{Arutyunov2021}}). More precisely, the authors proved that for a finitely generated FC-group $G$, the $A$-module of all derivations of $AG$ is isomorphic to the direct sum of its $A$-submodule of all inner derivations of $AG$ and $\bigoplus_{[u] \in G^{G}} \text{Hom}_{Ab}(Z(u),A),$ where $\text{Hom}_{Ab}(Z(u),A)$ is the set of additive homomorphisms from the centralizer $Z(u)$ of a fixed element $u \in G$ to the ring $A$ and $G^{G}$ is the set of all possible distinct conjugacy classes $[u]$ ($u \in G$) in $G$ (see {\cite[Theorems 3.2]{Arutyunov2021}}). Furthermore, if $G$ is a finite group, then the $A$-module of all outer derivations of $AG$ is isomorphic to $\bigoplus_{[u] \in G^{G}} \text{Hom}_{Ab}(Z(u),A)$ (see {\cite[Theorems 4.1]{Arutyunov2021}}).

In \cite{A.A.Arutyunov2020} also, a description of derivations is given in terms of characters of the groupoid of the adjoint action of the group. A method of describing the space of all outer derivations of the complex group algebra $\mathbb{C}G$ of a finitely presentable discrete group $G$ is given. In \cite{Arutyunov2023}, the author studied derivations of a complex group algebra $\mathbb{C}G$ of a finitely generated group $G$. The author constructed the ideals of inner and quasi-inner derivations and established a connection between derivations and characters on the groupoid of the adjoint action (see Proposition 2.11 and Corollary 3.2). The author described the space of outer and quasi-outer derivations of $\mathbb{C}G$ using the methods of combinatorial group theory, in particular, the number of ends of the group $G$ and the number of ends of the connected components of a conjugacy diagram (see, for example, Theorem 3.6, Corollary 3.7, Theorem 4.2, Corollary 4.3, Corollary 4.5, Corollary 4.7, Proposition 5.1). In \cite{Arutyunov2020a}, the authors described the space of all inner and outer derivations of a group algebra of a finitely presented discrete group in terms of the character spaces of the 2-groupoid of the adjoint action of the group. In \cite{Arutyunov2020b}, the author established a connection between derivations of complex group algebras and the theory that studies the ends of topological spaces. In {\cite[Theorems 5 and 6]{Arutyunov2020b}}, the author described the space of outer and quasi-outer derivations of a complex group algebra.

In \cite{Arutyunov2020}, the author studied derivation spaces in the group algebra $\mathbb{C}G$ of a generally infinite non-commutative discrete group $G$ in terms of characters on a groupoid associated with the group. In {\cite[Theorem 2]{Arutyunov2020}}, the author constructed a subalgebra of non-inner derivations of $\mathbb{C}G$, which can be embedded in the algebra of all outer derivations. In {\cite[Theorem 3]{Arutyunov2020}}, the author obtained the necessary conditions under which a character defines a derivation. In {\cite[Proposition 8]{Arutyunov2020}}, the author described the space of all derivations of a complex group algebra of a free abelian group of finite rank. In {\cite[Theorem 4 and Theorem 5]{Arutyunov2020}}, the author gave an explicit description of the space of all derivations and outer derivations of the complex group algebra of a rank two nilpotent group.
In \cite{arutyunov2019smooth}, the authors described the space of all outer derivations of the group algebra $\mathbb{C}G$ of a finitely presented discrete group in terms of the Cayley complex of the groupoid of the adjoint action of the group (see Theorem 3, Theorem 4, Corollary 1, Corollary 2). They showed that the space of outer derivations is isomorphic to the one-dimensional compactly supported group of the Cayley complex over $\mathbb{C}$. In \cite{mishchenko2020description}, the author studied the derivations of a complex group algebra $\mathbb{C}G$ of a finitely generated discrete group. In {\cite[Section 3]{mishchenko2020description}}, the author described the derivations of $\mathbb{C}G$ using characters on the groupoid of the adjoint action of the group (see Theorems 2 and 3). In {\cite[Section 4]{mishchenko2020description}}, the author described the space of outer derivations by establishing the isomorphism of this space with the one-dimensional cohomology of the Cayley complex of the groupoid (see Corollary 1). In {\cite[Section 5]{arutyunov2019smooth}} and {\cite[Section 5]{mishchenko2020description}}, the author described the space of ordinary derivations of a complex group algebra of the additive group $\mathbb{Z}$, a finitely generated free abelian group and a finitely generated free group.

\section{Inner $(\sigma, \tau)$-Derivations of Group Rings}\label{section 3}
Let $G$ be a group and $R$ be a commutative unital ring. Unless stated otherwise, in this section, we always assume that $\sigma$ and $\tau$ are any two unital $R$-algebra endomorphisms of the group algebra $RG$, which are $R$-linear extensions of the group endomorphisms of $G$. This section classifies all inner $(\sigma, \tau)$-derivations of $RG$ by giving the rank and a basis of the $R$-module of all inner $(\sigma, \tau)$-derivations of $RG$. Also, we prove that every $(\sigma, \tau)$-derivation of $RG$ is an inner $(\sigma, \tau)$-derivation if the order of $G$ is invertible in $R$, where in this result, $\sigma$ and $\tau$ are any two unital $R$-algebra endomorphisms of $RG$ and need not necessarily be group endomorphisms of $G$. Consequently, we explicitly classify all $(\sigma, \tau)$-derivations of $RG$ when $G$ is a finite group whose order is invertible in $R$. To prove our results, we need the concept of doubly-twisted conjugacy classes introduced in Subsection \ref{subsection 3.1}.

\subsection{Doubly-Twisted Conjugacy Classes}\label{subsection 3.1}
\begin{definition}\th\label{definition 3.1}
For $x, y \in G$, $x$ is said to be $(\sigma, \tau)$-conjugate to $y$ in $G$ if $$y = \sigma(g)x(\tau(g))^{-1}$$ for some $g \in G$. The set of all $(\sigma, \tau)$-conjugates to $x$ is denoted by $$x_{(\sigma, \tau)}^{G} = \{\sigma(g)x (\tau(g))^{-1} \mid g \in G\}$$ and is called the $(\sigma, \tau)$-conjugacy class of $x$ in $G$.
\end{definition}

In particular, if $\tau = \sigma$, then we simply refer the term $(\sigma, \tau)$-conjugation as $\sigma$-conjugation. Note that $(\sigma, \tau)$-conjugacy is an equivalence relation. Below, we state some results that are analogous to those in classic group theory.

\begin{lemma}\th\label{lemma 3.2}
Any two $(\sigma, \tau)$-conjugacy classes of $G$ are either equal or disjoint, that is, if $x, y \in G$, then either $x_{(\sigma, \tau)}^{G} = y_{(\sigma, \tau)}^{G}$ or $x_{(\sigma, \tau)}^{G} \cap y_{(\sigma, \tau)}^{G} = \emptyset$.
\end{lemma}

Note that if $1$ is the identity element in $G$, then for any $x \in G$, $x = \sigma(1)x(\tau(1))^{-1}$ so that $x \in x_{(\sigma, \tau)}^{G}$. Therefore, $G$ is a union of its all $(\sigma, \tau)$-conjugacy classes.

\begin{corollary}\th\label{corollary 3.3}
$G$ is union of its all $(\sigma, \tau)$-conjugacy classes and distinct $(\sigma, \tau)$-conjugacy classes are pairwise disjoint.
\end{corollary}

\begin{definition}\th\label{definition 3.4}
If $G = \bigcup_{i=1}^{r} (x_{i})_{(\sigma, \tau)}^{G}$ such that the $(\sigma, \tau)$-conjugacy classes $(x_{i})_{(\sigma, \tau)}^{G}$'s \\ ($1 \leq i \leq r$) are distinct, then $x_{1}, ..., x_{r}$ are called representatives of the $(\sigma, \tau)$-conjugacy classes $(x_{1})_{(\sigma, \tau)}^{G}$, ..., $(x_{r})_{(\sigma, \tau)}^{G}$ of $G$.
\end{definition}

We establish more results on the concept of $(\sigma, \tau)$-conjugation.

\begin{lemma}\th\label{lemma 3.5}
Let $x, y \in G$ be such that $x$ is $\sigma$-conjugate to $y$. Then for any $n \in \mathbb{Z}$, $x^{n}$ is $\sigma$-conjugate to $y^{n}$. Also, $x$ and $y$ have the same order.
\end{lemma}

\begin{definition}\th\label{definition 3.6}
Let $x \in G$, then the $(\sigma, \tau)$-centraliser of $x$ in $G$ is defined to be the set $$C_{(\sigma, \tau)}(x) = \{g \in G \mid x\tau(g) = \sigma(g)x\}.$$
\end{definition}

Note that $C_{(\sigma, \tau)}(x)$ is a monoid. In particular, if $G$ is finite, then $C_{(\sigma, \tau)}(x)$ is a subgroup of $G$. In general, $x$ need not be an element of $C_{(\sigma, \tau)}(x)$.

\begin{lemma}\th\label{lemma 3.7}
Let $G$ be a finite group and $x \in G$. Then $$|x_{(\sigma, \tau)}^{G}| = \bigg\vert \frac{G}{C_{(\sigma, \tau)}(x)} \bigg\vert = \frac{|G|}{|C_{(\sigma, \tau)}(x)|}.$$ In particular, $|x_{(\sigma, \tau)}^{G}|$ divides $|G|$.
\end{lemma}

\begin{definition}\th\label{definition 3.8}
The set $$Z_{(\sigma, \tau)}(G) = \{z \in G \mid z \tau(g) = \sigma(g) z, \hspace{0.1cm} \forall \hspace{0.1cm} g \in G\}$$ is called the $(\sigma, \tau)$-center of $G$.
\end{definition}

Observe that $Z_{(\sigma, \tau)}(G) = \{z \in G \mid C_{(\sigma, \tau)}(z) = G\}$. If $\tau = \sigma$, then we denote the $(\sigma, \tau)$-center of $G$ simply by $Z_{\sigma}(G)$ and it is a monoid. If $\tau = \sigma = id_{G}$ (identity map on $G$), then the $(\sigma, \tau)$-center of $G$ becomes the usual center of $G$ denoted by $Z(G)$. In particular, if $G$ is finite, then $Z_{\sigma}(G)$ becomes a subgroup of $G$. By \th\ref{lemma 3.7}, $|x_{(\sigma, \tau)}^{G}| = 1$ if and only if $x \in Z_{(\sigma, \tau)}(G)$. Note that we have established the following theorem.

\begin{theorem}\th\label{theorem 3.9}
Let $G$ be a finite group and $x_{1}, ..., x_{r}$ be representatives of the $(\sigma, \tau)$-conjugacy classes of $G$. Then $$|G| = |Z_{(\sigma, \tau)}(G)| + \sum_{x_{i} \notin Z_{(\sigma, \tau)}(G)} |(x_{i})_{(\sigma, \tau)}^{G}|,$$ where $|(x_{i})_{(\sigma, \tau)}^{G}| = \frac{|G|}{|C_{(\sigma, \tau)}(x_{i})|}$ ($1 \leq i \leq r$).
\end{theorem}

Note that the concept of $(\sigma, \tau)$-conjugation is analogous to the idea of usual conjugation in classic group theory. More such analogies may be drawn.

\subsection{The Main Results}\label{subsection 3.2}
\begin{definition}\th\label{definition 3.10}
The set $$Z_{(\sigma, \tau)}(RG) = \{z \in RG \mid z \tau(\alpha) = \sigma(\alpha) z, \hspace{0.1cm} \forall \hspace{0.1cm} \alpha \in RG\}$$ is called the $(\sigma, \tau)$-center of $RG$.
\end{definition}

It can be easily checked that $Z_{(\sigma, \tau)}(RG)$ is an $R$-submodule of the $R$-module $RG$. If $\tau = \sigma$, then we denote the $(\sigma, \tau)$-center of $RG$ by $Z_{\sigma}(RG)$ and $Z_{\sigma}(RG)$ is a subalgebra of $RG$. If $\tau = \sigma = id_{RG}$ (identity map on $RG$), then the $(\sigma, \tau)$-center of $RG$ becomes the usual center of $RG$ denoted by $Z(RG)$. Now, we need to find the rank and a basis of the $R$-module $Z_{(\sigma, \tau)}(RG)$.

\begin{definition}\th\label{definition 3.11}
Let $\{C_{i}\}_{i \in I}$ be the set of all $(\sigma, \tau)$-conjugacy classes of $G$ which are of finite size. For each index $i \in I$, define the sum $$\hat{C_{i}} = \sum_{g \in C_{i}} g.$$ These sums $\hat{C_{i}}$'s are elements in $RG$ and are called $(\sigma, \tau)$-class sums.
\end{definition}

\begin{lemma}\th\label{lemma 3.12}
Let $\{C_{i}\}_{i \in I}$ be the set of all $(\sigma, \tau)$-conjugacy classes of $G$ which are of finite size. Then the $(\sigma, \tau)$-class sums $\{\hat{C_{i}}\}_{i \in I}$ form an $R$-basis of $Z_{(\sigma, \tau)}(RG)$. In particular, if the group $G$ is finite having exactly $r$ distinct $(\sigma, \tau)$-conjugacy classes, then the rank of $Z_{(\sigma, \tau)}(RG)$ is $r$.
\end{lemma}
\begin{proof}
The proof follows on similar lines as that of Proposition $12.22$ in \cite{GordonJames2003} or that of Theorem $3.6.2$ in \cite{CPM2002}.
\end{proof}

\begin{theorem}\th\label{theorem 3.13}
Let $R$ be a commutative ring with unity, $G$ be a group and $\{C_{i}\}_{i \in I}$ ($I$ some indexing set) be the collection of all $(\sigma, \tau)$-conjugacy classes of $G$ which are of finite size with respective representatives collection $\{x_{i}\}_{i \in I}$. Let $J \subseteq I$ such that $|C_{j}| = 1$ for all $j \in J$ and $|C_{i}| \geq 2$ for all $i \in I \setminus J$. Further, let $\{C_{i'}\}_{i' \in I'}$ be the collection of all those $(\sigma, \tau)$-conjugacy classes of $G$ which are of infinite size. Then the set $$\mathcal{B}_{0} = \left\{D_{g} ~ \bigg| ~ g \in \left(\bigcup_{i \in I \setminus J} (C_{i} \setminus \{x_{i}\})\right) \cup \left(\bigcup_{i' \in I'} C_{i'}\right) \right\}$$ forms an $R$-basis of $\text{Inn}_{(\sigma, \tau)}(RG)$.
\end{theorem}
\begin{proof}
By Lemma \ref{lemma 3.12}, the set $\mathcal{B}_{1} = \{\hat{C_{i}} \mid i \in I\}$ is an $R$-basis of $Z_{(\sigma, \tau)}(RG)$. We divide the proof of this theorem into two steps.\vspace{6pt}

\textbf{STEP 1: Constructing an $R$-basis of $\text{Inn}_{(\sigma, \tau)}(RG)$.}

Assume that the basis $\mathcal{B}_{1}$ can be extended to an $R$-basis, say, $$\mathcal{B}_{2} = \{\hat{C}_{i}\}_{i \in I} \cup \{\gamma_{i''}\}_{i'' \in I''}$$ of $RG$ for some indexing set $I''$. (Note that we have to make this assumption to proceed because this is not true in general that a basis of an $R$-submodule $N$ of an $R$-module $M$ can be extended to a basis of $M$, in case of an arbitrary commutative unital ring $R$. However, if $R$ is a field, then a basis of $N$ can always be extended to a basis of $M$.)

Put $\mathcal{B} = \{D_{\gamma_{i''}} \mid i'' \in I''\}$. We show that $\mathcal{B}$ is an $R$-basis of $\text{Inn}_{(\sigma, \tau)}(RG)$.

Let $\lambda \in R$ and $\delta, \delta_{1}, \delta_{2} \in RG$. Then for all $\alpha \in RG$, $$D_{\lambda \delta}(\alpha) = (\lambda \delta)\tau(\alpha) - \sigma(\alpha)(\lambda \delta) = \lambda (\delta \tau(\alpha) - \sigma(\alpha) \delta) = \lambda D_{\delta}(\alpha) = (\lambda D_{\delta})(\alpha)$$ and 
\begin{eqnarray*}
(D_{\delta_{1}} + D_{\delta_{2}})(\alpha) & = & D_{\delta_{1}}(\alpha) + D_{\delta_{2}}(\alpha) \\ & = & (\delta_{1}\tau(\alpha) - \sigma( \alpha)\delta_{1}) + (\delta_{2} \tau(\alpha) - \sigma(\alpha)\delta_{2}) \\ & = &  (\delta_{1} + \delta_{2})\tau(\alpha) - \sigma(\alpha)(\delta_{1} + \delta_{2}) \\ & = & D_{\delta_{1} + \delta_{2}}(\alpha)
\end{eqnarray*}
Therefore, $D_{\lambda \delta} = \lambda D_{\delta}$ and $D_{\delta_{1} + \delta_{2}} = D_{\delta_{1}} + D_{\delta_{2}}$.

Now, let $D \in \text{Inn}_{(\sigma, \tau)}(RG)$. Then $D = D_{\delta}$ for some $$\delta = \sum_{k=1}^{s} \lambda_{i_{k}} \hat{C}_{i_{k}} + \sum_{l=1}^{t} \mu_{i''_{l}} \gamma_{i''_{l}}$$ in $RG$, for some $\{i_{k}\}_{k=1}^{s} \subseteq I$ (for some $s \in \mathbb{N}$) and $\{i''_{l}\}_{l=1}^{t} \subseteq I''$ (for some $t \in \mathbb{N}$). So $$
D_{\delta} = \sum_{k=1}^{s} \lambda_{i_{k}} D_{\hat{C}_{i_{k}}} + \sum_{l=1}^{t} \mu_{i''_{l}} D_{\gamma_{i''_{l}}}.$$ Note that $D_{\hat{C}_{i_{k}}} = 0$ for all $k \in \{1, ..., s\}$. So $D_{\delta} = \sum_{l=1}^{t} \mu_{i''_{l}} D_{\gamma_{i''_{l}}}$. Therefore, $\mathcal{B}$ spans $\text{Inn}_{(\sigma, \tau)}(RG)$ over $R$.

Now, let $\lambda_{i''_{l}} \in R$ for some $\{i''_{l}\}_{l=1}^{t} \subseteq I''$ (for some $t \in \mathbb{N}$) such that $\sum_{l=1}^{t} \lambda_{i''_{l}} D_{\gamma_{i''_{l}}} = 0$.

$\Rightarrow D_{\sum_{l=1}^{t} \lambda_{i''_{l}} \gamma_{i''_{l}}}(\alpha) = 0, \hspace{0.1cm} \forall \hspace{0.1cm} \alpha \in RG$.

$\Rightarrow \left(\sum_{l=1}^{t} \lambda_{i''_{l}} \gamma_{i''_{l}}\right) \tau(\alpha) - \sigma(\alpha) \left(\sum_{l=1}^{t} \lambda_{i''_{l}} \gamma_{i''_{l}}\right) = 0, \hspace{0.1cm} \forall \hspace{0.1cm} \alpha \in RG$.

$\Rightarrow \sum_{l=1}^{t} \lambda_{i''_{l}} \gamma_{i''_{l}} \in Z_{(\sigma, \tau)}(RG)$.

$\Rightarrow \sum_{l=1}^{t} \lambda_{i''_{l}} \gamma_{i''_{l}} = \sum_{k=1}^{s} \mu_{i_{k}} \hat{C}_{i_{k}}$ for some $\mu_{i_{k}} \in R$, for some $\{i_{k}\}_{k=1}^{s} \subseteq I$ (for some $s \in \mathbb{N}$).

$\Rightarrow \mu_{i_{k}} = 0, \hspace{0.1cm} \forall \hspace{0.1cm} k \in \{1, ..., s\}$ and $\lambda_{i''_{l}}=0, \hspace{0.1cm} \forall \hspace{0.1cm} l \in \{1, ..., t\}$.

Therefore, every finite subset of $\mathcal{B}$ is linearly independent over $R$. Therefore, the set $\mathcal{B}$ is $R$-linearly independent. Hence, $\mathcal{B}$ forms an $R$-basis of $\text{Inn}_{(\sigma, \tau)}(RG)$.\vspace{6pt}

\textbf{STEP 2: Constructing $\{\gamma_{i''}\}_{i'' \in I''}$.}

Put $$\mathcal{B}_{3} = \left\{g ~ \bigg| ~ g \in \bigcup_{i \in I \setminus J} (C_{i} \setminus \{x_{i}\})\right\} \bigcup \left\{g ~ \bigg| ~ g \in \bigcup_{i' \in I'} C_{i'}\right\}.$$

Since $G$ is an $R$-linearly independent subset of $RG$ and $\mathcal{B}_{3}$ is a subset of $G$, therefore, $\mathcal{B}_{3}$ is also an $R$-linearly independent subset of $RG$.

Since the $R$-spans of $\mathcal{B}_{1}$ and $\mathcal{B}_{3}$ intersect trivially, and the sets $\mathcal{B}_{1}, \mathcal{B}_{3}$ are both $R$-linearly independent, therefore, the set $\mathcal{B}_{2}' = \mathcal{B}_{1} \cup \mathcal{B}_{3}$ is an $R$-linearly independent subset of $RG$.

Now, we show that $\mathcal{B}_{2}'$ spans $RG$ over $R$.

Let $g \in G$ be arbitrary but fixed. Then, we have the following possible cases.\vspace{4pt}

\textbf{Case 1: $g \in Z_{(\sigma, \tau)}(G)$.}

Then by the observation just after Definition \ref{definition 3.8}, $|g_{(\sigma, \tau)}^{G}| = 1$. Thus, $g_{(\sigma, \tau)}^{G} = C_{j_{0}}$, that is, $C_{j_{0}} = \{g\}$ for some $j_{0} \in J$. So $\hat{C}_{j_{0}} = g$. Therefore, $g \in \mathcal{B}_{1} \subseteq \mathcal{B}_{2}'$. Hence, $g$ belongs to the $R$-span of $\mathcal{B}_{2}'$.\vspace{4pt}

\textbf{Case 2: $g \in G \setminus Z_{(\sigma, \tau)}(G)$.}

Then by the observation following Definition \ref{definition 3.8}, $|g_{(\sigma, \tau)}^{G}| \geq 2$. Therefore, $$g \in \left(\bigcup_{i \in I \setminus J} C_{i}\right) \bigcup \left(\bigcup_{i' \in I'} C_{i'}\right).$$

Then we have the following subcases.\vspace{2pt}

\textbf{Subcase 2.1: $g \in \bigcup_{i' \in I'} C_{i'}$.}

Then $g \in \mathcal{B}_{3} \subseteq \mathcal{B}_{2}'$. Hence, $g$ belongs to the $R$-span of $\mathcal{B}_{2}'$.\vspace{2pt}

\textbf{Subcase 2.2: $g \in \bigcup_{i \in I \setminus J} C_{i}$.}

Then $g \in C_{i}$ for some $i \in I \setminus J$.

Now, since $C_{i} = (x_{i})_{(\sigma, \tau)}^G$, so $C_{i} = (C_{i} \setminus \{x_{i}\}) \cup \{x_{i}\}$. Therefore, $g \in C_{i}$ implies that $g \in C_{i} \setminus \{x_{i}\}$ or $g \in \{x_{i}\}$.

If $g \in C_{i} \setminus \{x_{i}\}$, then $g \in \mathcal{B}_{3} \subseteq \mathcal{B}_{2}'$. Hence, $g$ belongs to the $R$-span of $\mathcal{B}_{2}'$.

If $g = x_{i}$, then observe that $g = \hat{C}_{i} - (\sum_{h \in C_{i} \setminus \{x_{i}\}} h)$ so that $g$ belongs to the $R$-span of $\mathcal{B}_{2}'$.

So, we have finally exhausted all the cases. In all possible cases, we have proved that $g$ belongs to the $R$-span of $\mathcal{B}_{2}'$.

Since $g \in G$ is arbitrary, therefore, every element in $G$ lies in the $R$-span of $\mathcal{B}_{2}'$. That is, $G$ is contained in the $R$-span of $\mathcal{B}_{2}'$. Thus, the $R$-span of $G$ is contained in the $R$-span of $\mathcal{B}_{2}'$.

Since $G$ is an $R$-basis of $RG$, so its $R$-span equals the whole of $RG$. So, we have proved that $RG$ is contained in the $R$-span of $\mathcal{B}_{2}'$.

Since $\mathcal{B}_{2}'$ is a subset of $RG$, therefore, its $R$-span is automatically contained in $RG$.

Therefore, the $R$-span of $\mathcal{B}_{2}'$ is equal to $RG$, that is, $\mathcal{B}_{2}'$ spans $RG$ over $R$.

Therefore, $\mathcal{B}_{2}'$ becomes an $R$-linearly independent subset of $RG$ which spans $RG$ over $R$. Therefore, $\mathcal{B}_{2}'$ is an $R$-basis of $RG$.

Therefore, $\mathcal{B}_{2}' = \mathcal{B}_{1} \cup \mathcal{B}_{3}$ is an $R$-basis of $RG$ which extends the $R$-basis $\mathcal{B}_{1}$ of $Z_{(\sigma, \tau)}(RG)$.

So we can take our set $\{\gamma_{i''}\}_{i'' \in I''}$ (as in STEP 1) to be $\mathcal{B}_{3}$ (that is, $I'' = (I \setminus J) \cup I'$).

Therefore, by STEP 1, the set $\mathcal{B}_{0} = \{D_{g} \mid g \in \mathcal{B}_{3}\}$ is an $R$-basis of $\text{Inn}_{(\sigma, \tau)}(\mathbb{F}G)$. Note that $\mathcal{B}_{0}$ is the same set as given in the statement of the theorem.

Hence proved.
\end{proof}

Recall the definition of a $(\sigma, \tau)$-FC group from \cite{AleksandrAlekseev2020} (see Section $6$). If $G$ is a group and $\sigma, \tau$ are its group endomorphisms, then $G$ is said to be a $(\sigma, \tau)$-FC group if every $(\sigma, \tau)$-conjugacy class of $G$ is of finite size.

The corollary given below is the immediate consequence of Theorem \ref{theorem 3.13}.

\begin{corollary}\th\label{corollary 3.13(1)}
Let $R$ be a commutative ring with unity, $G$ be a $(\sigma, \tau)$-FC group and $\{C_{i}\}_{i \in I}$ ($I$ some indexing set) be the collection of all $(\sigma, \tau)$-conjugacy classes of $G$ with respective representatives collection $\{x_{i}\}_{i \in I}$. Let $J \subseteq I$ such that $|C_{j}| = 1$ for all $j \in J$ and $|C_{i}| \geq 2$ for all $i \in I \setminus J$. Then the set $$\mathcal{B}_{0} = \left\{D_{g} ~ \bigg| ~ g \in \bigcup_{i \in I \setminus J} (C_{i} \setminus \{x_{i}\})\right\}$$ forms an $R$-basis of $\text{Inn}_{(\sigma, \tau)}(RG)$.
\end{corollary}

Below, we have another corollary for a finite group, which is a special case of a $(\sigma, \tau)$-FC group.

\begin{corollary}\th\label{corollary 3.13(2)}
Let $R$ be a commutative ring with unity, $G$ be a finite group of order $n$ and $C_{1}, ..., C_{r}$ be the all possible distinct $(\sigma, \tau)$-conjugacy classes of $G$ with representatives $x_{1}, ..., x_{r}$, respectively. Then the rank of the $R$-module $\text{Inn}_{(\sigma, \tau)}(RG)$ is $n-r$. Furthermore, if $s$ is a positive integer such that $|C_{i}| = 1$ for all $i \in \{1, ..., s\}$ and $|C_{i}| \geq 2$ for all $i \in \{s+1, ..., r\}$, then the set $$\mathcal{B}_{0} = \left\{D_{g} ~ \bigg| ~ g \in \bigcup_{i=s+1}^{r} (C_{i} \setminus \{x_{i}\})\right\}$$ is an $R$-basis of $\text{Inn}_{(\sigma, \tau)}(RG)$.
\end{corollary}

The following corollary is stated in particular, when $R = \mathbb{F}$ is a field.
\begin{corollary}\th\label{corollary 3.13(3)}
Let $\mathbb{F}$ be a field, $G$ be a finite group of order $n$ and $C_{1}, ..., C_{r}$ be the all possible distinct $(\sigma, \tau)$-conjugacy classes of $G$ with representatives $x_{1}, ..., x_{r}$, respectively. Then the dimension of $\text{Inn}_{(\sigma, \tau)}(\mathbb{F}G)$ over $\mathbb{F}$ is $n-r$. If $s$ is a positive integer such that $|C_{i}| = 1$ for all $i \in \{1, ..., s\}$ and $|C_{i}| \geq 2$ for all $i \in \{s+1, ..., r\}$, then the set $$\mathcal{B}_{0} = \left\{D_{g} \mid g \in \bigcup_{i=s+1}^{r} (C_{i} \setminus \{x_{i}\})\right\}$$ is an $\mathbb{F}$-basis of $\text{Inn}_{(\sigma, \tau)}(\mathbb{F}G)$.
\end{corollary}
\begin{proof}
The proof follows directly from Theorem \ref{theorem 3.13} or Corollary \ref{corollary 3.13(2)}, but we can also prove it by following a slightly different approach in the end.

By \th\ref{lemma 3.12}, $\text{dim}(Z_{(\sigma, \tau)}(\mathbb{F}G)) = r$ and $\mathcal{B}_{1} = \{\hat{C}_{i} \mid 1 \leq i \leq r\}$ is an $\mathbb{F}$-basis of $Z_{(\sigma, \tau)}(\mathbb{F}G)$. The basis $\mathcal{B}_{1}$ can be extended to an $\mathbb{F}$-basis, say, $\mathcal{B}_{2} = \{\hat{C}_{1}, ..., \hat{C}_{r}, \gamma_{r+1}, ..., \gamma_{n}\}$ of $\mathbb{F}G$.

Put $\mathcal{B} = \{D_{\gamma_{i}} \mid r+1 \leq i \leq n\}$. Then as shown in the proof of Theorem \ref{theorem 3.13}, it can be shown that $\mathcal{B}$ is a basis of $\text{Inn}_{(\sigma, \tau)}(\mathbb{F}G)$ over $\mathbb{F}$.

As $|\mathcal{B}| = n-r$, therefore, $\text{dim}(\text{Inn}_{(\sigma, \tau)}(\mathbb{F}G)) = n-r$. Note that this fact also follows from the fact that the map $\theta:\mathbb{F}G \rightarrow \text{Inn}_{(\sigma, \tau)}(\mathbb{F}G)$ ($\alpha \mapsto D_{\alpha}$) is an $\mathbb{F}$-linear map with null space $Z_{(\sigma, \tau)}(\mathbb{F}G)$.

Now, we show that $\mathcal{B}_{0}$ is an $\mathbb{F}$-basis of $\text{Inn}_{(\sigma, \tau)}(\mathbb{F}G)$.

Let $|C_{i}| = n_{i}, \hspace{0.1cm} \forall \hspace{0.1cm} i \in \{s+1, ..., r\}$. As $G = \bigcup_{i=1}^{r} C_{i}$ and $(\sigma, \tau)$-conjugacy classes are pairwise disjoint, so $$n = |G| = |\bigcup_{i=1}^{r} C_{i}| = \sum_{i=1}^{r} |C_{i}| = \sum_{i=1}^{s} |C_{i}| + \sum_{j=s+1}^{r} |C_{j}| = s + \sum_{j=s+1}^{r} n_{j}$$ so that $\sum_{j=s+1}^{r} n_{j} = n-s$.

Suppose that $\sum_{g \in \bigcup_{i=s+1}^{r} C_{i} \setminus \{x_{i}\}} \lambda_{g} D_{g} = 0$ for $\lambda_{g} \in \mathbb{F}$ ($g \in \bigcup_{i=s+1}^{r} C_{i} \setminus \{x_{i}\}$).

$\Rightarrow D_{\sum_{g \in \bigcup_{i=s+1}^{r} C_{i} \setminus \{x_{i}\}} \lambda_{g} g} = 0$.

$\Rightarrow \left(\sum_{g \in \bigcup_{i=s+1}^{r} C_{i} \setminus \{x_{i}\}} \lambda_{g} g\right) \tau(\alpha) = \sigma(\alpha) \left(\sum_{g \in \bigcup_{i=s+1}^{r} C_{i} \setminus \{x_{i}\}} \lambda_{g} g\right), \hspace{0.1cm} \forall \hspace{0.1cm} \alpha \in \mathbb{F}G$.

$\Rightarrow \sum_{g \in \bigcup_{i=s+1}^{r} C_{i} \setminus \{x_{i}\}} \lambda_{g} g \in Z_{(\sigma, \tau)}(\mathbb{F}G)$.

$\Rightarrow \sum_{g \in \bigcup_{i=s+1}^{r} C_{i} \setminus \{x_{i}\}} \lambda_{g} g = \sum_{i=1}^{r} \mu_{i} \hat{C}_{i}$ for some unique $\mu_{i} \in \mathbb{F}$ ($1 \leq i \leq r$).

$\Rightarrow \sum_{i=1}^{s} \mu_{i} x_{i} + \sum_{i=s+1}^{r} \mu_{i} x_{i} + \sum_{i=s+1}^{r} \left(\sum_{g \in C_{i} \setminus \{x_{i}\}} (\mu_{i} - \lambda_{g}) g\right) = 0$.

Since $G$ is an $\mathbb{F}$-linearly independent subset of $\mathbb{F}G$ and the $(\sigma, \tau)$-conjugacy classes $C_{i}$'s ($1 \leq i \leq r$) are pairwise disjoint, so we get that $$\lambda_{g} = 0, ~~~~ \forall ~ g \in \bigcup_{i=s+1}^{r} (C_{i} \setminus \{x_{i}\}).$$ Therefore, the set $\mathcal{B}_{0}$ is $\mathbb{F}$-linearly independent.

Further, since the sets $C_{s+1} \setminus \{x_{s+1}\}, ..., C_{r} \setminus \{x_{r}\}$ are pairwise disjoint, therefore, $$|\mathcal{B}_{0}| = \sum_{i=s+1}^{r} |C_{i} \setminus \{x_{i}\}| = \sum_{i=s+1}^{r} (|C_{i}| - 1) = \left(\sum_{i=s+1}^{r} n_{i}\right) - (r-s) = (n-s) - (r-s) = n-r.$$
So $|\mathcal{B}_{0}| = \text{dim}(\text{Inn}_{(\sigma, \tau)}(\mathbb{F}G))$. Therefore, $\mathcal{B}_{0}$ is a basis of $\text{Inn}_{(\sigma, \tau)}(\mathbb{F}G)$.

Hence proved.
\end{proof}

Note that the set $$\{D_{\gamma} \mid \gamma \in \{g - x_{i} \mid g \in \cup_{i=s+1}^{r} (C_{i} \setminus \{x_{i}\})\}\}$$ is another basis of $\text{Inn}_{(\sigma, \tau)}(\mathbb{F}G)$ over $\mathbb{F}$.

Below, we prove a significant result.

\begin{theorem}\th\label{theorem 3.14}
Let $R$ be a ring with unity and $G$ be a finite group of order $n$, which is invertible in $R$. Let $\sigma, \tau$ be arbitrary unital $R$-algebra endomorphisms of $RG$. Then every $(\sigma, \tau)$-derivation of $RG$ is inner.
\end{theorem}
\begin{proof}
Let $D \in \mathcal{D}_{(\sigma, \tau)}(RG)$. Put $$\gamma = \frac{1}{|G|} \sum_{g \in G} D(g^{-1}) \tau(g) ~~~~ \left(\text{or} - \frac{1}{|G|} \sum_{g \in G} \sigma(g^{-1})D(g)\right).$$ Then $\gamma \in RG$. 

Since $D:RG \rightarrow RG$ is a $(\sigma, \tau)$-derivation, so for any $\alpha, \beta \in RG$, $$D(\alpha \beta) = D(\alpha) \tau(\beta) + \sigma(\alpha) D(\beta)$$ so that $$\sigma(\alpha) D(\beta) = D(\alpha \beta) - D(\alpha) \tau(\beta).$$

Now, for any $x \in G$, \begin{eqnarray*}
D_{\gamma}(x) & = & \gamma \tau(x) - \sigma(x) \gamma \\ & = & \frac{1}{|G|} \sum_{g \in G} D(g^{-1}) \tau(gx) - \frac{1}{|G|} \sum_{g \in G} (D(x g^{-1}) - D(x) \tau(g^{-1})) \tau(g) \\ & = & \frac{1}{|G|} \sum_{y \in G} D(xy^{-1}) \tau(y) - \frac{1}{|G|} \sum_{g \in G} D(xg^{-1}) \tau(g) + \frac{1}{|G|} \sum_{g \in G} D(x) \\ & = & D(x).
\end{eqnarray*}

So $D = D_{\gamma}$. Therefore, $D$ is inner.
Hence proved.
\end{proof}

\begin{corollary}\th\label{corollary 3.15}
Let $R$ be a commutative ring with unity and $G$ be a finite group whose order is invertible in $R$. Let $C_{1}, ..., C_{r}$ be the all possible distinct $(\sigma, \tau)$-conjugacy classes of $G$ with representatives $x_{1}, ..., x_{r}$, respectively. 
Then every $(\sigma, \tau)$-derivation of $RG$ is inner, that is, $\mathcal{D}_{(\sigma, \tau)}(RG) = \text{Inn}_{(\sigma, \tau)}(RG)$. Furthermore, if $s$ is a positive integer such that $|C_{i}| = 1$ for all $i \in \{1, ..., s\}$ and $|C_{i}| \geq 2$ for all $i \in \{s+1, ..., r\}$, then the set $$\mathcal{B}_{0} = \left\{D_{g} \mid g \in \bigcup_{i=s+1}^{r} (C_{i} \setminus \{x_{i}\})\right\}$$ is an $R$-basis of $\mathcal{D}_{(\sigma, \tau)}(RG)$.
\end{corollary}
\begin{proof}
Follows from Corollary \ref{corollary 3.13(2)} and Theorem \ref{theorem 3.14}.
\end{proof}

Finally, we compare this section's results with those already known in the literature. Theorem \ref{theorem 3.13} is a much broader result that classifies all the inner $(\sigma, \tau)$-derivations of a group algebra $RG$ of an arbitrary group $G$ over an arbitrary commutative unital ring $R$, where $\sigma, \tau$ are $R$-algebra endomorphisms of $RG$ which are $R$-linear extensions of the group endomorphisms of $G$. Also, Corollary \ref{corollary 3.13(3)} generalizes Theorem 2.17 of \cite{doi:10.1142/S0219498825503712}, which was proved for ordinary derivations. In \cite{AleksandrAlekseev2020}, the authors considered the $(\sigma, \tau)$ and $(\sigma, id)$-derivations of the group algebra $\mathbb{C}G$ of a discrete countable group $G$, where $\sigma, \tau$ are $\mathbb{C}$-algebra endomorphisms of $\mathbb{C}G$ which are $\mathbb{C}$-linear extensions of the group endomorphisms of $G$. Theorem \ref{theorem 3.14} generalizes {\cite[Corollary 5]{AleksandrAlekseev2020}} in which the authors proved that every $(\sigma, \tau)$-derivation of the group algebra $\mathbb{C}G$ of a finite group $G$ is inner. In {\cite[Corollary 7]{AleksandrAlekseev2020}}, the authors gave a criterion for a $(\sigma, \tau)$-derivation of the group algebra $\mathbb{C}G$ of a finitely generated $(\sigma, \tau)$-FC group $G$ to be inner. In {\cite[Section 4]{AleksandrAlekseev2020}}, the authors studied quasi-inner $(\sigma, \tau)$-derivations of $\mathbb{C}G$ defined by the authors in {\cite[Definition 6]{AleksandrAlekseev2020}} using characters. In {\cite[Proposition 2]{AleksandrAlekseev2020}} and {\cite[Definition 8]{AleksandrAlekseev2020}}, the authors defined a $(\sigma, \tau)$-central $(\sigma, \tau)$-derivation. In {\cite[Proposition 3]{AleksandrAlekseev2020}}, the authors proved that a non-zero $(\sigma, \tau)$-central $(\sigma, \tau)$-derivation is not quasi-inner. In {\cite[Theorem 2]{AleksandrAlekseev2020}}, the authors gave a way to calculate quasi-inner $(\sigma, \tau)$-derivations.
In \cite{Arutyunov2021} also, the authors studied the ordinary inner derivations of finite and FC groups using topological and character techniques. In {\cite[Theorem 4.3]{Arutyunov2021}}, the authors gave a characterization under which all ordinary derivations of $AG$ are inner for a finite group $G$ and a finite commutative unital ring $A$. In {\cite[Corollary 4.2.4]{Arutyunov2021}}, the authors proved that for a torsion-free commutative unital ring $A$ and a finite group $G$, all derivations of $AG$ are inner. In \cite{Arutyunov2020a}, the authors described the space of all inner and outer derivations of a group algebra of a finitely presented discrete group in terms of the character spaces of the 2-groupoid of the adjoint action of the group. Theorem \ref{theorem 3.14} also generalizes  {\cite[Theorem 1.1]{Chaudhuri2019}}. Also, note that a result similar to Theorem \ref{theorem 3.14} was obtained for ordinary derivations to be locally inner in {\cite[Lemma 3]{OrestD.Artemovych2020}} and earlier in {\cite[Theorem 2.1]{ikeda1994derivation}}. For more details, we refer the reader to the third paragraph of the introduction section (Section \ref{section 1}), where we have mentioned the work done in literature on the problem of determining the conditions under which a derivation or a twisted derivation of a group algebra is inner thus illustrating the fact that this problem has always been of interest to researchers. We also refer the reader to the end of Subsection \ref{subsection 2.3} for more insights.

\section{Application: \texorpdfstring{$\sigma$}{Lg}-Derivations of Dihedral Group \\ Algebras}\label{section 4}
In this section, we apply the results obtained in Sections \ref{section 2} and \ref{section 3} to explicitly classify all inner and outer $\sigma$-derivations of the dihedral group algebra $\mathbb{F}D_{2n}$, where $\mathbb{F}$ is an arbitrary field and $$D_{2n} = \langle a, b \mid a^{n} = 1 = b^{2}, (ab)^{2} = 1 \rangle$$ ($n \geq 3$). In \cite{Johnson2013}, the author has given all group homomorphisms from $D_{2m}$ to $D_{2n}$ for arbitrary positive integers $m$ and $n$. When $n$ is odd, $D_{2n}$ has precisely $n^{2} + 1$ group endomorphisms which are of the following two forms:
\begin{enumerate}
\item[•] $\sigma_{-1}(a) = 1$ and $\sigma_{-1}(b) = 1$.
\item[•] $\sigma_{0}(a) = a^{s}$ ($1 \leq s \leq n-1$) and $\sigma_{0}(b) = a^{t}b$ ($0 \leq t \leq n-1$).
\end{enumerate}
When $n$ is even, $D_{2n}$ has precisely $(n+2)^{2}$ group endomorphisms which are of the following forms:
\begin{enumerate}
\item[•] $\sigma_{1}(a) = a^{s}$ ($0 \leq s \leq n-1$, $s \neq 0, \frac{n}{2}$) and $\sigma_{1}(b) = a^{t}b$ ($0 \leq t \leq n-1$).

\item[•] $\sigma_{2}(a) = a^{s}$ ($s = 0, \frac{n}{2}$) and $\sigma_{2}(b) = a^{t}$ ($t = 0, \frac{n}{2}$).

\item[•] $\sigma_{3}(a) = a^{s}$ ($s = 0, \frac{n}{2}$) and $\sigma_{3}(b) = a^{t}b$ ($0 \leq t \leq n-1$).

\item[•] $\sigma_{4}(a) = a^{s}b$ ($0 \leq s \leq n-1$), $\sigma_{4}(b) = a^{t}b$ ($t = s, s+\frac{n}{2}$).

\item[•] $\sigma_{5}(a) = a^{s}b$ ($0 \leq s \leq n-1$), $\sigma_{5}(b) = a^{t}$ ($t = 0, \frac{n}{2}$).
\end{enumerate}

We will obtain the results for just $\sigma_{0}$ when $n$ is odd and $\sigma_{1}$ when $n$ is even because results for the remaining ones become exercises.

\subsection{$\sigma$-Derivations of $\mathbb{F}D_{2n}$ when $\text{char}(\mathbb{F}) = 0$ or an odd prime $p$}\label{subsection 4.1}

\begin{definition}\th\label{definition 4.1}
Let $R$ be a ring with unity and $G$ be a group. Let $\beta \in RG$. Then the set $$\bar{C}(\beta) = \{\alpha \in RG \mid \alpha \beta = - \beta \alpha\}$$ is called the anti-centralizer of $\beta$ in $RG$.
\end{definition}

$\bar{C}(\beta)$ is a submodule of the $R$-module $RG$. If $R = \mathbb{F}$ is a field of characteristic $0$ or an odd rational prime $p$ and $g \in G \cap Z(\mathbb{F}G)$, then $\bar{C}(g) = \{0\}$. This is so because if $\alpha \in \bar{C}(g)$, then $\alpha g = - g \alpha$ so that $\alpha g = - \alpha g$ (as $g \alpha = \alpha g$). This gives $2 \alpha g = 0$ so that $2 \alpha = 0$. But since $\text{char}(\mathbb{F}) \neq 2$, therefore, $\alpha = 0$. Therefore, $\bar{C}(g) = \{0\}$.

\begin{lemma}\th\label{lemma 4.2}
Let $\mathbb{F}$ be a field of characteristic $0$ or an odd rational prime $p$ and $\sigma = \sigma_{0}$ or $\sigma_{1}$.
\begin{enumerate}
\item[(i)] The set $$\bar{\mathcal{B}}(\sigma(b)) = \{(a^{i} - a^{-i}), ~ a^{t}(a^{i} - a^{-i})b \mid 1 \leq i \leq \lfloor \frac{n-1}{2} \rfloor\}$$ is a basis of $\bar{C}(\sigma(b))$ over $\mathbb{F}$.
\item[(ii)] The set $$\bar{\mathcal{B}}(\sigma(ab)) = \{(a^{i} - a^{-i}), ~ a^{s+t}(a^{i} - a^{-i})b \mid 1 \leq i \leq \lfloor \frac{n-1}{2} \rfloor\}$$ is a basis of $\bar{C}(\sigma(ab))$ over $\mathbb{F}$.
\end{enumerate}

\end{lemma}
\begin{proof}
First, let $n$ be odd.

\textbf{(i)} It can be easily verified that $\bar{\mathcal{B}}(\sigma_{0}(b))$ is an $\mathbb{F}$-linearly independent subset of $\bar{C}(\sigma_{0}(b))$. 

Now, let $\alpha \in \bar{C}(\sigma_{0}(b))$ and $$\alpha = \sum_{\substack{0 \leq i \leq n-1 \\ 0 \leq j \leq 1}} \lambda_{i,j}a^{i}b^{j}$$ for some $\lambda_{i,j} \in \mathbb{F}$ ($0 \leq i \leq n-1$, $0 \leq j \leq 1$). 

We make the following assumption (A*), which will be used to simplify the calculations:
\begin{center}
(A*): For each $j \in \{0, 1\}$, $\lambda_{i,j} = \lambda_{i',j}$ whenever $i \equiv i'$ (mod $n$). 
\end{center}

\noindent Now, $$\alpha \sigma_{0}(b) = \sum_{0 \leq i \leq n-1} \lambda_{i,0} a^{i+t}b + \sum_{0 \leq i \leq n-1} \lambda_{i,1} a^{i-t}$$ and $$\sigma_{0}(b) \alpha = \sum_{0 \leq i \leq n-1} \lambda_{i,0} a^{t-i}b + \sum_{0 \leq i \leq n-1} \lambda_{i,1} a^{t-i}.$$ Therefore, from $\alpha \sigma_{0}(b) = - \sigma_{0}(b) \alpha$, we get that $$\sum_{0 \leq i \leq n-1} \lambda_{i,0} a^{i+t}b = - \sum_{0 \leq i \leq n-1} \lambda_{i,0} a^{t-i}b ~~~~ \text{and} ~~~~ \sum_{0 \leq i \leq n-1} \lambda_{i,1} a^{i-t} = - \sum_{0 \leq i \leq n-1} \lambda_{i,1} a^{t-i}.$$

From the first equality and using (A*),  
\begin{equation*}
\begin{aligned}
\sum_{0 \leq i \leq n-1} \lambda_{i,0} a^{i} = - \sum_{0 \leq i \leq n-1} \lambda_{(n-i),0} a^{i}.
\end{aligned}
\end{equation*} Therefore, $\lambda_{i,0} = - \lambda_{(n-i),0}$ for all $i \in \{0, 1, ..., n-1\}$. In particular, $\lambda_{0,0} = 0$. Therefore,
\begin{align*}
\sum_{0 \leq i \leq n-1} \lambda_{i,0} a^{i} & = \sum_{1 \leq i \leq n-1} \lambda_{i,0} a^{i} = \sum_{1 \leq i \leq \frac{n-1}{2}} \lambda_{i,0} a^{i} + \sum_{\frac{n-1}{2}+1 \leq i \leq n-1} \lambda_{i,0} a^{i} \\ & = \sum_{1 \leq i \leq \frac{n-1}{2}} \lambda_{i,0} a^{i} + \sum_{1 \leq i \leq \frac{n-1}{2}} \lambda_{(n-i),0} a^{-i} \\ & = \sum_{1 \leq i \leq \frac{n-1}{2}} \lambda_{i,0} a^{i} - \sum_{1 \leq i \leq \frac{n-1}{2}} \lambda_{i,0} a^{-i} \\ & = \sum_{1 \leq i \leq \frac{n-1}{2}} \lambda_{i,0} (a^{i} - a^{-i}).
\end{align*}

From the second equality, we get that $$\sum_{0 \leq i \leq n-1} \lambda_{i,1} a^{i-t} = - \sum_{0 \leq i \leq n-1} \lambda_{i,1} a^{n-i+t}.$$

$\Rightarrow \sum_{-t \leq j \leq -t+n-1} \lambda_{(t+j),1} a^{j} = - \sum_{t+1 \leq j \leq t+n} \lambda_{(n-j+t),1} a^{j}.$

$\Rightarrow \sum_{0 \leq i \leq n-1} \lambda_{(t+i),1}a^{i} = - \sum_{0 \leq i \leq n-1} \lambda_{(n-i+t),1} a^{i}$ in view of (A*).

Therefore, $\lambda_{(t+i),1} = - \lambda_{(n-i+t),1}$ for all $i \in \{0, 1, ..., n-1\}$. Note that $\lambda_{t,1} = 0$ (for $i = 0$). Therefore, $$\sum_{0 \leq i \leq n-1} \lambda_{i,1} a^{i}b = \sum_{1 \leq i \leq \frac{n-1}{2}} \lambda_{i,1} a^{i}b + \sum_{1 \leq i \leq \frac{n-1}{2}} \lambda_{(n-i),1} a^{-i}b.$$ 
In view of (A*), the sum $$\sum_{1 \leq i \leq \frac{n-1}{2}} \lambda_{i,1} a^{i}b + \sum_{1 \leq i \leq \frac{n-1}{2}} \lambda_{(n-i),1} a^{-i}b$$ is, in fact, equal to $$\sum_{1 \leq i \leq \frac{n-1}{2}} \lambda_{(t+i),1} a^{t+i}b + \sum_{1 \leq i \leq \frac{n-1}{2}} \lambda_{(n-i+t),1} a^{t-i}b.$$
Therefore, 
\begin{equation*}
\begin{aligned}
\sum_{0 \leq i \leq n-1} \lambda_{i,1} a^{i}b & = \sum_{1 \leq i \leq \frac{n-1}{2}} \lambda_{(t+i),1} a^{t+i}b + \sum_{1 \leq i \leq \frac{n-1}{2}} \lambda_{(n-i+t),1} a^{t-i}b \\ & = \sum_{1 \leq i \leq \frac{n-1}{2}} \lambda_{(t+i),1} a^{t+i}b - \sum_{1 \leq i \leq \frac{n-1}{2}} \lambda_{(t+i),1} a^{t-i}b \\ & = \sum_{1 \leq i \leq \frac{n-1}{2}} \lambda_{(t+i),1} a^{t}(a^{i} - a^{-i})b.
\end{aligned}
\end{equation*}

\noindent Therefore, \begin{eqnarray*}\alpha & = & \sum_{0 \leq i \leq n-1} \lambda_{i,0} a^{i} + \sum_{0 \leq i \leq n-1} \lambda_{i,1} a^{i}b \\ & = & \sum_{1 \leq i \leq \frac{n-1}{2}} \lambda_{i,0} (a^{i} - a^{-i}) + \sum_{1 \leq i \leq \frac{n-1}{2}} \lambda_{(t+i),1} a^{t}(a^{i} - a^{-i})b.\end{eqnarray*}

\noindent Therefore, the set $\bar{\mathcal{B}}(\sigma_{0}(b))$ spans $\bar{C}(\sigma_{0}(b))$ over $\mathbb{F}$. Hence, $\bar{\mathcal{B}}(\sigma_{0}(b))$ is an $\mathbb{F}$-basis of $\bar{C}(\sigma_{0}(b))$.\vspace{10pt}

\textbf{(ii)} The map $\theta:D_{2n} \rightarrow D_{2n}$ defined by $$\theta(a^{i}b^{j}) = a^{i}(a^{s}b)^{j} ~~~~~~ (0 \leq i \leq n-1, ~ 0 \leq j \leq 1),$$ is an automorphism of $D_{2n}$. Extending $\theta$ $\mathbb{F}$-linearly to the whole of $\mathbb{F}D_{2n}$ makes $\theta$ an automorphism of the group algebra $\mathbb{F}D_{2n}$. Note that $\theta(a) = a$ and $\theta(b) = a^{s}b$.

We show that $\alpha \in \bar{C}(\sigma_{0}(b))$ if and only if $\theta(\alpha) \in \bar{C}(\sigma_{0}(ab))$. 

First, let $\alpha \in \bar{C}(\sigma_{0}(b))$. Then $\alpha \in \mathbb{F}D_{2n}$ such that $$\alpha \sigma_{0}(b) = - \sigma_{0}(b) \alpha, ~~~~ \text{that is}, ~~~~ \alpha (a^{t}b) = - (a^{t}b) \alpha.$$ 

Note that $\alpha$ can be written as $\alpha = c_{0} + c_{1}b$ for some $c_{0}, c_{1} \in \mathbb{F} \langle a \rangle$. Then $\theta(\alpha) = c_{0} + c_{1}a^{s}b$. 

As $\alpha (a^{t}b) = - (a^{t}b) \alpha$, so $c_{0}a^{t}b + c_{1}ba^{t}b = - a^{t}bc_{0} - a^{t}bc_{1}b$ so that $$c_{0}a^{t}b = - a^{t}bc_{0} ~~~~~~ \text{and} ~~~~~~ c_{1}ba^{t} = - a^{t}bc_{1}.$$ Therefore, 
\begin{eqnarray*}
\theta(\alpha)\sigma_{0}(ab) & = & (c_{0} + c_{1}a^{s}b)(a^{s+t}b) \\ & = & a^{s}c_{0}a^{t}b + c_{1}ba^{t}b \\ & = & -a^{s}a^{t}bc_{0} - a^{t}bc_{1}b \\ & = & - a^{s+t}b(c_{0} + c_{1}a^{s}b) \\ & = & - \sigma_{0}(ab) \theta(\alpha).
\end{eqnarray*}
Therefore, $\theta(\alpha) \in \bar{C}(\sigma_{0}(ab))$.

Conversely, let $\theta(\alpha) \in \bar{C}(\sigma_{0}(ab))$. Then $\theta(\alpha)\sigma_{0}(ab) = - \sigma_{0}(ab) \theta(\alpha)$. This gives $$a^{s+t}c_{0}b + c_{1}ba^{t}b = - a^{s+t}bc_{0} - a^{t}bc_{1}b$$ so that $$c_{0}b = - bc_{0} ~~~~~~ \text{and} ~~~~~~ c_{1}ba^{t} = - a^{t}bc_{1}.$$ Therefore, 
\begin{eqnarray*}
\alpha \sigma_{0}(b) & = & (c_{0} + c_{1}b) a^{t}b \\ & = & a^{t}c_{0}b + c_{1}ba^{t}b \\ & = & -a^{t}bc_{0} - a^{t}bc_{1}b \\ & = & - \sigma_{0}(b) (c_{0} + c_{1}b) \\ & = & - \sigma_{0}(b) \alpha.
\end{eqnarray*}
Therefore, $\alpha \in \bar{C}(\sigma_{0}(b))$.

Therefore, the set $\theta(\bar{\mathcal{B}}(\sigma_{0}(b))) = \bar{\mathcal{B}}(\sigma_{0}(ab))$ is an $\mathbb{F}$-basis of $\bar{C}(\sigma_{0}(ab))$.

The proof for $n$ even is similar to that for $n$ odd.
\end{proof}

Let $\mathbb{F}$ be a field and $\sigma$ be an $\mathbb{F}$-algebra endomorphism of $\mathbb{F}D_{2n}$ which is an $\mathbb{F}$-linear extension of a group endomorphism of $D_{2n}$. The set $X = \{a, b \}$ is the generating set of $D_{2n}$ and $a^{n}, b^{2}, (ab)^{2}$ are its relators. Let $f:X \rightarrow \mathbb{F}D_{2n}$ be a map that can be extended to a $\sigma$-derivation of $\mathbb{F}D_{2n}$. By \th\ref{theorem 2.4(N)}, this happens if and only if $$\tilde{f}(a^{n}) = 0, ~~~~~~ \tilde{f}(b^{2}) = 0 ~~~~~~ \text{and} ~~~~~~ \tilde{f}((ab)^{2}) = 0,$$ where $\tilde{f}:F(X) \rightarrow \mathbb{F}D_{2n}$ is the unique extension of $f$ defined in (\ref{eq 2.2(N)}) and (\ref{eq 2.3(N)}) and satisfying (\ref{eq 2.1(N)}), and $F(X)$ is the free group on $X$.

Now, in view of \th\ref{lemma 2.3(N)} and (\ref{eq 2.1(N)}), $$\tilde{f}(b^{2}) = \tilde{f}(b) \tilde{\sigma}(b) + \tilde{\sigma}(b) \tilde{f}(b) = f(b)\sigma(b) + \sigma(b)f(b)$$ as $\tilde{\sigma} = \sigma$ on $G$. So $\tilde{f}(b^{2}) = 0$ if and only if $f(b)\sigma(b) = - \sigma(b)f(b)$, that is, $f(b) \in \bar{C}(\sigma(b))$. Similarly, $\tilde{f}((ab)^{2}) = 0$ if and only if $\tilde{f}(ab) \in \bar{C}(\sigma(ab))$.
Since $f(a) \in \mathbb{F}D_{2n}$, so $f(a)$ can be written as $f(a) = c_{0} + c_{1}b$ for some $c_{0}, c_{1} \in \mathbb{F} \langle a \rangle$. Now, using (\ref{eq 2.3(N)}), \begin{equation}\label{eq 4.1}
\begin{aligned}
\tilde{f}(a^{n}) & = \sum_{i=1}^{n}(\prod_{j=1}^{i-1} \sigma(a)) \tilde{f}(a) (\prod_{j=i+1}^{n} \sigma(a)) \\ & = \sum_{i=1}^{n} (\sigma(a))^{i-1} f(a) (\sigma(a))^{n-i} \\ & = \sum_{j=0}^{n-1} (\sigma(a))^{j} c_{0} (\sigma(a))^{n-(j+1)} + \sum_{j=0}^{n-1} (\sigma(a))^{j} c_{1}b (\sigma(a))^{n-(j+1)}.
\end{aligned}
\end{equation}
Further, since $\tilde{f}(ab) = \tilde{f}(a) \tilde{\sigma}(b) + \tilde{\sigma}(a) \tilde{f}(b) = f(a) \sigma(b) + \sigma(a) f(b)$, so \begin{equation}\label{eq 4.2}
f(a) = \tilde{f}(ab) \sigma(b) - \sigma(a)f(b)\sigma(b)
\end{equation}

We note that any $\sigma$-derivation $D$ of the group algebra $\mathbb{F}D_{2n}$ is uniquely determined by its images $D(a)$ and $D(b)$ on the generators $a$ and $b$ of $D_{2n}$. We denote such a derivation by $$D_{(\sigma; \bar{a}, \bar{b})},$$ where $\bar{a} = D(a)$ and $\bar{b} = D(b)$. Also, we know that $\mathcal{D}_{\sigma}(\mathbb{F}D_{2n})$ forms a vector space over $\mathbb{F}$. We determine this vector space's dimension and basis in the following theorems.

Consider $\sigma = \sigma_{0}$ or $\sigma_{1}$. Since $f(b) \in \bar{C}(\sigma(b))$ and by \th\ref{lemma 4.2}, $\bar{\mathcal{B}}(\sigma(b))$ is a basis of $\bar{C}(\sigma(b))$ over $\mathbb{F}$, therefore, \begin{equation}\label{eq 4.3}
f(b) = \sum_{i=1}^{\lfloor \frac{n-1}{2} \rfloor} \mu_{i} (a^{i} - a^{-i}) + \sum_{i=1}^{\lfloor \frac{n-1}{2} \rfloor} \nu_{i} a^{t}(a^{i} - a^{-i})b\end{equation} for some $\mu_{i}, \nu_{i} \in \mathbb{F}$ ($1 \leq i \leq \lfloor \frac{n-1}{2} \rfloor$).

Further, since $\tilde{f}(ab) \in \bar{C}(\sigma(ab))$ and by \th\ref{lemma 4.2}, $\bar{\mathcal{B}}(\sigma(ab))$ is a basis of $\bar{C}(\sigma(ab))$ over $\mathbb{F}$, therefore, \begin{equation}\label{eq 4.4}
\tilde{f}(ab) = \sum_{i=1}^{\lfloor \frac{n-1}{2} \rfloor} \delta_{i} (a^{i} - a^{-i}) + \sum_{i=1}^{\lfloor \frac{n-1}{2} \rfloor} \gamma_{i} a^{s+t}(a^{i} - a^{-i})b\end{equation} for some $\delta_{i}, \gamma_{i} \in \mathbb{F}$ ($1 \leq i \leq \lfloor \frac{n-1}{2} \rfloor$).

\begin{theorem}\th\label{theorem 4.3}
Let $n$ be odd, and $\mathbb{F}$ be a field of characteristic $0$ or an odd rational prime $p$. Denote $|\sigma_{0}(a)| = m$, $d = \frac{n}{m}$ and $t \equiv j_{0}$ (mod $s$) for some $j_{0} \in \{0, 1, ..., s-1\}$. 
\begin{enumerate}
\item[(i)] Suppose $\text{char}(\mathbb{F}) = 0$ or $p$ with $\text{gcd}(n,p) = 1$. Then the dimension of $\mathcal{D}_{\sigma_{0}}(\mathbb{F}D_{2n})$ over $\mathbb{F}$ is $\frac{1}{2}(3n-d)-1$ and 
\begin{equation*}
\begin{aligned}
\mathcal{S}_{0} & = \{D_{(\sigma_{0}; \bar{a}, \bar{b})} \mid (\bar{a}, \bar{b}) \in \{(a^{t}(a^{i}-a^{-i})b, 0),      ~~     (a^{s+t}(a^{i} - a^{-i})b, (a^{i} - a^{-i})),      \\ &\quad       (0, a^{t}(a^{i} - a^{-i})b) \mid 1 \leq i \leq \frac{n-1}{2}\}\}
\end{aligned}
\end{equation*} is a spanning set.
\item[(ii)] Suppose $\text{char}(\mathbb{F}) = p$ with $\text{gcd}(n,p) \neq 1$.
\begin{enumerate}
\item[(a)] If $\text{gcd}(d,p) = 1$, then the dimension of $\mathcal{D}_{\sigma_{0}}(\mathbb{F}D_{2n})$ over $\mathbb{F}$ is $2n - \frac{1}{2}(d+3)$ and 
\begin{equation*}
\begin{aligned}
\mathcal{S}_{0}' & = \{D_{(\sigma_{0}; \bar{a}, \bar{b})} \mid (\bar{a}, \bar{b}) \in \{(a^{s}(a^{i} - a^{-i}), 0),     ~~      (a^{s}(a^{i}-a^{-i}), a^{t}(a^{i}-a^{-i})b),     \\ &\quad     (a^{t}(a^{i}-a^{-i})b, 0),   ~~    (a^{s+t}(a^{i} - a^{-i})b, (a^{i} - a^{-i})) \mid 1 \leq i \leq \frac{n-1}{2}\}\}
\end{aligned}
\end{equation*} is a spanning set.
\item[(b)] If $\text{gcd}(d,p) \neq 1$, then the dimension of $\mathcal{D}_{\sigma_{0}}(\mathbb{F}D_{2n})$ over $\mathbb{F}$ is $4(\frac{n-1}{2})$ and the set $\mathcal{S}_{0}'$ is a basis.
\end{enumerate}
\end{enumerate}
\end{theorem}
\begin{proof}
By (\ref{eq 4.1}), $$\tilde{f}(a^{n}) = nc_{0} a^{s(n-1)} + \sum_{j=0}^{n-1} a^{2sj} c_{1}b a^{s(n-1)}.$$
\noindent Therefore, $\tilde{f}(a^{n}) = 0$ if and only if $$nc_{0} = 0 ~~~~~~ \text{and} ~~~~~~ \sum_{j=0}^{n-1} a^{sj}c_{1} = 0$$ (as $n$ is odd). Put $c_{1}' = a^{-t}c_{1}$. Then $$\sum_{j=0}^{n-1} a^{sj}c_{1} = 0 ~~~~~~ \text{if and only if} ~~~~~~ \sum_{j=0}^{n-1} a^{sj}c_{1}' = 0.$$ Therefore, $\tilde{f}(a^{n}) = 0$ if and only if $$nc_{0} = 0 ~~~~~~ \text{and} ~~~~~~ \sum_{j=0}^{n-1} a^{sj}c_{1}' = 0.$$ 

Now, we have two possible cases.\vspace{6pt}

\textbf{(i)} $\text{char}(\mathbb{F}) = 0$ or $p$ with $\text{gcd}(n,p) = 1$.

Then $nc_{0} = 0 \Leftrightarrow c_{0} = 0$. 

Since $c_{1}' \in \mathbb{F} \langle a \rangle$, so $$c_{1}' = \sum_{i=1}^{n} \lambda_{i} a^{i}$$ for some $\lambda_{i} \in \mathbb{F}$ ($1 \leq i \leq n$). 

Put $\lambda_{0} = \lambda_{n}$. Since $|\sigma_{0}(a)| = |a^{s}| = m$ and $d = \frac{n}{m}$, therefore,  $$0 = \sum_{j=0}^{n-1} a^{sj}c_{1}' = \sum_{j=1}^{n} a^{sj}c_{1}' = (\sum_{j=1}^{dm}a^{sj})c_{1}' = d(\sum_{j=1}^{m} a^{sj})c_{1}'.$$
Since $d$ divides $n$ and $\text{char}(\mathbb{F}) = 0$ or $p$ with $\text{gcd}(n,p) = 1$, therefore, $$d(\sum_{j=1}^{m} a^{sj})c_{1}' = 0 ~~~~ \text{if and only if} ~~~~ (\sum_{j=1}^{m} a^{sj})c_{1}' = 0.$$ Now, $|a^{s}| = m$ and $|a| = n$, so $n$ must divide $sm$, that is, $sm = nl$ for some $l \in \{1, ..., m-1\}$. This gives $s = dl$. Also, $|a^{d}| = m$. Therefore, 
\begin{align*}
0 & = c_{1}'(\sum_{j=1}^{m} a^{sj}) = (\sum_{i=1}^{n} \lambda_{i} a^{i})(\sum_{j=1}^{m} a^{sj}) = \sum_{i=1}^{n} \lambda_{i} (\sum_{j=1}^{m} a^{sj+i}) = \sum_{i=1}^{md} \lambda_{i} (\sum_{j=1}^{m} a^{dj+i})
\\ & = [\lambda_{1}(\sum_{j=1}^{m} a^{dj+1}) + \lambda_{2} (\sum_{j=1}^{m} a^{dj+2}) + ... + \lambda_{d} (\sum_{j=1}^{m} a^{dj+d})] + [\lambda_{d+1} (\sum_{j=1}^{m} a^{dj+d+1}) \\ &\quad + \lambda_{d+2} (\sum_{j=1}^{m} a^{dj+d+2})  + ... + \lambda_{d+d} (\sum_{j=1}^{m} a^{dj+d+d})] + ... + [\lambda_{(m-1)d+1} (\sum_{j=1}^{m} a^{dj+(m-1)d+1}) \\ &\quad + \lambda_{(m-1)d+2}(\sum_{j=1}^{m} a^{dj+(m-1)d+2}) + ... + \lambda_{(m-1)d+d}(\sum_{j=1}^{m} a^{dj+(m-1)d+d})]
\\ & = [\lambda_{1}(\sum_{j=1}^{m} a^{dj+1}) + \lambda_{2} (\sum_{j=1}^{m} a^{dj+2}) + ... + \lambda_{d} (\sum_{j=1}^{m} a^{dj+d})] + [\lambda_{d+1} (\sum_{j=1}^{m} a^{dj+1}) + \lambda_{d+2} (\sum_{j=1}^{m} a^{dj+2}) \\ &\quad + ... + \lambda_{d+d} (\sum_{j=1}^{m} a^{dj+d})] + ... + [\lambda_{(m-1)d+1} (\sum_{j=1}^{m} a^{dj+1}) + \lambda_{(m-1)d+2}(\sum_{j=1}^{m} a^{dj+2}) \\ &\quad + ... + \lambda_{(m-1)d+d}(\sum_{j=1}^{m} a^{dj+d})]
\\ & = (\sum_{i=0}^{m-1} \lambda_{di+1})(\sum_{j=1}^{m}a^{dj+1}) + (\sum_{i=0}^{m-1} \lambda_{di+2})(\sum_{j=1}^{m}a^{dj+2}) + ... + (\sum_{i=0}^{m-1} \lambda_{di+d})(\sum_{j=1}^{m}a^{dj+d}).
\end{align*}
Therefore, $$(\sum_{j=1}^{m} a^{sj})c_{1}' = 0 ~~~~~~ \text{if and only if} ~~~~~~ \sum_{i=0}^{m-1} \lambda_{di+j} = 0$$ for all $j \in \{1, 2, ..., d\}$. 

Now, using (\ref{eq 4.2}), (\ref{eq 4.3}) and (\ref{eq 4.4}), 
\begin{equation*}
\begin{aligned}
c_{0} + c_{1}b & = f(a) \\ & = \tilde{f}(ab) a^{t}b - a^{s}f(b)a^{t}b
\\ & = \sum_{i=1}^{\frac{n-1}{2}} (\gamma_{i} - \nu_{i})a^{s}(a^{i} - a^{-i}) + a^{t}(\sum_{i=1}^{\frac{n-1}{2}} \delta_{i}(a^{i} - a^{-i}) - \sum_{i=1}^{\frac{n-1}{2}} \mu_{i} a^{s} (a^{i} - a^{-i}))b.
\end{aligned}
\end{equation*}
Therefore, $$c_{0} = \sum_{i=1}^{\frac{n-1}{2}} (\gamma_{i} - \nu_{i})a^{s}(a^{i} - a^{-i}) ~~~~ \text{and} ~~~~ c_{1} = a^{t}(\sum_{i=1}^{\frac{n-1}{2}} \delta_{i}(a^{i} - a^{-i}) - \sum_{i=1}^{\frac{n-1}{2}} \mu_{i} (a^{s+i} - a^{s-i})).$$
Since $c_{0} = 0$, therefore, $\gamma_{i} = \nu_{i}$ for all $i \in \{1, ..., \frac{n-1}{2}\}$. Also, $$c_{1}' = a^{-t}c_{1} = \sum_{i=1}^{\frac{n-1}{2}} \delta_{i}(a^{i} - a^{-i}) - \sum_{i=1}^{\frac{n-1}{2}} \mu_{i} (a^{s+i} - a^{s-i}).$$ Note that $\lambda_{i}$'s belong to the set $\{\delta_{i}, \mu_{i} \mid 1 \leq i \leq \frac{n-1}{2}\}$.

\textbf{Claim:} For each $j \in \{1, 2, ..., d\}$,  the equations $$\sum_{i=0}^{m-1} \lambda_{di+j} = 0 ~~~~~~ \text{and} ~~~~~~ \sum_{i=0}^{m-1} \lambda_{di+(d-j)} = 0$$ are redundant.

\textbf{Proof of the Claim:} First, make the following observations:
\begin{enumerate}
\item[•] For each $j \in \{1, 2, ..., d\}$, $$\sum_{i=0}^{m-1} \lambda_{di+j}$$ is the sum of coefficients $\lambda_{j}, \lambda_{d+j}, ..., \lambda_{(m-1)d+j}$ corresponding to $a^{j}, a^{d+j}, ..., a^{(m-1)d+j}$ respectively.

\item[•] $\sum_{i=0}^{m-1} \lambda_{n-(di+j)}$ is the sum of coefficients $\lambda_{n-j}, \lambda_{n-(d+j)}, ..., \lambda_{n-((m-1)d+j)}$ corresponding to $$a^{n-j} = a^{-j}, ~~~~ a^{n-(d+j)} = a^{-(d+j)}, ~~~~ ..., ~~~~ a^{n-((m-1)d+j)} = a^{-((m-1)d+j)}$$ respectively.

\item[•] Since $n-(di+j) = (m-1-i)d + (d-j)$, therefore, $$\sum_{i=0}^{m-1} \lambda_{n-(di+j)} = \sum_{i=0}^{m-1} \lambda_{di + (d-j)}.$$

\item[•] From the expression $$c_{1}' = \sum_{i=1}^{\frac{n-1}{2}} \delta_{i}(a^{i} - a^{-i}) - \sum_{i=1}^{\frac{n-1}{2}} \mu_{i} (a^{s+i} - a^{s-i}),$$ it is very clear that if $\delta_{i}$ ($1 \leq i \leq \frac{n-1}{2}$) is the coefficient of $a^{di_{0}+j}$ (for some $i_{0} \in \{0, 1, ..., m-1\}$, $j \in \{1, ..., d\}$), then $-\delta_{i}$ is the coefficient of $a^{n-(di_{0}+j)} = a^{-(di_{0}+j)}$.

\item[•] The scalar $- \mu_{i}$ ($1 \leq i \leq \frac{n-1}{2}$) is the coefficient of $a^{s+i}$. Suppose $$s+i \equiv i_{1} ~ (\text{mod} ~ n),$$ where $i_{1} \in \{0, 1, ..., n-1\}$. Then $s+i = i_{1} + nk$ for some non-negative integer $k$. Let $i_{1} = di_{0} + j$ for some $i_{0} \in \{0, 1, ..., m-1\}$, $j \in \{1, ..., d\}$. So $-\mu_{i}$ is the coefficient of $$a^{s+i} = a^{i_{1}} = a^{di_{0}+j}.$$ Also, $\mu_{i}$ is the coefficient of $$a^{s-i} = a^{s+(n-i)}$$ and $$s + (n - i) = n - (di_{0}' + j),$$ where $i_{0}' = mk + i_{0} - 2l$. Now, since $m$ is a positive integer, so by division algorithm, there exists some $i_{0}'' \in \{0, 1, ..., m-1\}$ such that $i_{0}' = mq + i_{0}''$. Then $$a^{s-i} = a^{s+(n-i)} = a^{n - (di_{0}' + j)} = a^{n - (d(mq + i_{0}'') + j)} = a^{n - (di_{0}''+j)}.$$
\end{enumerate}
  
From the above observations, we can conclude that if a scalar (out of $\delta_{i}, - \delta_{i}, \mu_{i}, - \mu_{i}$ for $1 \leq i \leq \frac{n-1}{2}$) appears in the sum $$\sum_{i=0}^{m-1} \lambda_{di+j},$$ then the negative of this scalar will appear in the sum $$\sum_{i=0}^{m-1} \lambda_{n-(di+j)} = \sum_{i=0}^{m-1} \lambda_{di + (d-j)}$$ and vice versa. Therefore, the equations $$\sum_{i=0}^{m-1} \lambda_{di+j} = 0 ~~~~~~ \text{and} ~~~~~~ \sum_{i=0}^{m-1} \lambda_{di + (d-j)} = 0$$ are  redundant. Hence, the claim is proved.

$d$ is odd and $d = 1 + \underbrace{2 + 2 + ... + 2}_{\text{$r_{0}$-times}}$. So if $j \in \{1, ..., d-1\}$, then $d-j \in \{1, ..., d-1\}$, $j \neq d-j$, the sums $$\sum_{i=0}^{m-1} \lambda_{di+j} ~~~~~~ \text{and} ~~~~~~ \sum_{i=0}^{m-1} \lambda_{di + (d-j)}$$ are different and satisfy $$\sum_{i=0}^{m-1} \lambda_{di + (d-j)} = - \sum_{i=0}^{m-1} \lambda_{di+j}$$ so that the equations $$\sum_{i=0}^{m-1} \lambda_{di+j} = 0 ~~~~~~ \text{and} ~~~~~~ \sum_{i=0}^{m-1} \lambda_{di + (d-j)} = 0$$ are redundant. But for $j = d$, the sums $$\sum_{i=0}^{m-1} \lambda_{di+j} ~~~~~~ \text{and} ~~~~~~ \sum_{i=0}^{m-1} \lambda_{di + (d-j)}$$ are same because $$\sum_{i=0}^{m-1} \lambda_{di+j} = \sum_{i=1}^{m} \lambda_{di}, ~~~~ \sum_{i=0}^{m-1} \lambda_{di + (d-j)} = \sum_{i=0}^{m-1} \lambda_{di} ~~~~ \text{and} ~~~~ \sum_{i=1}^{m} \lambda_{di} = \sum_{i=0}^{m-1} \lambda_{di}$$ (as $\lambda_{0} = \lambda_{n}$). Therefore, for $j = d$, there is only one equation which always holds because for each coefficient $\lambda_{dj}$ (corresponding to $a^{dj}$) appearing in the sum $\sum_{i=0}^{m-1} \lambda_{di}$, the negative of this coefficient, namely, $-\lambda_{dj}$ (corresponding to $a^{-dj}$) also appears in the sum $\sum_{i=0}^{m-1} \lambda_{di}$.

So overall, we have $r_{0} = \frac{d-1}{2}$ number of distinct equations each containing $m$ number of distinct coefficients (out of $\delta_{i}, - \delta_{i}, \mu_{i}, - \mu_{i}$ for $1 \leq i \leq \frac{n-1}{2}$). So from each of these $r_{0}$ distinct equations, one coefficient can be expressed linearly in terms of other $m-1$ distinct coefficients present in the equation. Also, no two equations out of these $r_{0}$ distinct equations have any coefficient in common.

Finally, $f(a) = c_{1}b$ (as $c_{0} = 0$), where we have obtained above the conditions on the coefficients appearing in the expression for $c_{1}$ in the form of $r_{0}$ distinct equations. This and the expression for $f(b)$ in (\ref{eq 4.3}) together give the following set.
\begin{equation*}
\begin{aligned}
\mathcal{S}_{0} & = \{D_{(\sigma_{0}; \bar{a}, \bar{b})} \mid (\bar{a}, \bar{b}) \in \{(a^{t}(a^{i}-a^{-i})b, 0),      ~     (a^{s+t}(a^{i} - a^{-i})b, (a^{i} - a^{-i})),       ~      (0, a^{t}(a^{i} - a^{-i})b) \\ &\quad \mid 1 \leq i \leq \frac{n-1}{2}\}\},
\end{aligned}
\end{equation*} where $\bar{a} =  D(a) = f(a) = c_{1}b$ and  $\bar{b} = D(b) = f(b)$.

If $c_{1}, c_{2} \in \mathbb{F}$ and $\alpha_{1}, \alpha_{2}, \beta_{1}, \beta_{2} \in \mathbb{F}D_{2n}$, then 
\begin{eqnarray*}
\left(c_{1}D_{(\sigma_{0}; \alpha_{1}, \beta_{1})} + c_{2}D_{(\sigma_{0}; \alpha_{2}, \beta_{2})}\right)(a) & = & c_{1}D_{(\sigma_{0}; \alpha_{1}, \beta_{1})}(a) + c_{2}D_{(\sigma_{0}; \alpha_{2}, \beta_{2})}(a) \\ & = & c_{1}\alpha_{1} + c_{2}\alpha_{2} \\ & = & D_{(\sigma_{0}; c_{1}\alpha_{1} + c_{2} \alpha_{2}, c_{1}\beta_{1} + c_{2} \beta_{2})}(a).
\end{eqnarray*}
Similarly, $$D_{(\sigma_{0}; c_{1}\alpha_{1} + c_{2} \alpha_{2}, c_{1}\beta_{1} + c_{2} \beta_{2})}(b) = \left(c_{1}D_{(\sigma_{0}; \alpha_{1}, \beta_{1})} + c_{2}D_{(\sigma_{0}; \alpha_{2}, \beta_{2})}\right)(b).$$ Since any $\sigma_{0}$-derivation of $\mathbb{F}D_{2n}$ is completely determined by its image values $D(a)$ and $D(b)$, therefore, from above, we get that $$c_{1}D_{(\sigma_{0}; \alpha_{1}, \beta_{1})} + c_{2}D_{(\sigma_{0}; \alpha_{2}, \beta_{2})} = D_{(\sigma_{0}; c_{1}\alpha_{1} + c_{2} \alpha_{2}, c_{1}\beta_{1} + c_{2} \beta_{2})}.$$ Therefore, $\mathcal{S}_{0}$ spans $\mathcal{D}_{\sigma_{0}}(\mathbb{F}D_{2n})$ over $\mathbb{F}$. Furthermore, in view of the above observations, we have that the dimension of $\mathcal{D}_{\sigma_{0}}(\mathbb{F}D_{2n})$ over $\mathbb{F}$ is $$|\mathcal{S}_{0}| - r_{0} = 3(\frac{n-1}{2}) - (\frac{d-1}{2}) = \frac{1}{2}(3n-d)-1.$$\vspace{6pt}

\textbf{(ii)} $\text{char}(\mathbb{F}) = p$ with $\text{gcd}(n,p) = p$.

Then the equation $nc_{0} = 0$ holds. As computed in part (i), $f(a) = c_{0} + c_{1}b$, where $$c_{0} = \sum_{i=1}^{\frac{n-1}{2}} (\gamma_{i} - \nu_{i})a^{s}(a^{i} - a^{-i}) ~~~~ \text{and} ~~~~ c_{1} = a^{t}(\sum_{i=1}^{\frac{n-1}{2}} \delta_{i}(a^{i} - a^{-i}) - \sum_{i=1}^{\frac{n-1}{2}} \mu_{i} (a^{s+i} - a^{s-i})).$$
In this case, we again have two subcases.\vspace{4pt}

\textbf{(a)} $\text{gcd}(d,p) = 1$.

The proof for this subcase is on the same lines as that of part (i).
Put \begin{equation*}
\begin{aligned}
\mathcal{S}_{0}' & = \{D_{(\sigma_{0}; \bar{a}, \bar{b})} \mid (\bar{a}, \bar{b}) \in \{(a^{s}(a^{i} - a^{-i}), 0),     ~      (a^{s}(a^{i}-a^{-i}), a^{t}(a^{i}-a^{-i})b),     ~     (a^{t}(a^{i}-a^{-i})b, 0),    \\ &\quad     (a^{s+t}(a^{i} - a^{-i})b, (a^{i} - a^{-i})) \mid 1 \leq i \leq \frac{n-1}{2}\}\}.
\end{aligned}
\end{equation*} 
Then, as shown in part (i), it can be shown that $\mathcal{S}_{0}'$ is a spanning set for $\mathcal{D}_{\sigma_{0}}(\mathbb{F}D_{2n})$ over $\mathbb{F}$ so that the dimension of $\mathcal{D}_{\sigma_{0}}(\mathbb{F}D_{2n})$ over $\mathbb{F}$ is $$|\mathcal{S}_{0}'| - r_{0} = 4(\frac{n-1}{2}) - (\frac{d-1}{2}) = 2n - \frac{1}{2}(d+3).$$\vspace{4pt}

\textbf{(b)} $\text{gcd}(d,p) = p$.

Since $$\sum_{j=0}^{n-1} a^{sj}c_{1} = d(\sum_{j=1}^{m} a^{sj})c_{1},$$ $\text{char}(\mathbb{F}) = p$ and $p$ divides $d$, therefore, in this subcase, the equation $$\sum_{j=0}^{n-1} a^{sj}c_{1} = 0$$ holds. Therefore, the set $\mathcal{S}_{0}'$ becomes a basis of $\mathcal{D}_{\sigma_{0}}(\mathbb{F}D_{2n})$ over $\mathbb{F}$. Hence, the dimension of $\mathcal{D}_{\sigma_{0}}(\mathbb{F}D_{2n})$ over $\mathbb{F}$ is $|\mathcal{S}_{0}'| = 4(\frac{n-1}{2})$.

\end{proof}

Taking $\sigma_{0}$ to be the identity endomorphism in \th\ref{theorem 4.3} gives that the dimension of the $\mathbb{F}$-vector space $\mathcal{D}(\mathbb{F}D_{2n})$ of ordinary derivations of $\mathbb{F}D_{2n}$ when $n$ is odd is $3(\frac{n-1}{2})$ if $\text{char}(\mathbb{F}) = 0$ or $p$ with $\text{gcd}(n,p) = 1$ and $2(n-1)$ if $\text{char}(\mathbb{F}) = p$ with $\text{gcd}(n,p) \neq 1$. Further, we also get an $\mathbb{F}$-basis for $\mathcal{D}(\mathbb{F}D_{2n})$.

\begin{theorem}\th\label{theorem 4.4}
Let $n$ be even, and $\mathbb{F}$ be a field of characteristic $0$ or an odd rational prime $p$. Denote $|\sigma_{1}(a)| = m$, $d = \frac{n}{m}$ and $t \equiv j_{0}$ (mod $s$) for some $j_{0} \in \{0, 1, ..., s-1\}$. 
\begin{enumerate}
\item[(i)] Suppose $m$ is odd.
\begin{itemize}
\item[(a)] Suppose $\text{char}(\mathbb{F}) = 0$ or $p$ with $\text{gcd}(n,p) = 1$. Then the dimension of $\mathcal{D}_{\sigma_{1}}(\mathbb{F}D_{2n})$ over $\mathbb{F}$ is $\frac{1}{2}(3n-d) - 2$ and 
\begin{equation*}
\begin{aligned}
\mathcal{S}_{1} & = \{D_{(\sigma_{1}; \bar{a}, \bar{b})} \mid (\bar{a}, \bar{b}) \in \{(a^{t}(a^{i}-a^{-i})b, 0),     ~~      (a^{s+t}(a^{i} - a^{-i})b, (a^{i} - a^{-i})),      \\ &\quad       (0, a^{t}(a^{i} - a^{-i})b) \mid 1 \leq i \leq \frac{n}{2} - 1\}\}
\end{aligned}
\end{equation*} is a spanning set.

\item[(b)] Suppose $\text{char}(\mathbb{F}) = p$ with $\text{gcd}(n,p) \neq 1$.
\begin{itemize}
\item[(b1)] If $\text{gcd}(d,p) = 1$, then the dimension of $\mathcal{D}_{\sigma_{1}}(\mathbb{F}D_{2n})$ over $\mathbb{F}$ is $2n - \frac{d}{2} - 3$ and
\begin{equation*}
\begin{aligned}
\mathcal{S}_{1}' & = \{D_{(\sigma_{1}; \bar{a}, \bar{b})} \mid (\bar{a}, \bar{b}) \in \{(a^{s}(a^{i} - a^{-i}), 0),     ~~      (a^{s}(a^{i}-a^{-i}), a^{t}(a^{i}-a^{-i})b),     \\ &\quad     (a^{t}(a^{i}-a^{-i})b, 0),  ~~  (a^{s+t}(a^{i} - a^{-i})b, (a^{i} - a^{-i})) \mid 1 \leq i \leq \frac{n}{2} - 1\}\}
\end{aligned}
\end{equation*} is a spanning set.
\item[(b2)] If $\text{gcd}(d,p) \neq 1$, then the dimension of $\mathcal{D}_{\sigma_{1}}(\mathbb{F}D_{2n})$ over $\mathbb{F}$ is $2(n-2)$ and the set $\mathcal{S}_{1}'$ is a basis.
\end{itemize}
\end{itemize}

\item[(ii)] Suppose $m$ is even.
\begin{itemize}
\item[(a)] Suppose $\text{char}(\mathbb{F}) = 0$ or $p$ with $\text{gcd}(n,p) = 1$. Then the dimension of $\mathcal{D}_{\sigma_{1}}(\mathbb{F}D_{2n})$ over $\mathbb{F}$ is $\frac{3n}{2} - d - 2$ and $\mathcal{S}_{1}$ is a spanning set.

\item[(b)] Suppose $\text{char}(\mathbb{F}) = p$ with $\text{gcd}(n,p) \neq 1$.
\begin{itemize}
\item[(b1)] If $\text{gcd}(d,p) = 1$, then the dimension of $\mathcal{D}_{\sigma_{1}}(\mathbb{F}D_{2n})$over $\mathbb{F}$ is $2n - d - 3$ and $\mathcal{S}_{1}'$ is a spanning set.
\item[(b2)] If $\text{gcd}(d,p) \neq 1$, then the dimension of $\mathcal{D}_{\sigma_{1}}(\mathbb{F}D_{2n})$ over $\mathbb{F}$ is $2(n-2)$ and the set $\mathcal{S}_{1}'$ is a basis.
\end{itemize}
\end{itemize}
\end{enumerate}
\end{theorem}
\begin{proof}
\textbf{(i)} Let $m$ be odd. 

\textbf{(a)} $\text{char}(\mathbb{F}) = 0$ or $p$ with $\text{gcd}(n,p) = 1$. 

Proceed as in the proof of \th\ref{theorem 4.3}. Since $d$ is even, so $d = \underbrace{2 + ... + 2}_\text{$r_{0}$-times} = 2r_{0}$ for some positive integer $r_{0}$. Using the proof of \th\ref{theorem 4.3}, we have the following observations here: 
\begin{enumerate}
\item[•] If $j \in \{1, ..., d-1\} \setminus \{\frac{d}{2}\}$, then $d-j \in \{1, ..., d-1\}$, $d-j \neq \frac{d}{2}$ and $j \neq d-j$.  For such $j$, we will have two redundant equations $$\sum_{i=0}^{m-1} \lambda_{di+j} = 0$$ (that corresponds to $j$) and $$ \sum_{i=0}^{m-1} \lambda_{di + (d-j)} = 0$$ (that corresponds to $d-j$).

\item[•] For $j = \frac{d}{2}$, $$\sum_{i=0}^{m-1} \lambda_{di + j} = \sum_{i=0}^{m-1} \lambda_{di + (d-j)} = \sum_{i=0}^{m-1} \lambda_{di + \frac{d}{2}}.$$ Therefore, for $j = \frac{d}{2}$, there is only one equation, namely, $$\sum_{i=0}^{m-1} \lambda_{di + \frac{d}{2}} = 0.$$ This equation always holds because if a scalar (out of $\delta_{i}, - \delta_{i}, \mu_{i}, - \mu_{i}$ for $1 \leq i \leq \frac{n}{2} - 1$) appears in the sum $$\sum_{i=0}^{m-1} \lambda_{di+\frac{d}{2}},$$ then the negative of this scalar will appear in the sum $$\sum_{i=0}^{m-1} \lambda_{n-(di+\frac{d}{2})} = \sum_{i=0}^{m-1} \lambda_{di + (d-\frac{d}{2})} = \sum_{i=0}^{m-1} \lambda_{di + \frac{d}{2}}$$ so that this sum is always equal to $0$.

\item[•] Also, as seen in the proof of \th\ref{theorem 4.3}, there is only one equation corresponding to $j = d$ and this too always holds.
\end{enumerate}
  
Therefore, out of $d$ equations, two equations corresponding to $j = \frac{d}{2}$ and $j = d$ always hold. Amongst the remaining $d-2$ values of $j$ ($j \in \{1, ..., d-1\} \setminus \{\frac{d}{2}\}$), there are exactly $\frac{d-2}{2}$ number of distinct equations because the equations corresponding to $j$ and $d-j$ are redundant. So overall, we have $\frac{d-2}{2} = \frac{d}{2} - 1 = r_{0} - 1$ number of distinct equations each containing $m$ number of distinct coefficients (out of $\delta_{i}, - \delta_{i}, \mu_{i}, - \mu_{i}$ for $1 \leq i \leq \frac{n}{2} - 1$) so that one coefficient can be expressed linearly in terms of the remaining $m-1$ distinct coefficients present in the equation. Also, no two equations out of these $r_{0}-1$ distinct equations have any coefficient in common.   
Therefore, the set $\mathcal{S}_{1}$ (given in the theorem's statement) becomes a spanning set for $\mathcal{D}_{\sigma_{1}}(\mathbb{F}D_{2n})$ over $\mathbb{F}$. Also, in view of the above observations, the dimension of $\mathcal{D}_{\sigma_{1}}(\mathbb{F}D_{2n})$ over $\mathbb{F}$ is $$|\mathcal{S}_{1}| - (r_{0}-1) = 3(\frac{n}{2} - 1) - (\frac{d}{2} - 1) = \frac{1}{2}(3n - d) - 2.$$\vspace{6pt}

\textbf{(b)} $\text{char}(\mathbb{F}) = p$ with $\text{gcd}(n,p) = p$.

Then the equation $nc_{0} = 0$ holds. So $f(a) = c_{0} + c_{1}b$, where $$c_{0} = \sum_{i=1}^{\frac{n}{2} - 1} (\gamma_{i} - \nu_{i})a^{s}(a^{i} - a^{-i}) ~~~~ \text{and} ~~~~ c_{1} = a^{t}(\sum_{i=1}^{\frac{n}{2} - 1} \delta_{i}(a^{i} - a^{-i}) - \sum_{i=1}^{\frac{n}{2} - 1} \mu_{i} (a^{s+i} - a^{s-i})).$$ Again, we have two subcases.\vspace{4pt}

\textbf{(b1)} $\text{gcd}(d,p) = 1$.

The proof for this subcase is on the same lines as part (i) (a). Clearly, the set $\mathcal{S}_{1}'$ (given in the theorem's statement) is a spanning set for $\mathcal{D}_{\sigma_{1}}(\mathbb{F}D_{2n})$ over $\mathbb{F}$ and the dimension of $\mathcal{D}_{\sigma_{1}}(\mathbb{F}D_{2n})$ over $\mathbb{F}$ is $$|\mathcal{S}_{1}'| - (r_{0}-1) = 4(\frac{n}{2} - 1) - (\frac{d}{2} - 1) = 2n - \frac{d}{2} - 3.$$\vspace{4pt}

\textbf{(b2)} $\text{gcd}(d,p) = p$.

As seen in the proof of \th\ref{theorem 4.3} (ii) (b), in this subcase, the equation $$\sum_{j=0}^{n-1} a^{sj}c_{1} = 0$$ holds. Then the set $\mathcal{S}_{1}'$ becomes a basis of $\mathcal{D}_{\sigma_{1}}(\mathbb{F}D_{2n})$ over $\mathbb{F}$ so that the dimension of $\mathcal{D}_{\sigma_{1}}(\mathbb{F}D_{2n})$ over $\mathbb{F}$ is $|\mathcal{S}_{1}'| = 2(n-2)$.\vspace{20pt}

\textbf{(ii)} We skip the proof of this part as the proof is similar to that of part (i). All steps are analogous. One observation is that the roles of $s$, $d$, $l$ and $m$ get replaced by that of $2s$, $2d$, $\frac{l}{2}$ and $\frac{m}{2}$ respectively. 
\end{proof}

If we take $\sigma_{1}$ to be the identity endomorphism, then from \th\ref{theorem 4.4} (ii), we get that the dimension of the $\mathbb{F}$-vector space $\mathcal{D}(\mathbb{F}D_{2n})$ of ordinary derivations of $\mathbb{F}D_{2n}$ when $n$ is even is $3(\frac{n}{2} - 1)$ if $\text{char}(\mathbb{F}) = p$ with $\text{gcd}(n,p) = 1$ and $2(n-2)$ if $\text{char}(\mathbb{F}) = p$ with $\text{gcd}(n,p) = p$. In addition, we also obtain a basis for $\mathcal{D}(\mathbb{F}D_{2n})$ over $\mathbb{F}$.

\subsection{$\sigma$-Derivations of $\mathbb{F}D_{2n}$ when $\text{char}(\mathbb{F}) = 2$}\label{subsection 4.2}

In this subsection, we classify all $\sigma$-derivations (for $\sigma = \sigma_{0}$ when $n$ is odd and $\sigma = \sigma_{1}$ when $n$ is even) of $\mathbb{F}D_{2n}$ over a field $\mathbb{F}$ of characteristic $2$. 

First, let us recall the following definition.
\begin{definition}\label{definition 4.5}
Let $G$ be a group and $R$ be a commutative ring with unity. Let $\beta \in RG$. Then the set $$C(\beta) = \{\alpha \in RG \mid \alpha \beta = \beta \alpha\}$$ is called the centralizer of $\beta$ in $RG$.
\end{definition}
Note that $C(\beta)$ is a subalgebra of $RG$ and if $\beta \in Z(RG)$, then $C(\beta) = RG$.

\begin{lemma}\th\label{lemma 4.6}
Let $\mathbb{F}$ be a field of characteristic $2$.
\begin{enumerate}
\item[(i)] Let $n$ be odd. 
\begin{enumerate}
\item[(a)] The set $$\mathcal{B}(\sigma_{0}(b)) = \{1, ~ a^{t}b\} \cup \{(a^{i} + a^{-i}), ~ a^{t}(a^{i} + a^{-i})b \mid 1 \leq i \leq \frac{n-1}{2}\}$$ is a basis of $C(\sigma_{0}(b))$ over $\mathbb{F}$.
\item[(b)] The set $$\mathcal{B}(\sigma_{0}(ab)) = \{1, ~ a^{s+t}b\} \cup \{(a^{i} + a^{-i}), ~ a^{s+t}(a^{i} + a^{-i})b \mid 1 \leq i \leq \frac{n-1}{2}\}$$ is a basis of $C(\sigma_{0}(ab))$ over $\mathbb{F}$.
\end{enumerate}

\item[(ii)] Let $n$ be even.
\begin{enumerate}
\item[(a)] The set $$\mathcal{B}(\sigma_{1}(b)) = \{1, ~ a^{\frac{n}{2}}, ~ a^{t}b, ~ a^{t + \frac{n}{2}}b\} \cup \{(a^{i} + a^{-i}), ~ a^{t}(a^{i} + a^{-i})b \mid 1 \leq i \leq \frac{n}{2} - 1\}$$ is a basis of $C(\sigma_{1}(b))$ over $\mathbb{F}$.
\item[(b)] The set $$\mathcal{B}(\sigma_{1}(ab)) = \{1, ~ a^{\frac{n}{2}}, ~ a^{s+t}b, a^{s+t + \frac{n}{2}}b\} \cup \{(a^{i} + a^{-i}), ~ a^{s+t}(a^{i} + a^{-i})b \mid 1 \leq i \leq \frac{n}{2} - 1\}$$ is a basis of $C(\sigma_{1}(ab))$ over $\mathbb{F}$.
\end{enumerate}
\end{enumerate}
\end{lemma}
\begin{proof}
Similar to the proof of \th\ref{lemma 4.2}.
\end{proof}

Recall the discussion just before \th\ref{theorem 4.3}. Here since $\text{char}(\mathbb{F}) = 2$, so $$\tilde{f}(b^{2}) = f(b)\sigma(b) + \sigma(b)f(b) = 0$$ if and only if $$f(b)\sigma(b) = \sigma(b)f(b), ~~~~ \text{that is}, ~~~~ f(b) \in C(\sigma(b)).$$ Similarly, $\tilde{f}((ab)^{2}) = 0$ if and only if $\tilde{f}(ab) \in C(\sigma(ab))$. Also, (\ref{eq 4.2}) becomes \begin{equation}\label{eq 4.5}
f(a) = \tilde{f}(ab) \sigma(b) + \sigma(a)f(b)\sigma(b).
\end{equation}

\begin{theorem}\th\label{theorem 4.7}
Let $n$ be odd and $\mathbb{F}$ be a field of characteristic $2$. Denote $|\sigma_{0}(a)| = m$, $d = \frac{n}{m}$ and $t \equiv j_{0}$ (mod $s$) for some $j_{0} \in \{0, 1, ..., s-1\}$. Then the dimension of $\mathcal{D}_{\sigma_{0}}(\mathbb{F}D_{2n})$ over $\mathbb{F}$ is $\frac{1}{2}(3n - d) + 1$ and the set 
\begin{equation*}
\begin{aligned}
\mathcal{S}_{0} & = \{D_{(\sigma_{0}; \bar{a}, \bar{b})} \mid (\bar{a}, \bar{b}) \in \{(a^{t}(1+a^{s})b, 1),    ~~    (0, a^{t}b)\} ~ \cup ~ \{(a^{t}(a^{i}+a^{-i})b, 0),     \\ &\quad     (a^{s+t}(a^{i} + a^{-i})b, (a^{i} - a^{-i})),     ~~      (0, a^{t}(a^{i} + a^{-i})b) \mid 1 \leq i \leq \frac{n-1}{2}\}\}
\end{aligned}
\end{equation*} is a spanning set.
\end{theorem}
\begin{proof}
Since $f(b) \in C(\sigma_{0}(b))$ and by \th\ref{lemma 4.6} (i) (a), $\mathcal{B}(\sigma_{0}(b))$ is an $\mathbb{F}$-basis of $C(\sigma_{0}(b))$, therefore, \begin{equation*}
f(b) = \mu_{0} + \nu_{0}a^{t}b + \sum_{i=1}^{\frac{n-1}{2}} \mu_{i} (a^{i} + a^{-i}) + \sum_{i=1}^{\frac{n-1}{2}} \nu_{i} a^{t}(a^{i} + a^{-i})b\end{equation*} for some $\mu_{i}, \nu_{i} \in \mathbb{F}$ ($1 \leq i \leq \frac{n-1}{2}$). 

Further, since $\tilde{f}(ab) \in C(\sigma_{0}(ab))$ and by \th\ref{lemma 4.6} (i) (b), $\mathcal{B}(\sigma_{0}(ab))$ is an $\mathbb{F}$-basis of $C(\sigma(ab))$, therefore, \begin{equation*}
\tilde{f}(ab) = \delta_{0} + \gamma_{0}a^{s+t}b + \sum_{i=1}^{\frac{n-1}{2}} \delta_{i} (a^{i} + a^{-i}) + \sum_{i=1}^{\frac{n-1}{2}} \gamma_{i} a^{s+t}(a^{i} + a^{-i})b\end{equation*} for some $\delta_{i}, \gamma_{i} \in \mathbb{F}$ ($1 \leq i \leq \frac{n-1}{2}$). 

As seen in the proof of \th\ref{theorem 4.3}, $\tilde{f}(a^{n}) = 0$ if and only if $$nc_{0} = 0 ~~~~~~ \text{and} ~~~~~~ \sum_{j=0}^{n-1} a^{sj} c_{1}' = 0,$$ where $c_{1}' = a^{-t}c_{1}$. 

Further, if $$c_{1}' = \sum_{i=1}^{n} \lambda_{i}a^{i}$$ for some $\lambda_{i} \in \mathbb{F}$ ($1 \leq i \leq n$) ($\lambda_{0} = \lambda_{n}$), then $$\sum_{j=0}^{n-1} a^{sj}c_{1}' = 0$$ if and only if $$\sum_{i=0}^{m-1} \lambda_{di+j} = 0 ~~~~ \text{for all} ~~~~ 1 \leq j \leq d.$$ 

By (\ref{eq 4.5}), $$c_{0} + c_{1}b = f(a) = \tilde{f}(ab) a^{t}b + a^{s}f(b)a^{t}b.$$ Therefore, $$c_{0} = (\gamma_{0} + \nu_{0})a^{s} + \sum_{i=1}^{\frac{n-1}{2}} (\gamma_{i} + \nu_{i})a^{s}(a^{i} + a^{-i})$$ and $$c_{1} = a^{t}(\delta_{0} + \mu_{0}a^{s} + \sum_{i=1}^{\frac{n-1}{2}} \delta_{i}(a^{i} + a^{-i}) + \sum_{i=1}^{\frac{n-1}{2}} \mu_{i} (a^{s+i} + a^{s-i})).$$ But since $c_{0} = 0$, so we must have that $\gamma_{i} = \nu_{i}$ for all $0 \leq i \leq \frac{n-1}{2}$. Also, $$c_{1}' = a^{-t}c_{1} = \delta_{0} + \mu_{0}a^{s} + \sum_{i=1}^{\frac{n-1}{2}} \delta_{i}(a^{i} + a^{-i}) + \sum_{i=1}^{\frac{n-1}{2}} \mu_{i} (a^{s+i} + a^{s-i}).$$

Now, proceeding exactly as in the proof of \th\ref{theorem 4.3} (Case (i): $\text{char}(\mathbb{F}) = 0$ or $p$ with $\text{gcd}(n,p)=1$) and using the fact that $\text{char}(\mathbb{F})=2$, we get that if a scalar (out of $\delta_{i}, \mu_{i}$ for $0 \leq i \leq \frac{n-1}{2}$) appears in the sum $$\sum_{i=0}^{m-1} \lambda_{di+j},$$ then that scalar also appears in the sum $$\sum_{i=0}^{m-1} \lambda_{n-(di+j)} = \sum_{i=0}^{m-1} \lambda_{di + (d-j)}$$ and vice versa. Therefore, $$\sum_{i=0}^{m-1} \lambda_{di+j} = \sum_{i=0}^{m-1} \lambda_{di + (d-j)}$$ so that the equations $$\sum_{i=0}^{m-1} \lambda_{di+j} = 0$$ (corresponding to $j$) and $$\sum_{i=0}^{m-1} \lambda_{di + (d-j)} = 0$$ (corresponding to $d-j$) are exactly the same. Therefore, there are precisely $\frac{d-1}{2}$ distinct equations for $j \in \{1, ..., d-1\}$. Further, for $j = d$, there is only one equation which is given by 
\begin{equation*}
\begin{aligned}
& \delta_{0} + \mu_{0} + \sum (\text{coefficients corresponding to  $a^{di}$ that appear in the sums $\sum_{i=1}^{\frac{n-1}{2}} \delta_{i}(a^{i} + a^{-i})$ and} \\ &\quad \sum_{i=1}^{\frac{n-1}{2}} \mu_{i}(a^{s+i} + a^{s-i})) = 0.
\end{aligned}
\end{equation*}
Note that the sum of coefficients corresponding to the powers $a^{di}$ that appear in the sums $$\sum_{i=1}^{\frac{n-1}{2}} \delta_{i}(a^{i} + a^{-i}) ~~~~~~ \text{and} ~~~~~~ \sum_{i=1}^{\frac{n-1}{2}} \mu_{i}(a^{s+i} + a^{s-i})$$ is $0$ as for each coefficient $\lambda_{dj}$ (corresponding to $a^{dj}$) appearing in the sum $$\sum_{i=0}^{m-1} \lambda_{di},$$ the coefficient $\lambda_{dj}$ (corresponding to $a^{-dj}$) also appears in the sum $$\sum_{i=0}^{m-1} \lambda_{n-(di+d)} = \sum_{i=0}^{m-1} \lambda_{di + (d-d)} = \sum_{i=0}^{m-1} \lambda_{di}.$$ Therefore, $\delta_{0} + \mu_{0} = 0$ so that $\delta_{0} = \mu_{0}$. Therefore, for $j=d$, we have only one equation, namely, $\delta_{0} + \mu_{0} = 0$.

So overall, we have $r_{0}= \frac{d-1}{2}$ number of distinct equations each containing $m$ number of distinct coefficients (out of $\delta_{i}, \mu_{i}$ for $1 \leq i \leq \frac{n-1}{2}$) together with the equation $\delta_{0} = \mu_{0}$. Also, no two equations out of these $r_{0}$ distinct equations have any coefficient in common. Further, from each of these $r_{0}$ distinct equations, one coefficient can be expressed linearly in terms of other $m-1$ distinct coefficients present in the equation. Also, since $\delta_{0} = \mu_{0}$, so the expression for $c_{1}$ becomes $$c_{1} = a^{t}(\mu_{0}(1 + a^{s}) + \sum_{i=1}^{\frac{n-1}{2}} \delta_{i}(a^{i} + a^{-i}) + \sum_{i=1}^{\frac{n-1}{2}} \mu_{i} (a^{s+i} + a^{s-i})).$$ Therefore, the set $\mathcal{S}_{0}$ (given in the theorem's statement) becomes a spanning set for \\ $\mathcal{D}_{\sigma_{0}}(\mathbb{F}D_{2n})$ over $\mathbb{F}$. Further, in view of the above observations, we have that the dimension of $\mathcal{D}_{\sigma_{0}}(\mathbb{F}D_{2n})$ over $\mathbb{F}$ is $$|\mathcal{S}_{0}| - r_{0} = 3(\frac{n-1}{2}) + 2 - (\frac{d-1}{2}) = \frac{1}{2}(3n - d) + 1.$$
\end{proof}

From \th\ref{theorem 4.7}, we get that when $n$ is odd and $\text{char}(\mathbb{F}) = 2$, then the dimension of the $\mathbb{F}$-vector space $\mathcal{D}(\mathbb{F}D_{2n})$ of ordinary derivations of $\mathbb{F}D_{2n}$ is $\frac{3n + 1}{2}$ by taking $\sigma_{0}$ to be the identity endomorphism.

\begin{theorem}\th\label{theorem 4.8}
Let $n$ be even and $\mathbb{F}$ be a field of characteristic $2$. Denote $|\sigma_{1}(a)| = m$, $d = \frac{n}{m}$ and $t \equiv j_{0}$ (mod $s$) for some $j_{0} \in \{0, 1, ..., s-1\}$. Then the dimension of $\mathcal{D}_{\sigma_{1}}(\mathbb{F}D_{2n})$ over $\mathbb{F}$ is $2(n+2)$ and a basis is 
\begin{equation*}
\begin{aligned}
\mathcal{B}_{1} & = \{D_{(\sigma_{1}; \bar{a}, \bar{b})} \mid (\bar{a}, \bar{b}) \in \{(a^{s+t}b, 1),   ~~   (a^{s+t+\frac{n}{2}}b, a^{\frac{n}{2}}),     ~~      (a^{s}, a^{t}b),    ~~    (a^{s+ \frac{n}{2}}, a^{t + \frac{n}{2}}b),  \\ &\quad   (a^{t}b,0),      ~~ (a^{t + \frac{n}{2}}b,0),    ~~     (1,0),    ~~     (a^{s + \frac{n}{2}},0)\} ~ \cup ~ \{(a^{s+t}(a^{i}+a^{-i})b,  (a^{i}+a^{-i})),     \\ &\quad       (a^{s}(a^{i}+a^{-i}), a^{t}(a^{i}+a^{-i})b),    ~~     (a^{t}(a^{i}+a^{-i})b,0),      ~~      (a^{s}(a^{i}+a^{-i}),0) \mid 1 \leq i \leq \frac{n}{2}-1\}\}.
\end{aligned}
\end{equation*}
\end{theorem}
\begin{proof}
Since $f(b) \in C(\sigma_{1}(b))$ and by \th\ref{lemma 4.6} (ii) (a), $\mathcal{B}(\sigma_{1}(b))$ is an $\mathbb{F}$-basis of $C(\sigma_{1}(b))$, therefore, $$f(b) = \mu_{-1} + \mu_{0}a^{\frac{n}{2}} + \nu_{-1}a^{t}b + \nu_{0}a^{t + \frac{n}{2}}b + \sum_{i=1}^{\frac{n}{2} - 1} \mu_{i} (a^{i} + a^{-i}) + \sum_{i=1}^{\frac{n}{2}-1} \nu_{i} a^{t}(a^{i} + a^{-i})b$$ for some $\mu_{i}, \nu_{i} \in \mathbb{F}$ ($1 \leq i \leq \frac{n}{2}-1$).

Further, since $\tilde{f}(ab) \in C(\sigma_{1}(ab))$ and by \th\ref{lemma 4.6} (ii) (b), $\mathcal{B}(\sigma_{1}(ab))$ is an $\mathbb{F}$-basis of $C(\sigma_{1}(ab))$, therefore, $$\tilde{f}(ab) = \delta_{-1} + \delta_{0}a^{\frac{n}{2}} + \gamma_{-1}a^{t}b + \gamma_{0}a^{s + t + \frac{n}{2}}b + \sum_{i=1}^{\frac{n}{2} - 1} \delta_{i} (a^{i} + a^{-i}) + \sum_{i=1}^{\frac{n}{2}-1} \gamma_{i} a^{s+t}(a^{i} + a^{-i})b$$ for some $\delta_{i}, \gamma_{i} \in \mathbb{F}$ ($1 \leq i \leq \frac{n}{2}-1$).

Now, in view of (\ref{eq 4.1}), $\tilde{f}(a^{n}) = 0$ if and only if $$nc_{0} = 0 ~~~~~~ \text{and} ~~~~~~ \sum_{j=0}^{n-1}a^{2sj}c_{1} = 0.$$ Since $n$ is even, $nc_{0} = 0$ holds and $$\sum_{j=0}^{n-1}a^{2sj} = (\sum_{j=0}^{n-1}a^{sj})^{2} = n(\sum_{j=0}^{n-1}a^{sj})$$ so that $$\sum_{j=0}^{n-1}a^{2sj}c_{1} = 0$$ holds. Therefore, $\tilde{f}(a^{n}) = 0$ holds. By (\ref{eq 4.5}), $f(a) = \tilde{f}(ab)a^{t}b + a^{s}f(b)a^{t}b$ so that

\begin{equation*}
\begin{aligned}
f(a) & = \delta_{-1}a^{t}b + \delta_{0}a^{t+\frac{n}{2}}b + \gamma_{-1} + \gamma_{0}a^{s + \frac{n}{2}} + \sum_{i=1}^{\frac{n}{2} - 1} \delta_{i} a^{t}(a^{i} + a^{-i})b  + \sum_{i=1}^{\frac{n}{2}-1} \gamma_{i} a^{s}(a^{i} + a^{-i}) \\ &\quad + \mu_{-1}a^{s+t}b + \mu_{0}a^{s+t+ \frac{n}{2}}b + \nu_{-1}a^{s} + \nu_{0}a^{s + \frac{n}{2}} + \sum_{i=1}^{\frac{n}{2} - 1} \mu_{i} a^{s+t}(a^{i} + a^{-i})b \\ &\quad + \sum_{i=1}^{\frac{n}{2}-1} \nu_{i} a^{s}(a^{i} + a^{-i}).
\end{aligned}
\end{equation*}

Therefore, the set $\mathcal{B}_{1}$ (given in the theorem's statement) is a basis of $\mathcal{D}_{\sigma_{1}}(\mathbb{F}D_{2n})$ over $\mathbb{F}$. Hence, the dimension of $\mathcal{D}_{\sigma_{1}}(\mathbb{F}D_{2n})$ over $\mathbb{F}$ is $|\mathcal{B}_{1}| = 2(n+2)$.
\end{proof}

If we take $\sigma_{1}$ to be the identity endomorphism in  \th\ref{theorem 4.8}, then for $n$ even and $\text{char}(\mathbb{F}) = 2$, we get the dimension of the $\mathbb{F}$-vector space of ordinary derivations of $\mathbb{F}D_{2n}$ to be $2(n+2)$.

\subsection{$\sigma$-Conjugacy Classes of $D_{2n}$}\label{subsection 4.3}
In this subsection, we compute $\sigma$-conjugacy classes of $D_{2n}$, which we will need in the next subsection to determine all inner $\sigma$-derivations of $D_{2n}$.

\begin{theorem}\th\label{theorem 4.9}
Let $n$ be odd. Denote $|\sigma_{0}(a)| = m$, $d = \frac{n}{m}$ and $t \equiv j_{0}$ (\text{mod $s$}) for some $j_{0} \in \{0, 1, ..., s-1\}$. Then $D_{2n}$ has precisely $\frac{1}{2}(n+d)+1$ $\sigma_{0}$-conjugacy classes given by $$\{1\}, ~~ \{a^{k}, ~ a^{-k}\} ~ \text{for} ~ 1 \leq k \leq \frac{n-1}{2}, ~~ (a^{j_{0}}b)_{\sigma_{0}}^{D_{2n}} = \{a^{2si}a^{j_{0}}b \mid 0 \leq i \leq m-1\},$$ $$(a^{u}b)_{\sigma_{0}}^{D_{2n}} = \{a^{2si}a^{u}b, ~ a^{2si}a^{2j_{0}-u}b \mid 0 \leq i \leq m-1\} ~ \text{for} ~ j_{0}+1 \leq u \leq j_{0}+ (\frac{d-1}{2}).$$
\end{theorem}

\begin{proof}
Clearly $1_{\sigma_{0}}^{D_{2n}} = \{1\}$. Let $k \in \{1, ..., \frac{n-1}{2}\}$. If $g = a^{i}$, then $\sigma_{0}(g)a^{k}(\sigma_{0}(g))^{-1} = a^{k}$ and if $g = a^{i}b$, then $\sigma_{0}(g)a^{k}(\sigma_{0}(g))^{-1} = a^{-k}$. Therefore, $$(a^{k})_{\sigma_{0}}^{D_{2n}} = \{\sigma_{0}(g)a^{k}(\sigma_{0}(g))^{-1} \mid g \in D_{2n}\} = \{a^{k}, a^{-k}\}.$$

Further, let $u \in \{j_{0}, j_{0}+1, ..., j_{0}+r_{0}\}$, where $r_{0}$ is a positive integer such that \\ $d = 1 + \underbrace{2 + ... + 2}_{\text{$r_{0}$-times}}$ (as $d$ is odd). Then 
\begin{eqnarray*}
(a^{u}b)_{\sigma_{0}}^{D_{2n}} & = & \{a^{2si}a^{u}b, ~ a^{2si}a^{2t-u}b \mid 0 \leq i \leq m-1\} \\ & = & \{a^{2si}a^{u}b, ~ a^{2si}a^{2j_{0}-u}b \mid 0 \leq i \leq m-1\}
\end{eqnarray*}
(as $t \equiv j_{0}$ (mod $s$)). 

Note that for $u = j_{0}$, $a^{2si}a^{j_{0}} = a^{2si}a^{2j_{0}-u}$ for all $0 \leq i \leq m-1$ so that $$(a^{j_{0}}b)_{\sigma_{0}}^{D_{2n}} = \{a^{2si}a^{j_{0}}b \mid 0 \leq i \leq m-1\}.$$ So $|(a^{j_{0}}b)_{\sigma_{0}}^{D_{2n}}| = m$. Now, let $u \in \{j_{0}+1, ..., j_{0}+r_{0}\}$ so that $u = j_{0} + k$ for some $1 \leq k \leq r_{0}$. If possible, suppose that $a^{2si_{1}}a^{u} = a^{2si_{2}}a^{2j_{0}-u}$ for some $i_{1}, i_{2} \in \{0, 1, ..., m-1\}$. 

$\Rightarrow a^{2s(i_{2}-i_{1})} = a^{2(u-j_{0})}$.

$\Rightarrow 2s(i_{2}-i_{1}) \equiv 2(u-j_{0})$ (mod $n$).

$\Rightarrow s(i_{2}-i_{1}) \equiv u-j_{0}$ (mod $n$) as $n$ is odd.

$\Rightarrow a^{s(i_{2}-i_{1})} = a^{(u-j_{0})}$. 

$\Rightarrow a^{s(i_{2}-i_{1})} = a^{k}$ as $u = j_{0}+k$. 

$\Rightarrow a^{k} \in \langle a^{s} \rangle$.

Since $|a^{s}| = m$ and $d = \frac{n}{m}$, so $d$ is the smallest positive integer such that $a^{d} \in \langle a^{s} \rangle$. 
Further, since $r_{0} < d = 1 + 2r_{0}$ and $k \in \{1, ..., r_{0}\}$, so $k < d$. So overall, we have got that $a^{k} \in \langle a^{s} \rangle$ and $k$ is a positive integer with $k < d$. But this contradicts that $d$ is the smallest positive integer such that $a^{d} \in \langle a^{s} \rangle$. Therefore, for $u \in \{j_{0}+1, ..., j_{0}+r_{0}\}$, the elements $a^{2si}a^{j_{0}}b$ and $a^{2sj}a^{2j_{0}-u}b$ ($0 \leq i, j \leq m-1$) of the set $(a^{u}b)_{\sigma_{0}}^{D_{2n}}$ are distinct from each other. Therefore, $|(a^{u}b)_{\sigma_{0}}^{D_{2n}}| = 2m$ for all $u \in \{j_{0}+1, ..., j_{0}+r_{0}\}$. 

Also, these sets $(a^{u}b)_{\sigma_{0}}^{D_{2n}}$'s for $u \in \{j_{0}, j_{0}+1, ..., j_{0}+r_{0}\}$ are pairwise disjoint and $$|1_{\sigma_{0}}^{D_{2n}}| + \sum_{k=1}^{\frac{n-1}{2}} |(a^{k})_{\sigma_{0}}^{D_{2n}}| + |(a^{j_{0}}b)_{\sigma_{0}}^{D_{2n}}| + \sum_{k=1}^{r_{0}} |(a^{j_{0}+k}b)_{\sigma_{0}}^{D_{2n}}| = 1 + 2 (\frac{n-1}{2}) + m + 2mr_{0} = 2n = |D_{2n}|$$ as $m+2mr_{0} = n$. Therefore, we get the result in view of \th\ref{corollary 3.3}. Moreover, the number of these distinct $\sigma_{0}$-conjugacy classes is $$1 + (\frac{n-1}{2}) + 1 + r_{0} = \frac{1}{2}(n+d) + 1$$ as $r_{0} = \frac{d-1}{2}$.
\end{proof}

\begin{theorem}\th\label{theorem 4.10}
Let $n$ be even. Denote $|\sigma_{1}(a)| = m$, $d = \frac{n}{m}$ and $t \equiv j_{0}$ (mod $s$) for some $j_{0} \in \{0, 1, ..., s-1\}$. Then, the following statements hold.
\begin{enumerate}
\item[(i)] If $m$ is odd (so that $d$ is even), then $D_{2n}$ has precisely $\frac{1}{2}(n+d) + 2$ $\sigma_{1}$-conjugacy classes given by $$\{1\}, ~~ \{a^{\frac{n}{2}}\}, ~~ \{a^{k}, a^{-k}\} ~ \text{for} ~ 1 \leq k \leq \frac{n}{2} - 1,$$ $$(a^{u}b)_{\sigma_{1}}^{D_{2n}} = \{a^{2si}a^{u}b \mid 0 \leq i \leq m-1\} ~ \text{for} ~ u = j_{0}, \frac{d}{2},$$ $$(a^{u}b)_{\sigma_{1}}^{D_{2n}} = \{a^{2si}a^{u}b, a^{2si}a^{2j_{0}-u}b \mid 0 \leq i \leq m-1\} ~ \text{for} ~ j_{0}+1 \leq u \leq j_{0} + \frac{d}{2}-1.$$

\item[(ii)] If $m$ is even, then $D_{2n}$ has precisely $\frac{n}{2} + d + 2$ $\sigma_{1}$-conjugacy classes given by $$\{1\}, ~~ \{a^{\frac{n}{2}}\}, ~~ \{a^{k}, a^{-k}\} ~ \text{for} ~ 1 \leq k \leq \frac{n}{2} - 1,$$ $$(a^{u}b)_{\sigma_{1}}^{D_{2n}} = \{a^{2si}a^{u}b \mid 0 \leq i \leq m-1\} ~ \text{for} ~ u = j_{0}, d,$$ $$(a^{u}b)_{\sigma_{1}}^{D_{2n}} = \{a^{2si}a^{u}b, a^{2si}a^{2j_{0}-u}b \mid 0 \leq i \leq m-1\} ~ \text{for} ~ j_{0}+1 \leq u \leq j_{0} + d - 1.$$
\end{enumerate}
\end{theorem}
\begin{proof}
Note that $1_{\sigma_{1}}^{D_{2n}} = \{1\}$ and $(a^{\frac{n}{2}})_{\sigma_{1}}^{D_{2n}} = \{a^{\frac{n}{2}}\}$. Also, $(a^{k})_{\sigma_{1}}^{D_{2n}} = \{a^{k}, a^{-k}\}$ for \\ $1 \leq k \leq \frac{n}{2} - 1$.

\textbf{(i)} We have $$(a^{u}b)_{\sigma_{1}}^{D_{2n}} = \{a^{2si}a^{u}b, ~ a^{2si}a^{2j_{0}-u}b \mid 0 \leq i \leq m-1\}$$ for all $u \in \{j_{0}, j_{0}+1, ..., j_{0}+r_{0}\}$, where $r_{0}$ is a positive integer such that $d = \underbrace{2 + ... + 2}_{\text{$r_{0}$-times}}$ (as $d$ is even). Note that $d$ is the smallest positive integer such that $a^{d} \in \langle a^{s} \rangle = \langle a^{2s} \rangle$.

For $u = j_{0}$, $a^{2si}a^{u}b = a^{2si}a^{2j_{0}-u}b$ for all $0 \leq i \leq m-1$ so that $$(a^{j_{0}}b)_{\sigma_{1}}^{D_{2n}} = \{a^{2si}a^{j_{0}}b \mid 0 \leq i \leq m-1\}.$$ So $|(a^{j_{0}}b)_{\sigma_{1}}^{D_{2n}}| = m$ as $|a^{2s}| = m = |a^{s}|$. Now, let $u \in \{j_{0}+1, ..., j_{0}+r_{0}\}$. Then $u = j_{0} + k$ for some $k \in \{1, ..., r_{0}\}$. Let $i_{1} \in \{0, 1, ..., m-1\}$ be arbitrary but fixed. Then there exists some $i_{2} \in \{0, 1, ..., m-1\}$ such that $$a^{2si_{1}}a^{u}b = a^{2si_{2}}a^{2j_{0}-u}b \Leftrightarrow a^{2s(i_{2}-i_{1})} = a^{2(u-j_{0})} \Leftrightarrow a^{2s(i_{2}-i_{1})} = a^{2k}$$ as $u = j_{0}+k \Leftrightarrow a^{2si_{3}} = a^{2k}$, where $i_{3} = i_{2} - i_{1} \Leftrightarrow a^{2k} \in \langle a^{2s} \rangle$.

As $1 \leq k \leq r_{0}$, so $2 \leq 2k \leq 2r_{0} = d$. So $2k$ is a positive integer $\leq d$. Now, if $a^{2k} \in \langle a^{2s} \rangle$, then since $2k \in \mathbb{N}$ with $2k \leq d$ and $d$ is the smallest positive integer such that $a^{d} \in \langle a^{2s} \rangle$, therefore, the above implications will hold if and only if $2k = d$, that is, $k = \frac{d}{2} = r_{0}$ so that $u = j_{0} + k = j_{0} + r_{0}$. 

Since $a^{2k} \in \langle a^{2s} \rangle$ for $k = r_{0} = \frac{d}{2}$, so in view of the above implications, we will have that for $u = j_{0}+r_{0}$, for every $i_{1} \in \{0, 1, ..., m-1\}$, there will exist some $i_{2} \in \{0, 1, ..., m-1\}$ such that $a^{2si_{1}}a^{u}b = a^{2si_{2}}a^{2j_{0}-u}b$. Therefore, $$\{a^{2si}a^{u}b \mid 0 \leq i \leq m-1\} = \{a^{2si}a^{2j_{0}-u}b \mid 0 \leq i \leq m-1\}$$ so that $$(a^{j_{0}+r_{0}}b)_{\sigma_{1}}^{D_{2n}} = \{a^{2si}a^{j_{0}+r_{0}}b \mid 0 \leq i \leq m-1\}.$$ Then $|(a^{j_{0}+r_{0}}b)_{\sigma_{1}}^{D_{2n}}| = m$. Furthermore, for $2k < d$, that is, for other values of $k$, that is, for $k \in \{1, ..., r_{0}-1\}$, that is, for $u \in \{j_{0}+1, ..., j_{0}+r_{0}-1\}$, we will have that for each $i, j \in \{0, 1, ..., m-1\}$, $a^{2si}a^{u}b$ and $a^{2sj}a^{2j_{0}-u}b$ ($0 \leq i, j \leq m-1$) are distinct elements of $(a^{u}b)_{\sigma_{1}}^{D_{2n}}$. Therefore, $$(a^{u}b)_{\sigma_{1}}^{D_{2n}} = \{a^{2si}a^{u}b, ~ a^{2si}a^{2j_{0}-u}b \mid 0 \leq i \leq m-1\}$$ so that $|(a^{u}b)_{\sigma_{1}}^{D_{2n}}| = 2m$. 

Also, these sets $(a^{u}b)_{\sigma_{1}}^{D_{2n}}$'s for $u \in \{j_{0}, j_{0}+1, ..., j_{0}+r_{0}\}$ are pairwise disjoint. Since $d = 2r_{0}$ and $n = md = 2mr_{0}$, so 
\begin{equation*}
\begin{aligned}
& |1_{\sigma_{1}}^{D_{2n}}| + |(a^{\frac{n}{2}})_{\sigma_{1}}^{D_{2n}}| + \sum_{k=1}^{\frac{n}{2}-1} |(a^{k})_{\sigma_{1}}^{D_{2n}}| + |(a^{j_{0}}b)_{\sigma_{1}}^{D_{2n}}| + \sum_{k=1}^{r_{0}-1}|(a^{j_{0}+k}b)_{\sigma_{1}}^{D_{2n}}| + |(a^{j_{0}+r_{0}}b)_{\sigma_{1}}^{D_{2n}}| \\ & = 1 + 1 + 2(\frac{n}{2} - 1) + m + 2m(r_{0}-1) + m \\ & = 2n = |D_{2n}|.
\end{aligned}
\end{equation*}
Therefore, we get the result in view of \th\ref{corollary 3.3}. Furthermore, the number of these distinct $\sigma_{1}$-conjugacy classes is $$1 + 1 + (\frac{n}{2}-1) + 1 + (r_{0}-1) + 1 = \frac{n}{2} + r_{0} + 2 = \frac{1}{2}(n+d) + 2.$$\vspace{6pt}

\textbf{(ii)} The proof follows on similar lines as that of part (i).
\end{proof}

\subsection{Inner and Outer $\sigma$-Derivations of $\mathbb{F}D_{2n}$}\label{subsection 4.4}
\begin{theorem}\th\label{theorem 4.11}
Let $n$ be odd, and $\mathbb{F}$ be an arbitrary field. Denote $|\sigma_{0}(a)| = m$, $d = \frac{n}{m}$ and $t \equiv j_{0}$ (mod $s$) for some $j_{0} \in \{0, 1, ..., s-1\}$. Then the dimension of $\text{Inn}_{\sigma_{0}}(\mathbb{F}D_{2n})$ over $\mathbb{F}$ is $\frac{1}{2}(3n-d)-1$ and a basis is 
\begin{equation*}
\begin{aligned}
\mathcal{B}_{0} & = \{D_{g} \mid g \in \{a^{k} \mid 1 \leq k \leq \frac{n-1}{2}\} \cup \bigcup_{u=j_{0}}^{j_{0} + (\frac{d-1}{2})} \{a^{2si}a^{u}b \mid 1 \leq i \leq m-1\} \\ &\quad \cup \bigcup_{u=j_{0}+1}^{j_{0} + (\frac{d-1}{2})} \{a^{2si}a^{2j_{0} - u}b \mid 0 \leq i \leq m-1\}\}.
\end{aligned}
\end{equation*}
\end{theorem}
\begin{proof}
By \th\ref{theorem 4.9}, when $n$ is odd, $D_{2n}$ has $r = \frac{1}{2}(n+d)+1$ $\sigma_{0}$-conjugacy classes. Then, by Corollary \ref{corollary 3.13(3)}, the dimension of $\text{Inn}_{\sigma_{0}}(\mathbb{F}D_{2n})$ over $\mathbb{F}$ is $|D_{2n}| - r = \frac{1}{2}(3n-d)-1$ and the set $\mathcal{B}_{0}$ is clearly an $\mathbb{F}$-basis.
\end{proof}

When $n$ is odd, the dimension of the $\mathbb{F}$-vector space $\text{Inn}(\mathbb{F}D_{2n})$ of ordinary inner derivations of $\mathbb{F}D_{2n}$ is $3(\frac{n-1}{2})$ by taking $\sigma_{0}$ to be the identity endomorphism in \th\ref{theorem 4.11}. In addition, we also obtain an $\mathbb{F}$-basis of $\text{Inn}(\mathbb{F}D_{2n})$.

\begin{theorem}\th\label{theorem 4.12}
Let $n$ be even and $\mathbb{F}$ be an arbitrary field. Denote $|\sigma_{1}(a)| = m$, $d = \frac{n}{m}$ and $t \equiv j_{0}$ (mod $s$) for some $j_{0} \in \{0, 1, ..., s-1\}$. Then, the following statements hold.
\begin{enumerate}
\item[(i)] If $m$ is odd, then the dimension of $\text{Inn}_{\sigma_{1}}(\mathbb{F}D_{2n})$ over $\mathbb{F}$ is $\frac{1}{2}(3n-d) - 2$ and a basis is 
\begin{equation*}
\begin{aligned}
\mathcal{B}_{o,1} & = \{D_{g} \mid g \in \{a^{k} \mid 1 \leq k \leq \frac{n}{2} - 1\} \cup \bigcup_{u = j_{0}}^{j_{0} + \frac{d}{2}} \{a^{2si}a^{u}b \mid 1 \leq i \leq m-1\} \\ &\quad \cup \bigcup_{u = j_{0} + 1}^{j_{0} + \frac{d}{2} - 1} \{a^{2si}a^{2j_{0} - u}b \mid 0 \leq i \leq m-1\}\}.
\end{aligned}
\end{equation*}

\item[(ii)] If $m$ is even, then the dimension of $\text{Inn}_{\sigma_{1}}(\mathbb{F}D_{2n})$ over $\mathbb{F}$ is $\frac{3n}{2} - d - 2$ and a basis is 
\begin{equation*}
\begin{aligned}
\mathcal{B}_{e,1} & = \{D_{g} \mid g \in \{a^{k} \mid 1 \leq k \leq \frac{n}{2} - 1\} \cup \bigcup_{u = j_{0}}^{j_{0} + d} \{a^{2si}a^{u}b \mid 1 \leq i \leq m-1\} \\ &\quad \cup \bigcup_{u = j_{0} + 1}^{j_{0} + d - 1} \{a^{si}a^{2j_{0} - u}b \mid 0 \leq i \leq m-1\}\}.
\end{aligned}
\end{equation*}
\end{enumerate} 
\end{theorem}
\begin{proof}
Follows from Corollary \ref{corollary 3.13(3)} and \th\ref{theorem 4.10}.
\end{proof}

Taking $\sigma_{1}$ to be the identity endomorphism in \th\ref{theorem 4.12} gives that when $n$ is even, the dimension of the $\mathbb{F}$-vector space $\text{Inn}(\mathbb{F}D_{2n})$ of ordinary inner derivations of $\mathbb{F}D_{2n}$ is $3(\frac{n}{2} - 1)$. We also obtain an $\mathbb{F}$-basis of $\text{Inn}(\mathbb{F}D_{2n})$.

As a consequence of Remark \ref{remark 2.7(NN)} and the theorems obtained in this section, below is a necessary corollary that answers the $\sigma$-twisted derivation problem for the dihedral group algebra $\mathbb{F}D_{2n}$ over a field $\mathbb{F}$ of arbitrary characteristic.

\begin{corollary}\th\label{corollary 4.13}
Let $\mathbb{F}$ be a field of characteristic $0$ or a rational prime $p$ and $\sigma = \sigma_{0}$ or $\sigma_{1}$.
\begin{enumerate}
\item[(i)] If $\text{char}(\mathbb{F}) = 0$ or odd prime $p$ with $\text{gcd}(n,p)=1$, then all $\sigma$-derivations of $\mathbb{F}D_{2n}$ are inner, that is,  $\text{Inn}_{\sigma}(\mathbb{F}D_{2n}) = \mathcal{D}_{\sigma}(\mathbb{F}D_{2n})$. In other words, $\mathbb{F}D_{2n}$ has no non-zero outer $\sigma$-derivations, that is, $\text{Out}_{\sigma}(\mathbb{F}D_{2n}) = \{0\}$.

\item[(ii)] If $\text{char}(\mathbb{F}) = 2$ or $\text{char}(\mathbb{F}) = p$ with $\text{gcd}(n,p) \neq 1$, then $\text{Inn}_{\sigma}(\mathbb{F}D_{2n}) \subsetneq \mathcal{D}_{\sigma}(\mathbb{F}D_{2n})$. In other words, $\mathbb{F}D_{2n}$ has non-zero outer $\sigma$-derivations, that is, $\text{Out}_{\sigma}(\mathbb{F}D_{2n}) \neq \{0\}$.
\end{enumerate}
\end{corollary}

\section{Application to Coding Theory: IDD Code}\label{section 6}
In {\cite[Example 3.16 and Example 3.17]{Creedon2019}}, the authors showed the applications of derivations to coding theory by constructing the extended binary Golay $[24, 12, 8]$ code and the extended binary quadratic residue 
$[48, 24, 12]$ code as images of derivations of group algebras $\mathbb{F}_{2}C_{24}$ and $\mathbb{F}_{2}C_{48}$ respectively. Here $\mathbb{F}_{2}$ denotes a  characteristic $2$ field of order $2$ and $C_{24}$, $C_{48}$ are cyclic groups of orders $24$, $48$ respectively. For the basics of coding theory, we refer the reader to \cite{dougherty2017algebraic, hill1986first, ling2004coding}.

In this section, we unify the above examples into a theory of Image of Derivation-Derived code, written symbolically as IDD code. Let $R$ be a commutative ring with unity and $G$ be a finite group of order $n$. We give a general construction of a group ring $RG$ code arising from a subset of the image of a twisted derivation. We then construct an equivalent code in $R^{n} (= \underbrace{R \times ... \times R}_{n-\text{times}})$, an $R$-module. As an illustration of the above construction, we give examples by constructing some codes using the twisted derivations of group algebras over finite fields.

\subsection{Construction and Definition}
Let $R$ be a commutative unital ring and $G = \{g_{1}, g_{2}, ..., g_{n}\}$ be a finite group of order $n$ with the given ordering of elements. Let $\sigma, \tau$ be $R$-algebra endomorphisms of the group algebra $RG$. Let $D$ be a $(\sigma, \tau)$-derivation of $RG$.

The range $\text{Im}(D)$ of $D$ is an $R$-submodule of the $R$-module $RG$ with the generating set $D(G) = \{D(g_{i}) \mid 1 \leq i \leq n\}$. It generates an $(n, r)$-code in $RG$, where $r = \text{rank}(\text{Im}(D))$, the rank of the $R$-module $\text{Im}(D)$.

Suppose $S = \{D(g_{i_{1}}), D(g_{i_{2}}), ..., D(g_{i_{s}})\}$ is an $R$-linearly independent subset of $\text{Im}(D)$. Then $S$ generates a submodule of $\text{Im}(D)$ with rank $s$. Thus, $S$ generates an $(n,s)$-code.

\begin{definition}\label{definition 6.1}
Let $D$ be a $(\sigma, \tau)$-twisted derivation of $RG$ and $T = \{g_{k_{1}}, g_{k_{2}}, ..., g_{k_{t}}\}$, $t < n$ be a subset of $G$ such that $D(T)$ is an $R$-linearly independent subset of $RG$. Then, the submodule $W = \langle D(T) \rangle$ generates an $(n,t)$-code and is called an Image of Derivation-Derived Code, written as IDD code. The code $W$ is said to have length $n$ and dimension $t$ which is the rank of the $R$-module $W$.
\end{definition}

\subsection{Equivalent IDD Code in $R^{n}$}
Below, we give the construction of an equivalent code in $R^{n}$.

We use $\underline{\alpha}$ to indicate that $\underline{\alpha}$ is a vector as opposed to an element $\alpha$ of the group ring $RG$. For $\underline{\alpha} = (\lambda_{1}, \lambda_{2}, ..., \lambda_{n}) \in R^{n}$, the mapping $\zeta:R^{n} \rightarrow RG$ defined by $$\zeta(\underline{\alpha}) = \sum_{i=1}^{n}\lambda_{i}g_{i} = \alpha$$ is an element in $RG$ according to the given listing of $G$.

For each $i \in \{1, 2, ..., n\}$, suppose $D(g_{i}) = \sum_{j=1}^{n} \lambda_{ij}g_{j}$.  Then $$\zeta^{-1}(D(g_{i})) = (\lambda_{i1}, \lambda_{i2}, ..., \lambda_{in}).$$

Construct an $n \times n$ matrix $B$ with the first $n$ rows as $$\zeta^{-1}(D(g_{1})), ~ \zeta^{-1}(D(g_{2})), ~ ..., ~ \zeta^{-1}(D(g_{n})).$$ More precisely, $$B = \begin{pmatrix}
\alpha_{11} & \alpha_{12} & \cdots & \alpha_{1n} \\
\alpha_{21} & \alpha_{22} & \cdots & \alpha_{2n} \\
\vdots & \vdots & \ddots & \vdots \\
\alpha_{n1} & \alpha_{n2} & ... & \alpha_{nn} \\
\end{pmatrix}.$$

\begin{theorem}\label{theorem 6.2}
Any rows $i_{1}, i_{2}, ..., i_{t}$ of $B$ are linearly independent if and only if the set $\{D(\theta_{i_{1}}), D(\theta_{i_{2}}), ..., D(\theta_{i_{t}})\}$ is linearly independent.
\end{theorem}

Below, we give the steps of construction:
\begin{enumerate}
\item[(i)] Suppose $W$ is an IDD code. Then $W = \langle S \rangle$, for some $R$-linearly independent subset $S$ of $D(G)$, say, $S = \{D(g_{k_{1}}), D(g_{k_{2}}), ..., D(g_{k_{s}})\}$.
\item[(ii)] Pick $\underline{w} \in R^{s}$ and $\underline{w} = (\lambda_{1}, \lambda_{2}, ..., \lambda_{s})$.
\item[(iii)] Write $\underline{w}$ as $\underline{\alpha}$ in $R^{n}$ with $\lambda_{j}$ in position $k_{j}$ and zero elsewhere.
\item[(iv)] $\underline{\alpha}$ can be mapped to an element in $RG$ by $\alpha = \zeta(\underline{\alpha}) = \sum_{j=1}^{s}\lambda_{j}g_{k_{j}}$ and a codeword $$D(\alpha) = D(\zeta(\underline{\alpha})) = D(\sum_{j=1}^{s}\lambda_{j}g_{k_{j}}) = \sum_{j=1}^{s}\lambda_{j}D(g_{k_{j}})$$ equated with a codeword in $\mathcal{E}$ given by $$\zeta^{-1}(D(\alpha)) = \underline{\alpha}B.$$
\item[(v)] Thus, in view of Theorem \ref{theorem 6.2}, we obtain an $(n,s)$ code $$\mathcal{E} = \{\underline{x}B \mid \underline{x} ~ \text{extends} ~ \underline{w} \in \mathbb{F}^{s} ~ \text{to} ~ \mathbb{F}^{n}\}.$$
\end{enumerate}

\subsection{Examples}
In this subsection, we give some examples by constructing some IDD codes using SAGEMATH, thus illustrating the above construction.

\begin{example}
Consider a finite field $\mathbb{F}_{2^{t}}$ of characteristic $2$ and $G = C_{18} = \langle x \mid x^{18} = 1\rangle = \{1, x, ..., x^{17}\}$, a cyclic group of order $18$.\vspace{30pt}

\textbf{(i)} The map $D:\mathbb{F}_{2^{t}}C_{18} \rightarrow \mathbb{F}_{2^{t}}C_{18}$ defined by $$D(x) = 1 + x + x^2 + x^3 + x^4 + x^5 + x^8 + x^{11}$$ and $$D(x^{k}) = k(\sigma(x))^{k-1}D(x)$$ $(1 \leq k \leq 17$) is a $\sigma$-derivation of $\mathbb{F}_{2^{t}}C_{18}$ for every group endomorphism $\sigma$ of $C_{18}$ (after extending it $\mathbb{F}_{2^{t}}$-linearly to the whole of $\mathbb{F}_{2^{t}}C_{18}$), in view of Theorem \ref{theorem 2.9}.

Below, we take $\sigma = id$ (identity map), and we obtain the following binary IDD codes.
\begin{center}
\begin{longtable}{|c|c|c|c|}
\hline 
\textbf{Basis $S_{i}$} & \thead{\textbf{Code} \\ \textbf{Description:} \\ $[n.k.d]$} & \thead{\textbf{Code} \\ \textbf{Properties}} & \thead{\textbf{Dual Code} \\ \textbf{Description:} \\ $[n,k,d]$} \\ 
\hline \hline
$S_{1} = \{D(x^{i}) \mid i \in \{1, 5, 7, 9, 11, 13, 15, 17\}\}$ & $[18,8,6]$ & \thead{Optimal \& \\ non-LCD} & $[18,10,4]$ \\ 
\hline 
$S_{1} \setminus \{D(x^{13})\}$ & $[18,7,6]$ & non-LCD & $[18,11,3]$ \\ 
\hline 
$S_{2} \setminus \{D(x^{7})\}$ & $[18,6,6]$ & non-LCD & $[18,12,2]$ \\ 
\hline 
$S_{3} \setminus \{D(x^{9})\}$ & $[18,5,6]$ & non-LCD & $[18,13,2]$ \\ 
\hline
$S_{1} = \{D(x^{i}) \mid i \in \{1, 5, 9, 13\}\}$ & $[18,4,8]$ & \thead{Optimal \& \\ non-LCD} & \thead{$[18,14,2]$ \\ (optimal)} \\
\hline 
\end{longtable}
\end{center}

An $8 \times 18$ generator matrix for the binary $[18,8,6]$ optimal non-LCD IDD code generated using basis $S_{1}$ is $$\begin{pmatrix}
1 & 1 & 1 & 1 & 1 & 1 & 0 & 0 & 1 & 0 & 0 & 1 & 0 & 0 & 0 & 0 & 0 & 0 \\
0 & 0 & 0 & 0 & 1 & 1 & 1 & 1 & 1 & 1 & 0 & 0 & 1 & 0 & 0 & 1 & 0 & 0 \\
0 & 0 & 0 & 0 & 0 & 0 & 1 & 1 & 1 & 1 & 1 & 1 & 0 & 0 & 1 & 0 & 0 & 1 \\
0 & 1 & 0 & 0 & 0 & 0 & 0 & 0 & 1 & 1 & 1 & 1 & 1 & 1 & 0 & 0 & 1 & 0 \\
1 & 0 & 0 & 1 & 0 & 0 & 0 & 0 & 0 & 0 & 1 & 1 & 1 & 1 & 1 & 1 & 0 & 0 \\
0 & 0 & 1 & 0 & 0 & 1 & 0 & 0 & 0 & 0 & 0 & 0 & 1 & 1 & 1 & 1 & 1 & 1 \\
1 & 1 & 0 & 0 & 1 & 0 & 0 & 1 & 0 & 0 & 0 & 0 & 0 & 0 & 1 & 1 & 1 & 1 \\
1 & 1 & 1 & 1 & 0 & 0 & 1 & 0 & 0 & 1 & 0 & 0 & 0 & 0 & 0 & 0 & 1 & 1
\end{pmatrix}.$$
\vspace{60pt}

\textbf{(ii)} Similarly, the map $D_{2}:\mathbb{F}_{2^{t}}C_{18} \rightarrow \mathbb{F}_{2^{t}}C_{18}$ defined by $$D(x) = 1 + x + x^3 + x^4 + x^7 + x^9 + x^{13} + x^{16} + x^{17}$$ will be a $\sigma$-derivation of $\mathbb{F}_{2^{t}}C_{18}$ for every group endomorphism $\sigma$ of $C_{18}$.

Consider $\sigma$ defined by $\sigma(x) = x^{2}$. Then, we have the following table of binary IDD codes.

\begin{center}
\begin{longtable}{|c|c|c|c|}
\hline 
\textbf{Basis $S_{i}$} & \thead{\textbf{Code} \\ \textbf{Description:} \\ $[n.k.d]$} & \thead{\textbf{Code} \\ \textbf{Properties}} & \thead{\textbf{Dual Code} \\ \textbf{Description:} \\ $[n,k,d]$} \\ 
\hline \hline
$S_{1} = \{D(x^{i}) \mid i \in \{1, 3, 5, 7, 9, 11, 13, 15, 17\}\}$ & $[18,9,5]$ & non-LCD & $[18,9,5]$ \\ 
\hline
$S_{2} = S_{1} \setminus \{D(x^{15})\}$ & $[18,8,5]$ & non-LCD & $[18,10,3]$ \\ 
\hline 
$S_{3} = S_{2} \setminus \{D(x^{11})\}$ & $[18,7,5]$ & LCD & $[18,11,3]$ \\ 
\hline 
$S_{4} = S_{3} \setminus \{D(x^{3})\}$ & $[18,6,5]$ & non-LCD & $[18,12,3]$ \\ 
\hline 
$S_{5} = S_{4} \setminus \{D(x^{7})\}$ & $[18,5,5]$ & LCD & \thead{$[18,13,3]$ \\ (optimal)} \\ 
\hline 
$S_{6} = S_{5} \setminus \{D(x^{17})\}$ & $[18,4,8]$ & \thead{Optimal \& \\ non-LCD} & \thead{$[18,14,2]$ \\ (optimal)} \\ 
\hline 
$S_{7} = S_{4} \setminus \{D(x^{17})\}$ & $[18,5,6]$ & non-LCD & $[18,13,2]$ \\ 
\hline 
$S_{8} = S_{7} \setminus \{D(x^{13})\}$ & $[18,4,7]$ & non-LCD & \thead{$[18,14,2]$ \\ (optimal)} \\ 
\hline 
$S_{9} = S_{6} \setminus \{D(x^{13})\}$  & $[18,3,9]$ & LCD & $[18,15,1]$ \\ 
\hline 
$S_{10} = S_{8} \setminus \{D(x)\}$ or $S_{8} \setminus \{D(x^{5})\}$ & $[18,3,8]$ & LCD & $[18,15,1]$ \\ 
\hline 
$S_{11} = S_{8} \setminus \{D(x^{9})\}$ & $[18,3,7]$ & LCD & $[18,15,1]$ \\ 
\hline 
$S_{12} = S_{11} \setminus \{D(x^{7})\}$ & $[18,2,9]$ & LCD & $[18,16,1]$ \\ 
\hline 
\end{longtable} 
\end{center}

A $9 \times 18$ generator matrix for the binary $[18,9,5]$ non-LCD IDD code generated using basis $S_{1}$ is $$\begin{pmatrix}
1 & 1 & 0 & 1 & 1 & 0 & 0 & 1 & 0 & 1 & 0 & 0 & 0 & 1 & 0 & 0 & 1 & 1 \\
0 & 0 & 1 & 1 & 1 & 1 & 0 & 1 & 1 & 0 & 0 & 1 & 0 & 1 & 0 & 0 & 0 & 1 \\
0 & 0 & 0 & 1 & 0 & 0 & 1 & 1 & 1 & 1 & 0 & 1 & 1 & 0 & 0 & 1 & 0 & 1 \\
0 & 1 & 0 & 1 & 0 & 0 & 0 & 1 & 0 & 0 & 1 & 1 & 1 & 1 & 0 & 1 & 1 & 0 \\
0 & 1 & 1 & 0 & 0 & 1 & 0 & 1 & 0 & 0 & 0 & 1 & 0 & 0 & 1 & 1 & 1 & 1 \\
1 & 1 & 1 & 1 & 0 & 1 & 1 & 0 & 0 & 1 & 0 & 1 & 0 & 0 & 0 & 1 & 0 & 0 \\
0 & 1 & 0 & 0 & 1 & 1 & 1 & 1 & 0 & 1 & 1 & 0 & 0 & 1 & 0 & 1 & 0 & 0 \\
0 & 1 & 0 & 0 & 0 & 1 & 0 & 0 & 1 & 1 & 1 & 1 & 0 & 1 & 1 & 0 & 0 & 1 \\
1 & 0 & 0 & 1 & 0 & 1 & 0 & 0 & 0 & 1 & 0 & 0 & 1 & 1 & 1 & 1 & 0 & 1
\end{pmatrix}.$$
\end{example}\vspace{20pt}

\begin{example}
Consider a finite field $\mathbb{F}_{2^{t}}$ of characteristic $2$ and $G = C_{14} = \langle x \mid x^{14} = 1\rangle = \{1, x, ..., x^{13}\}$, a cyclic group of order $14$.

Each of the following maps $D_{i}$ define a $\sigma$-derivation of $\mathbb{F}_{2^{t}}C_{14}$, in view of Theorem \ref{theorem 2.9}, for every endomorphism $\sigma$ of $C_{14}$.

Below we take $\sigma = id$ (the identity map). The table below contains some IDD codes of different parameters and properties using different derivations of the group ring $\mathbb{F}_{2^{t}}C_{14}$.

\begin{center}
\begin{longtable}{|c|c|c|c|}
\hline 
\textbf{$D_{i}(x)$ with Basis} & \thead{\textbf{Code} \\ \textbf{Description:} \\ $[n,k,d]$} & \thead{\textbf{Other} \\ \textbf{Properties} \\ \textbf{of Code}} & \thead{\textbf{Dual Code} \\ \textbf{Description:} \\ $[n,k,d]$} \\ 
\hline \hline
\thead{$1 + x + x^2 + x^3 + x^4 + x^6 + x^9$ \\ $\text{Basis} = \{D(x^{i}) \mid i \in \{1, 3, 5, 7, 9, 11, 13\}\}$} & $[14,7,4]$ & \thead{Optimal \& \\ LCD} & \thead{$[14,7,4]$ \\ (optimal)} \\ 
\hline 
\thead{$1 + x + x^{2} + x^{3} + x^{4} + x^{9}$ \\ $\text{Basis} = \{D(x^{i}) \mid i \in \{1, 3, 5, 7, 9, 11, 13\}\}$} & $[14,7,4]$  & \thead{Optimal \& \\ non-LCD} & \thead{$[14,7,4]$ \\ (optimal)} \\ 
\hline 
\thead{$1 + x + x^{2} + x^{3} + x^{5} + x^{8} + x^{11}$ \\ $\text{Basis} = \{D(x^{i}) \mid i \in \{1, 3, 5, 7, 9, 11, 13\}\}$} & $[14,7,3]$ & LCD & $[14,7,3]$ \\ 
\hline 
\thead{$1 + x + x^{3} + x^{4} + x^{5} + x^{6} + x^{9}$ \\ $\text{Basis} = \{D(x^{i}) \mid i \in \{1, 3, 5, 7\}\}$} & $[14,4,7]$ & \thead{Optimal \& \\ non-LCD} & \thead{$[14,10,3]$ \\ (optimal)} \\ 
\hline 
\thead{$1 + x + x^{2} + x^{4} + x^{5} + x^{8}$ \\ $\text{Basis} = \{D(x^{i}) \mid i \in \{1, 3, 5, 7, 9, 11\}\}$} & $[14,6,4]$ & LCD & $[14,8,3]$ \\ 
\hline 
\thead{$1 + x + x^{2} + x^{4} + x^{5} + x^{7} + x^{10}$ \\ $\text{Basis} = D(x^{i}) \mid \{i \in \{1, 3, 5, 7, 9, 11, 13\}\}$} & $[14,7,3]$ & non-LCD & $[14,7,3]$ \\ 
\hline 
\end{longtable} 
\end{center}

A $7 \times 14$ generator matrix for the binary $[14,7,4]$ optimal LCD IDD code obtained using $D_{1}$ is $$\begin{pmatrix}
1 & 1 & 1 & 1 & 1 & 0 & 1 & 0 & 0 & 1 & 0 & 0 & 0 & 0 \\
0 & 0 & 1 & 1 & 1 & 1 & 1 & 0 & 1 & 0 & 0 & 1 & 0 & 0 \\
0 & 0 & 0 & 0 & 1 & 1 & 1 & 1 & 1 & 0 & 1 & 0 & 0 & 1 \\
0 & 1 & 0 & 0 & 0 & 0 & 1 & 1 & 1 & 1 & 1 & 0 & 1 & 0 \\
1 & 0 & 0 & 1 & 0 & 0 & 0 & 0 & 1 & 1 & 1 & 1 & 1 & 0 \\
1 & 0 & 1 & 0 & 0 & 1 & 0 & 0 & 0 & 0 & 1 & 1 & 1 & 1 \\
1 & 1 & 1 & 0 & 1 & 0 & 0 & 1 & 0 & 0 & 0 & 0 & 1 & 1
\end{pmatrix}.$$

Further, a $7 \times 14$ generator matrix for the binary $[14,7,3]$ LCD IDD code obtained using $D_{3}$ is $$\begin{pmatrix}
1 & 1 & 1 & 1 & 0 & 1 & 0 & 0 & 1 & 0 & 0 & 1 & 0 & 0 \\
0 & 0 & 1 & 1 & 1 & 1 & 0 & 1 & 0 & 0 & 1 & 0 & 0 & 1 \\
0 & 1 & 0 & 0 & 1 & 1 & 1 & 1 & 0 & 1 & 0 & 0 & 1 & 0 \\
1 & 0 & 0 & 1 & 0 & 0 & 1 & 1 & 1 & 1 & 0 & 1 & 0 & 0 \\
0 & 0 & 1 & 0 & 0 & 1 & 0 & 0 & 1 & 1 & 1 & 1 & 0 & 1 \\
0 & 1 & 0 & 0 & 1 & 0 & 0 & 1 & 0 & 0 & 1 & 1 & 1 & 1 \\
1 & 1 & 0 & 1 & 0 & 0 & 1 & 0 & 0 & 1 & 0 & 0 & 1 & 1
\end{pmatrix}.$$
\vspace{40pt}

\textbf{(i)} Consider $D_{1}$. We obtain the following binary IDD codes.

\begin{center}
\begin{longtable}{|c|c|c|c|}
\hline 
\textbf{Basis $S_{i,1}$} & \thead{\textbf{Code} \\ \textbf{Description:} \\ $[n,k,d]$} & \thead{\textbf{Other} \\ \textbf{Properties} \\ \textbf{of Code}} & \thead{\textbf{Dual Code} \\ \textbf{Description:} \\ $[n,k,d]$} \\ 
\hline \hline
$\{D(x^{i}) \mid i \in \{1, 3, 5, 7, 11, 13\}\}$ & $[14,6,4]$ & LCD & $[14,8,3]$ \\ 
\hline
\thead{$\{D(x^{i}) \mid i \in \{1, 3, 5, 7, 13\}\}$, \\ $\{D(x^{i}) \mid i \in \{1, 3, 5, 11, 13\}\}$} & $[14,5,5]$ & LCD & $[14,9,3]$ \\ 
\hline
\thead{$\{D(x^{i}) \mid i \in \{1, 3, 5, 7\}\}$, \\ $\{D(x^{i}) \mid i \in \{1, 3, 5, 13\}\}$} & $[14,4,6]$ & non-LCD & $[14,10,2]$ \\ 
\hline
$\{D(x^{i}) \mid i \in \{1, 5, 13\}\}$ & $[14,3,6]$ & non-LCD & $[14,11,1]$ \\ 
\hline
$\{D(x^{i}) \mid i \in \{1, 3, 5\}\}$ & $[14,3,6]$ & LCD & $[14,11,1]$ \\ 
\hline
$\{D(x^{i}) \mid i \in \{1, 3, 7, 11\}\}$ & $[14,4,4]$ & non-LCD & $[14,10,2]$ \\ 
\hline
$\{D(x^{i}) \mid i \in \{1, 3, 9, 13\}\}$ & $[14,4,5]$ & LCD & $[14,10,2]$ \\ 
\hline
$\{D(x^{i}) \mid i \in \{1, 5\}\}$ & $[14,2,7]$ & LCD & $[14,12,1]$ \\ 
\hline
\end{longtable} 
\end{center}

\textbf{(ii)} Consider $D_{3}$. We obtain the following binary IDD codes.

\begin{center}
\begin{longtable}{|c|c|c|c|}
\hline 
\textbf{Basis $S_{i,3}$} & \thead{\textbf{Code} \\ \textbf{Description:} \\ $[n,k,d]$} & \thead{\textbf{Other} \\ \textbf{Properties} \\ \textbf{of Code}} & \thead{\textbf{Dual Code} \\ \textbf{Description:} \\ $[n,k,d]$} \\ 
\hline \hline
$\{D(x^{i}) \mid i \in \{1, 3, 5, 7, 9, 13\}\}$ & $[14,6,3]$ & LCD & $[14,8,2]$ \\ 
\hline
$\{D(x^{i}) \mid i \in \{1, 3, 5, 9, 13\}\}$ & $[14,5,3]$ & LCD & $[14,9,2]$ \\ 
\hline
$\{D(x^{i}) \mid i \in \{1, 3, 5\}\}$ & $[14,3,7]$ & LCD & \thead{$[14,11,2]$ \\ (optimal)} \\ 
\hline
$\{D(x^{i}) \mid i \in \{1, 3, 5, 9\}\}$ & $[14,4,4]$ & LCD & $[14,10,2]$ \\ 
\hline
$\{D(x^{i}) \mid i \in \{1, 3, 11, 13\}\}$ & $[14,4,3]$ & LCD & $[14,10,2]$ \\ 
\hline
$\{D(x^{i}) \mid i \in \{1, 5\}\}$ & $[14,2,7]$ & LCD & $[14,12,1]$ \\ 
\hline
\end{longtable} 
\end{center}
\end{example}\vspace{20pt}

\begin{example}
Consider a finite field $\mathbb{F}_{3^{t}}$ of characteristic $3$ and $G = C_{24} = \langle x \mid x^{24} = 1\rangle = \{1, x, ..., x^{23}\}$, a cyclic group of order $24$.

The map $D:\mathbb{F}_{3^{t}}C_{24} \rightarrow \mathbb{F}_{3^{t}}C_{24}$ defined by $$D(x) = 1 + x + x^3 + x^4 + x^5 + x^7 + x^9 + x^{12} + x^{14}$$ and $$D(x^{k}) = k(\sigma(x))^{k-1}D(x)$$ $(1 \leq k \leq 23$) is a $\sigma$-derivation of $\mathbb{F}_{3^{t}}C_{24}$ for every group endomorphism $\sigma$ of $C_{24}$ (after extending it $\mathbb{F}_{3^{t}}$-linearly to the whole of $\mathbb{F}_{3^{t}}C_{24}$), in view of Theorem \ref{theorem 2.9}.

Below, we take $\sigma(x) = x^5$. We obtain the following ternary IDD codes.

\footnotesize
\begin{center}
\begin{longtable}{|c|c|c|c|}
\hline 
\textbf{Basis $S_{i}$} & \thead{\textbf{Code} \\ \textbf{Description:} \\ $[n,k,d]$} & \thead{\textbf{Other} \\ \textbf{Properties} \\ \textbf{of Code}} & \thead{\textbf{Dual Code} \\ \textbf{Description:} \\ $[n,k,d]$} \\ 
\hline \hline
$S_{1} = \{D(x^{i}) \mid 1 \leq i \leq 23, ~ i \notin \{3, 6, 9, 12, 15, 18, 21\}\}$ & $[24,16,3]$ & LCD & $[24,8,7]$ \\ 
\hline
$S_{2} = S_{1} \setminus \{D(x)\}$ & $[24,15,4]$ & LCD & $[24,9,7]$ \\ 
\hline
$S_{3} = S_{1} \setminus \{D(x^{2})\}$ & $[24,15,3]$ & LCD & $[24,9,7]$ \\ 
\hline
$S_{4} = S_{1} \setminus \{D(x), D(x^{2})\}$ & $[24,14,4]$ & non-LCD & $[24,10,6]$ \\ 
\hline
$S_{5} = S_{1} \setminus \{D(x), D(x^{4})\}$ & $[24,14,4]$ & LCD & $[24,10,7]$ \\ 
\hline
$S_{6} = S_{1} \setminus \{D(x), D(x^{7})\}$ & $[24,14,5]$ & LCD & $[24,10,7]$ \\ 
\hline
$S_{7} = S_{1} \setminus \{D(x), D(x^{19})\}$ & $[24,14,5]$ & non-LCD & $[24,10,7]$ \\ 
\hline
$S_{8} = S_{1} \setminus \{D(x), D(x^{2}), D(x^{7})\}$ & $[24,13,5]$ & non-LCD & $[24,11,5]$ \\ 
\hline
$S_{9} = S_{1} \setminus \{D(x), D(x^{4}), D(x^{7})\}$ & $[24,13,5]$ & non-LCD & $[24,11,7]$ \\ 
\hline
$S_{10} = S_{1} \setminus \{D(x^{7}), D(x^{11}), D(x^{14})\}$ & $[24,13,5]$ & LCD & $[24,11,6]$ \\ 
\hline
$S_{11} = S_{1} \setminus \{D(x), D(x^{4}), D(x^{7}), D(x^{11})\}$ & $[24,12,6]$ & non-LCD & $[24,12,6]$ \\ 
\hline
$S_{12} = S_{1} \setminus \{D(x), D(x^{4}), D(x^{7}), D(x^{14})\}$ & $[24,12,6]$ & LCD & $[24,12,6]$ \\ 
\hline
$S_{13} = S_{1} \setminus \{D(x), D(x^{4}), D(x^{7}), D(x^{23})\}$ & $[24,12,5]$ & LCD & $[24,12,5]$ \\ 
\hline
$S_{14} = S_{1} \setminus \{D(x^{4}), D(x^{7}), D(x^{16}), D(x^{23})\}$ & $[24,12,4]$ & LCD & $[24,12,4]$ \\ 
\hline
$S_{15} = S_{1} \setminus \{D(x), D(x^{2}), D(x^{4}), D(x^{7}), D(x^{14})\}$ & $[24,11,7]$ & non-LCD & $[24,13,5]$ \\ 
\hline
$S_{16} = S_{1} \setminus \{D(x), D(x^{4}), D(x^{7}), D(x^{14}), D(x^{23})\}$ & $[24,11,6]$ & LCD & $[24,13,5]$ \\ 
\hline
$S_{17} = S_{1} \setminus \{D(x), D(x^{2}), D(x^{4}), D(x^{5}), D(x^{7}), D(x^{14})\}$ & $[24,10,7]$ & LCD & $[24,14,5]$ \\ 
\hline
$S_{18} = \{D(x^{i}) \mid i \in \{8, 10, 11, 16, 17, 19, 20, 22, 23\}\}$ & $[24,9,8]$ & LCD & $[24,15,3]$ \\ 
\hline
$S_{19} = \{D(x^{i}) \mid i \in \{8, 10, 11, 13, 16, 19, 20, 22, 23\}\}$ & $[24,9,7]$ & LCD & $[24,15,3]$ \\ 
\hline
$S_{20} = \{D(x^{i}) \mid i \in \{8, 10, 11, 16, 17, 19, 20, 23\}\}$ & $[24,8,8]$ & LCD & $[24,16,3]$ \\ 
\hline
$S_{21} =  \{D(x^{i}) \mid i \in \{8, 11, 16, 17, 19, 20, 23\}\}$ & $[24,7,9]$ & LCD & $[24,17,2]$ \\ 
\hline
$S_{22} = \{D(x^{i}) \mid i \in \{11, 16, 17, 19, 20, 23\}\}$ & $[24,6,9]$ & LCD & $[24,18,2]$ \\ 
\hline
$S_{23} = \{D(x^{i}) \mid i \in \{11, 16, 17, 19,  23\}\}$ & $[24,5,9]$ & LCD & $[24,19,1]$ \\ 
\hline
$S_{43} = \{D(x^{i}) \mid i \in \{16, 17, 19,  23\}\}$ & $[24,4,9]$ & LCD & $[24,20,1]$ \\ 
\hline
\end{longtable} 
\end{center}

A $16 \times 24$ generator matrix for the ternary $[24,16,3]$ LCD IDD code obtained using the basis $S_{1}$ is $$\begin{pmatrix}
1 & 1 & 0 & 1 & 1 & 1 & 0 & 1 & 0 & 1 & 0 & 0 & 1 & 0 & 0 & 0 & 0 & 0 & 0 & 0 & 0 & 0 & 0 & 0 \\
0 & 0 & 0 & 0 & 0 & 2 & 2 & 0 & 2 & 2 & 2 & 0 & 2 & 0 & 2 & 0 & 0 & 2 & 0 & 2 & 0 & 0 & 0 & 0 \\
1 & 0 & 0 & 1 & 0 & 1 & 0 & 0 & 0 & 0 & 0 & 0 & 0 & 0 & 0 & 1 & 1 & 0 & 1 & 1 & 1 & 0 & 1 & 0 \\
2 & 2 & 0 & 2 & 0 & 2 & 0 & 0 & 2 & 0 & 2 & 0 & 0 & 0 & 0 & 0 & 0 & 0 & 0 & 0 & 2 & 2 & 0 & 2 \\
0 & 0 & 0 & 0 & 0 & 0 & 1 & 1 & 0 & 1 & 1 & 1 & 0 & 1 & 0 & 1 & 0 & 0 & 1 & 0 & 1 & 0 & 0 & 0 \\
0 & 2 & 0 & 0 & 0 & 0 & 0 & 0 & 0 & 0 & 0 & 2 & 2 & 0 & 2 & 2 & 2 & 0 & 2 & 0 & 2 & 0 & 0 & 2 \\
1 & 1 & 1 & 0 & 1 & 0 & 1 & 0 & 0 & 1 & 0 & 1 & 0 & 0 & 0 & 0 & 0 & 0 & 0 & 0 & 0 & 1 & 1 & 0 \\
0 & 0 & 2 & 2 & 0 & 2 & 2 & 2 & 0 & 2 & 0 & 2 & 0 & 0 & 2 & 0 & 2 & 0 & 0 & 0 & 0 & 0 & 0 & 0 \\
1 & 0 & 1 & 0 & 0 & 0 & 0 & 0 & 0 & 0 & 0 & 0 & 1 & 1 & 0 & 1 & 1 & 1 & 0 & 1 & 0 & 1 & 0 & 0 \\
2 & 0 & 2 & 0 & 0 & 2 & 0 & 2 & 0 & 0 & 0 & 0 & 0 & 0 & 0 & 0 & 0 & 2 & 2 & 0 & 2 & 2 & 2 & 0 \\
0 & 0 & 0 & 1 & 1 & 0 & 1 & 1 & 1 & 0 & 1 & 0 & 1 & 0 & 0 & 1 & 0 & 1 & 0 & 0 & 0 & 0 & 0 & 0 \\
0 & 0 & 0 & 0 & 0 & 0 & 0 & 0 & 2 & 2 & 0 & 2 & 2 & 2 & 0 & 2 & 0 & 2 & 0 & 0 & 2 & 0 & 2 & 0 \\
0 & 1 & 0 & 1 & 0 & 0 & 1 & 0 & 1 & 0 & 0 & 0 & 0 & 0 & 0 & 0 & 0 & 0 & 1 & 1 & 0 & 1 & 1 & 1 \\
2 & 0 & 2 & 2 & 2 & 0 & 2 & 0 & 2 & 0 & 0 & 2 & 0 & 2 & 0 & 0 & 0 & 0 & 0 & 0 & 0 & 0 & 0 & 2 \\
0 & 0 & 0 & 0 & 0 & 0 & 0 & 0 & 0 & 1 & 1 & 0 & 1 & 1 & 1 & 0 & 1 & 0 & 1 & 0 & 0 & 1 & 0 & 1 \\
0 & 0 & 2 & 0 & 2 & 0 & 0 & 0 & 0 & 0 & 0 & 0 & 0 & 0 & 2 & 2 & 0 & 2 & 2 & 2 & 0 & 2 & 0 & 2
\end{pmatrix}.$$
\end{example}
\vspace{20pt}

\begin{example}
Consider a finite field $\mathbb{F}_{2^{t}}$ of characteristic $2$ and $$G = D_{12} = \langle a, b \mid a^{6} = 1 = b^{2}, b^{-1}ab = a^{-1}\rangle = \{1, a, a^{2}, a^{3}, a^{4}, a^{5}, b, ab, a^{2}b, a^{3}b, a^{4}b, a^{5}b\},$$ dihedral group of order $12$.

By Theorem \ref{theorem 4.8}, the map $D:\mathbb{F}_{2^{t}}D_{12} \rightarrow \mathbb{F}_{2^{t}}D_{12}$ defined by $$D(a) = 1 + a + a^{3} + a^{4} + ab + a^{2}b + a^{4}b + a^{5}b$$ and $$D(b) = a + a^{2} + a^{4} + a^{5} + b + a^{2}b + a^{3}b + a^{5}b$$ is a $\sigma_{1}$-derivations of $D_{2n}$ for the group endomorphism $\sigma_{1}$ of $D_{12}$ defined by $\sigma_{1}(a) = a^{2}$ and $\sigma_{1}(b) = ab$.

Each of the following sets 
\begin{center}
\begin{tabular}{|c|}
\hline 
\textbf{Basis $S_{i}$} \\ 
\hline \hline
$S_{1} = \{D(a), D(a^{2}), D(a^{3}), D(b)\}$ \\ 
\hline 
$S_{2} = \{D(a^{2}), D(a^{3}), D(a^{5}), D(a^{2}b)\}$ \\ 
\hline 
$S_{3} = \{D(a), D(a^{2}), D(a^{5}), D(b)\}$ \\ 
\hline 
$S_{4} = \{D(a^{3}), D(a^{5}), D(b), D(a^{2}b)\}$ \\ 
\hline 
$S_{5} = \{D(a), D(a^{5}), D(b), D(a^{2}b)\}$ \\ 
\hline 
\end{tabular} 
\end{center}
generate a $[12,4,4]$ binary self-orthogonal non-LCD IDD code.
\end{example}

\section{Conclusion}\label{section 7}
In this article, we have studied the $(\sigma, \tau)$-derivations of $RG$ for a commutative unital ring $R$ and a group $G$. In Section \ref{section 2}, we first obtained a necessary and sufficient condition under which a map defined on a generating set of a group can be extended to a $\sigma$- and $(\sigma, \tau)$-derivation of the group ring and then characterized $\sigma$-derivations of commutative group algebras over a field of positive characteristic. In Section \ref{section 3}, we studied inner $(\sigma, \tau)$-derivations of group algebras using $(\sigma, \tau)$-conjugacy classes in a group. We characterized all inner $(\sigma, \tau)$-derivations of the group algebra $RG$ of an arbitrary group $G$ over an arbitrary commutative ring $R$ in terms of the rank and a basis of the corresponding $R$-module of all inner $(\sigma, \tau)$-derivations. Also, we saw that a $(\sigma, \tau)$-derivation of a group ring $RG$ becomes inner if the order of the group $G$ is invertible in the (unital) ring $R$. Finally, in Section \ref{section 4}, we applied the results obtained in Sections \ref{section 2} and \ref{section 3} to explicitly classify all inner and outer $\sigma$-derivations of dihedral group algebras over a field of arbitrary characteristic. Hence, we answered the $\sigma$-twisted derivation problem in dihedral group algebras. In Section \ref{section 6}, we have given the notion (definition with construction) of an IDD code and have illustrated it by implementing it in SAGEMATH and constructing codes with various parameters. \vspace{20pt}

\noindent \textbf{Acknowledgements}\vspace{8pt}

\noindent The authors thank and are deeply grateful to the referees and the editor for their critical reviews, comments, and suggestions, which have greatly improved the presentation and quality of the paper. The second author is the ConsenSys Blockchain chair professor. He thanks ConsenSys AG for that privilege.

\bibliographystyle{plain}
\end{document}